\pdfoutput=1
\documentclass[final,3p]{elsarticle} 

\usepackage{amssymb}
\usepackage{amsmath}
\usepackage{amsthm}
\usepackage{thmtools}
\usepackage{algorithm, algpseudocode}
\usepackage{float}
\usepackage{pgfplots}
\pgfplotsset{compat=1.7}
\newtheorem{definition}{Definition}
\newtheorem{lemma}[definition]{Lemma}
\newtheorem{theorem}[definition]{Theorem}
\newtheorem{corollary}[definition]{Corollary}
\newtheorem{proposition}[definition]{Proposition}
\newtheorem*{remark}{Remark}

\usepackage[english]{babel}
\usepackage{amssymb}
\usepackage{bbm}
\usepackage{hyperref}
\usepackage[nameinlink,capitalize]{cleveref}
\usepackage[shortlabels]{enumitem}

\newcommand{\midd}[1]{\mathrel{}\middle#1\mathrel{}}
\newcommand{\Prb}[1] {\mathbb{P}\left[#1\right]}
\newcommand{\bbR} {\mathbb{R}}
\newcommand{\calO} {\mathcal{O}}
\newcommand{\abs}[1] {\left|#1\right|}
\newcommand{\1} {\mathbbm{1}}
\newcommand{\hpipe}{\rotatebox[origin=c]{90}{$|$}}
\newcommand{\norm}[1] {\left\lVert #1\right\rVert}
\newcommand{\dx}[1] {\;\mathrm{d}#1}
\newcommand{\bbN} {\mathbb{N}}
\newcommand{\transp} {\scriptscriptstyle \mathsf{T}}
\newcommand{\ivar} {\mathbf{i}}
\newcommand{\iprod}[2] {\left< #1 , \; #2 \right>}
\DeclareMathOperator*{\argmax}{arg\,max}
\DeclareMathOperator*{\argmin}{arg\,min}
\newcommand{\mathhalfcell}[1] {\begin{minipage}{7.7cm}#1\vspace{0cm}\end{minipage}}
\newcommand{\inlinehalfcell}[1] {\begin{minipage}{7.7cm}\vspace{0.2cm}\centering #1\vspace{0.2cm}\end{minipage}}
\newcommand{\inlinefullcell}[1] {\begin{minipage}{15.4cm}\vspace{0.2cm}\centering #1\vspace{0.2cm}\end{minipage}}

\makeatletter
\renewcommand*\env@matrix[1][\arraystretch]{%
  \edef\arraystretch{#1}%
  \hskip -\arraycolsep
  \let\@ifnextchar\new@ifnextchar
  \array{*\c@MaxMatrixCols c}}
\makeatother

\addto\extrasenglish{

}

\definecolor{tumBlue}{RGB}{0,101,189}
\definecolor{tumDarkBlue}{RGB}{0,82,147}
\definecolor{tumOrange}{RGB}{227,114,34}
\definecolor{tumGreen}{RGB}{162,173,0}
\definecolor{commentpurple}{RGB}{102, 0, 102}
\colorlet{sectionblue}{tumBlue}
\definecolor{linkred}{RGB}{127,0,0}

\journal{Performance Evaluation}

\begin{document}

\begin{frontmatter}



\title{Formal Error Bounds for the State Space Reduction of Markov Chains}


\author[unibw]{Fabian Michel} 
\author[unibw]{Markus Siegle} 

\affiliation[unibw]{organization={Universit\"at der Bundeswehr M\"unchen, Institut f\"ur technische Informatik},
            addressline={Werner-Heisenberg-Weg 39}, 
            city={Neubiberg},
            postcode={85579}, 
            country={Germany}}

\begin{abstract}
  We study the approximation of a Markov chain on a reduced state space,
  for both discrete- and continuous-time Markov chains. In this context,
  we extend the existing theory of formal error bounds for the approximated
  transient distributions. As a special case, we consider aggregated (or lumped)
  Markov chains, where the state space reduction is achieved by partitioning the
  state space into macro states. In the discrete-time setting, we bound the stepwise
  increment of the error, and in the continuous-time setting, we bound the
  rate at which the error grows. In addition, the same error bounds can
  also be applied to bound how far an approximated stationary distribution
  is from stationarity. Subsequently, we compare these error bounds with
  relevant concepts from the literature, such as exact and ordinary lumpability,
  as well as deflatability and aggregatability. These concepts define stricter
  than necessary conditions to identify settings in which
  the aggregation error is zero. We also consider possible algorithms
  for finding suitable aggregations for which the formal error bounds are
  low, and we analyse first experiments with these algorithms on a range of different models.
\end{abstract}




\end{frontmatter}



\section{Introduction}

State aggregation in dynamic systems has been studied extensively since
the 1960s (see \cite{ncd}). Due to the curse of dimensionality, models with
large state spaces are often computationally intractable without state space
reduction. One basic reduction technique is to aggregate multiple
states into a single state in the aggregated model, and conditions under which an aggregated Markov chain is again
a Markov chain are well known (see strong and weak lumpability in \cite{finmc,fincharweaklumpctmc}).
Various cases where exact transient or stationary probabilities of
the original model can be derived from an aggregated model have been
identified and analysed (see, e.g.~\cite{exactordlump}).

However, formal error bounds for the approximation error when exact
aggregation is not possible have only been studied rarely. \cite{exactordlump}
already gave upper and lower bounds for the transient distribution of
a Markov chain which are derived from an aggregated model. Much later,
\cite{adaptformalagg} has presented improved bounds for the transient
distributions of discrete-time Markov chains, which can also be applied
to continuous-time Markov chains via uniformisation. We extend
the theory developed in \cite{adaptformalagg} to support a more general
way of disaggregation in \autoref{sec:dtmcerror} and to the continuous-time domain without falling
back on uniformisation in \autoref{sec:ctmcerror}. In doing this,
we will consider a much more abstract way of state space reduction
where no intuitive meaning for an ``aggregated state'' exists -- every original
state can be a part of every aggregated state, where the influence
on an aggregated state is given by some (positive or negative) weight.
In contrast, previous work mostly focused on partitioning the state
space into groups of states forming a macro state in the aggregated model.

Subsequently, we analyse the cases where
the error bounds on the transient distribution are zero in \autoref{sec:errzero}.
We also consider how the bounds relate to approximations of the stationary
distribution in \autoref{ssec:stat_dist}. Furthermore, we show that
the error bounds from \cite{adaptformalagg} are tight for transient
distributions in general in \autoref{sec:erroptimal}, but not for stationary
distributions. To complete the theoretical analysis, we then
compare the error bounds with lumpability concepts
from \cite{finmc,exactordlump,mcaggreactnet,probalgmcagg} in \autoref{sec:lump}. Many of
these types of lumpability imply, but are not equivalent to the error bound being zero.

In \autoref{sec:findagg}, we present two different algorithms, one from
\cite{probalgmcagg} (which we extend slightly to improve performance
and numerical stability) and one based on \cite{exactperfequiv}, with the goal to identify an
aggregation resulting in low error bounds.
We then apply these algorithms to sample models in \autoref{ssec:experiments} and analyse the results.

\section{Preliminaries}

\subsection{Linear Algebra}

Let $v \in \bbR^m$ be a vector and $A \in \bbR^{m \times n}$ a matrix.
All considered vectors are column vectors. We use $\cdot^{\transp}$ to denote
the transpose of a matrix or vector, so for the column vector $v$, the
transpose $v^{\transp}$ is a row vector.
$\abs{v}$ and $\abs{A}$ denote the vector and matrix where the absolute
value was applied component-wise. For vectors $u, v \in \bbR^m$,
we use $\iprod{u}{v} := \sum_{i=1}^m u(i) v(i) = u^{\transp} v$ to denote the usual
inner product of two vectors, where $v(i)$ is the $i$-th entry of $v$.
We further use $\mathbf{1}_m$ to denote
the column vector $(1, \ldots, 1)^{\transp} \in \bbR^m$ and $\mathbf{1}_{m \times n} \in \bbR^{m \times n}$
to denote a matrix where every entry is one. In contrast, $I_n \in \bbR^{n \times n}$
denotes the identity matrix. We set $\norm{v}_1 := \iprod{\abs{v}}{\mathbf{1}_m} = \sum_{i=1}^m \abs{v(i)}$
and $\norm{A}_\infty = \max_{i = 1, \ldots, m} \sum_{j=1}^n \abs{A(i, j)}$.
We will later use that the matrix norm $\norm{\cdot}_\infty$ is submultiplicative: for matrices $A$
and $B$, we have $\norm{AB}_\infty \leq \norm{A}_\infty \cdot \norm{B}_\infty$
(see \cite[equations (2.3.3) and (2.3.10)]{matrixcomputations}).
Note that $\norm{v^{\transp}}_\infty = \norm{v}_1$ if we interpret $v$ as a matrix.
However, to avoid using different notation for the norms of row and column
vectors, we will only use $\norm{\cdot}_1$ to denote the norm of a (row or column)
vector, i.e.~$\norm{v}_1 = \norm{v^{\transp}}_1 = \sum_{i=1}^m \abs{v(i)}$.
As $\norm{\cdot}_1$ is only used for vectors and $\norm{\cdot}_\infty$ only
for matrices in this paper, this will (hopefully) not cause any confusion.

\begin{lemma}
  \label{lem:1norm_mult}
  For a vector $v \in \bbR^m$ and a matrix $A \in \bbR^{m \times n}$, we have:
  \begin{align*}
    \norm{v^{\transp}A}_1 \leq \begin{cases}
      \iprod{\abs{v}}{\abs{A} \cdot \mathbf{1}_n} \\
      \norm{v}_1 \cdot \norm{A}_\infty
    \end{cases}
  \end{align*}
\end{lemma}

\begin{proof}
  We have
  \begin{align*}
    \norm{v^{\transp}A}_1 
    &= \sum_{j=1}^n \abs{\sum_{i=1}^m v(i) \cdot A(i, j)}
    \leq \sum_{j=1}^n \sum_{i=1}^m \abs{v(i)} \cdot \abs{A(i, j)} \\
    &= \sum_{i=1}^m \abs{v(i)} \sum_{j=1}^n \abs{A(i, j)}
    \begin{cases}
      = \iprod{\abs{v}}{\abs{A} \cdot \mathbf{1}_n} \\
      \leq \left(\sum_{i=1}^m \abs{v(i)}\right)\left(\max_{i=1,\ldots,m} \sum_{j=1}^n \abs{A(i, j)}\right)
      = \norm{v}_1 \cdot \norm{A}_\infty
    \end{cases} \qedhere
  \end{align*}
\end{proof}

We will use the term stochastic matrix to denote matrices $A \in \bbR^{m \times n}$
(where it is not necessarily the case that $m = n$) whose entries are all
non-negative and where every row sums to $1$ (i.e.~$A \cdot \mathbf{1}_n = \mathbf{1}_m$).

\begin{lemma}
  \label{lem:matexp_norm_bound}
  Let $A \in \bbR^{n \times n}$ and $t \in \bbR$. Then
  \begin{align*}
    \norm{e^{At}}_\infty \leq e^{t \norm{A}_\infty}
  \end{align*}
\end{lemma}

\begin{proof}
  By submultiplicativity of $\norm{\cdot}_\infty$ for matrices, we have
  \begin{align*}
    \norm{e^{At}}_\infty
    &= \norm{\sum_{k \geq 0} \frac{t^k A^k}{k!}}_\infty
    \leq \sum_{k \geq 0} \norm{\frac{t^k A^k}{k!}}_\infty
    = \sum_{k \geq 0} \frac{t^k}{k!} \norm{A^k}_\infty
    \leq \sum_{k \geq 0} \frac{t^k}{k!} \norm{A}_\infty^k
    = e^{t \norm{A}_\infty} \qedhere
  \end{align*}
\end{proof}

\subsection{Markov chains and state space reduction}
\label{ssec:prelim_mc}

We consider time-homogeneous discrete- and continuous-time
Markov chains (DTMCs and CTMCs) on the finite state space $S = \{1, \ldots, n\}$.
The dynamics are given by the stochastic transition matrix $P \in \bbR^{n \times n}$ for DTMCs,
where we have $P(i, j) = \Prb{X_{k+1} = j \midd| X_k = i}$ if $X_k$ denotes the
state of the DTMC at time $k$. For CTMCs, the dynamics are defined via the
generator matrix $Q \in \bbR^{n \times n}$, where $Q(i, j)$ is the transition rate from $i$ to $j$,
and $Q(i, i) = - \sum_{j \neq i} Q(i, j)$.
Given an initial distribution $p_0 \in \bbR^n$, the transient distribution
of a DTMC (respectively CTMC) is given by $p_k^{\transp} = p_0^{\transp} P^k$ (respectively $p_t^{\transp} = p_0^{\transp} e^{Qt}$).

We want to reduce the state space of the Markov chain to speed up computation
of various properties. We often refer to state space reduction as aggregation,
even though we take a very abstract view of aggregation: there are not necessarily groups of
states which are aggregated into one macro state. Instead, we define the
aggregation of a Markov chain with an aggregated state space of dimension $m$ (where $m \leq n$) as follows: given a disaggregation matrix
$A \in \bbR^{m \times n}$ (which can be arbitrary), an (arbitrary) aggregated step
matrix $\Pi \in \bbR^{m \times m}$ for DTMCs, an (arbitrary) aggregated
evolution matrix $\Theta \in \bbR^{m \times m}$ for CTMCs, and an (arbitrary) aggregated initial vector $\pi_0 \in \bbR^m$,
we approximate the dynamics of the original chain by setting
$\widetilde{p}_k^{\transp} := \pi_k^{\transp} A := \pi_0^{\transp} \Pi^k A$ and
$\widetilde{p}_t^{\transp} := \pi_t^{\transp} A := \pi_0^{\transp} e^{\Theta t} A$.
$\widetilde{p}_k$ and $\widetilde{p}_t$ are intended
to approximate the transient distributions of the original Markov chains, i.e.~$p_k$ and $p_t$.

In this abstract view of aggregation, we do not require that $\Pi$ is stochastic, that
$\Theta$ is a generator or that $\pi_0$ is a probability distribution. However,
intuitively, $\Pi$ should be an approximation of the step dynamics of the
DTMC in a lower-dimensional state space, and it often makes sense to choose
a stochastic matrix $\Pi$, which would imply that we approximate the
original chain with a DTMC with fewer states. The more abstract view
with arbitrary matrices $\Pi$ allows us, in theory, to approximate
the step dynamics in the lower-dimensional state space with any linear map.
We opt for the greater generality since most of the results presented in this
paper already hold in this very abstract setting.

In practical applications, a typical choice would be to choose a
stochastic matrix $\Pi$, a generator matrix $\Theta$, a probability distribution
$\pi_0$, and a stochastic matrix $A$ (the latter guarantees that
$\widetilde{p}_k$ is a probability distribution as well). We call $A$ disaggregation
matrix since $A$ describes how to blow up the aggregated transient
distribution $\pi_k$ to the full-state-space approximation $\widetilde{p}_k$
via the equation $\widetilde{p}_k^{\transp} = \pi_k^{\transp} A$, which
corresponds to disaggregating $\pi_k$.

The most commonly studied type of aggregation is even more restrictive in
the possible choices for $\Pi$, $\Theta$, $A$ and $\pi_0$. Usually, the
state space $S$ of the original chain is partitioned into aggregates
by some partition $\Omega = \{\Omega_1, \ldots, \Omega_m\}$ of $S$, with $\sigma \in \Omega$
being a subset of $S$ which represents all states belonging to one aggregate. The
aggregation function $\omega : S \to \Omega$ maps a state $s$ to
the aggregate to which $s$ belongs, i.e.~$s \in \omega(s)$. Instead
of an arbitrary disaggregation matrix $A$, we define probability distributions $\alpha_\sigma \in \bbR^n$ with
support on $\sigma \in \Omega$. As a shorthand, we write $\alpha(s) := \alpha_{\omega(s)}(s)$.
The value $\alpha(s)$ should approximate the conditional probability of being in state $s$
when we know that we are in aggregate $\omega(s)$, i.e.~the probability
$\Prb{X_k = s \midd| X_k \in \omega(s)}$. This probability is in general
dependent on time, but we only consider time-independent approximations $\alpha$.
We then define the disaggregation matrix $A$ and the aggregation matrix $\Lambda$ as follows:
\begin{align*}
  \Lambda = \begin{pmatrix}
    | &  & | \\
    \1_{\Omega_1} & \hdots & \1_{\Omega_m} \\
    | &  & |
  \end{pmatrix} \in \bbR^{n \times m}, \;\;
  A = \begin{pmatrix}
    \hpipe \, \alpha_{\Omega_1}^{\transp} \, \hpipe \\
    \vdots \\
    \hpipe \, \alpha_{\Omega_m}^{\transp} \, \hpipe
  \end{pmatrix} \in \bbR^{m \times n}
  \qquad \textrm{(note: } A\Lambda = I \textrm{)}
\end{align*}
where $\1_{\sigma} \in \bbR^n$ is defined by
\begin{align*}
  \1_{\sigma}(s) = \begin{cases}
    1 & \textrm{ if } s \in \sigma \\
    0 & \textrm{ otherwise}
  \end{cases}
\end{align*}
A natural definition for $\Pi$ and $\Theta$ is then given by
$\Pi = A P \Lambda$ and $\Theta = A Q \Lambda$, which will ensure
that $\Pi$ is stochastic and that $\Theta$ is a generator.
In this case, $\Pi(\rho, \sigma)$ for $\rho,\sigma \in \Omega$ is
an approximation of the probability to transition from one aggregate state
into another, that is, an approximation of $\Prb{X_{k+1} \in \sigma \midd| X_k \in \rho}$.
Note that this probability may also depend on time (i.e.~on $k$) in general, in contrast to the
probability $\Prb{X_{k+1} = s \midd| X_k = r}$ for $r,s \in S$. However, we again consider only
time-independent approximations of $\Prb{X_{k+1} \in \sigma \midd| X_k \in \rho}$.
Simlarly, for CTMCs, we should have
\begin{align*}
  \Theta(\rho, \sigma) \approx \lim_{u \to 0} \frac{\Prb{X_{t+u} \in \sigma \midd| X_t \in \rho}}{u}
\end{align*}
if we aim at a faithful approximation of the dynamics.
Furthermore,
$\pi_0^{\transp} = p_0^{\transp} \Lambda$ is the natural choice for
the initial distribution when working with actual aggregates.

\begin{definition}
  \label{def:compat}
  Given a partition $\Omega$ of the state space of a DTMC or CTMC,
  we call a probability distribution $p$ on the state space $S$ \textbf{compatible}
  with distributions $\alpha_\sigma$ with support on $\sigma \in \Omega$
  if $p^{\transp} \Lambda A  = p^{\transp}$.
\end{definition}

Compatibility of $p$ and the distributions $\alpha$ means that
\begin{align*}
  \alpha(s) = \frac{p(s)}{\sum_{s' \in \omega(s)} p(s')} \textrm{ for all } s \in S \textrm{ s.t. } \sum_{s' \in \omega(s)} p(s') > 0
\end{align*}

To make it easier to identify this special case of state space reduction via the partition $\Omega$,
we will always use $r,s$ to denote the states of the original chain and
$\rho,\sigma$ to denote the aggregates in this context, and we will call
this type of aggregation \textbf{weighted state space partitioning}, while we will use
$i,j$ to denote the states of the original chain and $\chi, \phi$ to denote the abstract states
of the aggregated chain in the more abstract context where we allow arbitrary
$\Pi$, $\Theta$, $A$ and $\pi_0$.

\section{Bounding the approximation error}

\subsection{Error bounds for DTMCs}
\label{sec:dtmcerror}

In this section, we derive formal error
bounds for the difference between the actual transient distribution $p_k$
of the Markov chain and the approximation $\widetilde{p}_k$ of this distribution,
obtained after state space reduction. We loosely follow \cite{adaptformalagg} in
our derivation, but a few adaptions are necessary as we consider a more general
setting for aggregation. \cite[Theorem 11]{exactordlump} also gave
error bounds similar to \cite{adaptformalagg}, but the bounds in \cite{adaptformalagg} are slightly
better. For now, we consider the disaggregation matrix $A$,
the aggregated step matrix $\Pi$, and the aggregated initial vector $\pi_0$ as arbitrary but fixed.

\begin{theorem}
  \label{thm:dtmc_bound}
  Let $e_k$ be the error vector after $k$ steps, i.e.
  \begin{align*}
    e_k^{\transp} = \widetilde{p}_k^{\transp} - p_k^{\transp} = \pi_0^{\transp} \Pi^k A - p_0^{\transp} P^k
  \end{align*}
  Then, the following hold:
  \begin{enumerate}[(i)]
    \item \label{thm:dtmc_bound_precise} $\displaystyle \norm{e_k}_1 \leq \norm{\pi_0^{\transp} A - p_0^{\transp}}_1 + \sum_{j=0}^{k-1} \iprod{\abs{\pi_j}}{\abs{\Pi A - A P} \cdot \mathbf{1}_n}$
    \item \label{thm:dtmc_bound_general} $\displaystyle \norm{e_k}_1 \leq \norm{\pi_0^{\transp} A - p_0^{\transp}}_1 + \norm{\pi_0}_1 \cdot \norm{\Pi A - A P}_\infty \cdot \begin{cases}
      \frac{\norm{\Pi}^k_\infty - 1}{\norm{\Pi}_\infty - 1} & \textrm{ if } \norm{\Pi}_\infty \neq 1 \\
      k & \textrm{ otherwise}
    \end{cases}$
    \item \label{thm:dtmc_bound_imprecise} if $\Pi$ is stochastic and $\pi_0$ is a probability distribution, then
    $\displaystyle \norm{e_k}_1 \leq \norm{\pi_0^{\transp} A - p_0^{\transp}}_1 + k \cdot \norm{\Pi A - A P}_\infty$
  \end{enumerate}
\end{theorem}

\begin{proof}
  We want to give a bound on $\norm{e_k}_1$, where $e_k = \widetilde{p}_k - p_k$
  is the error vector after $k$ steps. We have
  \begin{align}
    \label{eq:dtmc_step_bnd}
    \begin{split}
      \norm{e_k}_1
      &= \norm{\pi_0^{\transp} \Pi^k A - p_0^{\transp} P^k}_1
      = \left\lVert \vphantom{\left(\Pi^{k-1} A - p_0^{\transp} P^{k-1}\right)} \right.
      \underbrace{\left(\pi_0^{\transp} \Pi^{k-1} A - p_0^{\transp} P^{k-1}\right)}_{e_{k-1}^{\transp}} \, \cdot \, P + \underbrace{\pi_0^{\transp} \Pi^{k-1}}_{\pi_{k-1}^{\transp}} \left(\Pi A - A P\right)
      \left. \vphantom{\left(\Pi^{k-1} A - p_0^{\transp} P^{k-1}\right)} \right\rVert_1 \\
      &\leq \norm{e_{k-1}^{\transp} \cdot P}_1 + \norm{\pi_{k-1}^{\transp} \cdot \left(\Pi A - A P\right)}_1
      \overset{\textrm{\autoref{lem:1norm_mult}}}{\leq} \norm{e_{k-1}}_1 \cdot \, \underbrace{\norm{P}_\infty}_{=1} + \norm{\pi_{k-1}^{\transp} \cdot \left(\Pi A - A P\right)}_1 \\
      &\overset{\textrm{\autoref{lem:1norm_mult}}}{\leq}
      \begin{cases}
        \norm{e_{k-1}}_1 + \iprod{\abs{\pi_{k-1}}}{\abs{\Pi A - A P} \cdot \mathbf{1}_n} \\
        \norm{e_{k-1}}_1 + \norm{\pi_{k-1}}_1 \cdot \norm{\Pi A - A P}_\infty
      \end{cases}
    \end{split}
  \end{align}
  We can continue with the first bound as follows:
  \begin{align*}
    \norm{e_k}_1 &\leq \norm{e_0}_1 + \sum_{j=0}^{k-1} \iprod{\abs{\pi_j}}{\abs{\Pi A - A P} \cdot \mathbf{1}_n}
  \end{align*}
  which proves \ref{thm:dtmc_bound_precise}.
  Using that
  \begin{align*}
    \norm{\pi_{k-1}}_1 = \norm{\pi_0^{\transp} \cdot \Pi^{k-1}}_1 \;\overset{\substack{\textrm{\autoref{lem:1norm_mult} \&}\\\textrm{submultiplicativity}}}{\leq}\; \norm{\pi_0}_1 \cdot \norm{\Pi}_\infty^{k-1}
  \end{align*}
  we can also continue with the second bound in \eqref{eq:dtmc_step_bnd}, obtaining
  \begin{align*}
    \norm{e_k}_1 &\leq \norm{e_0}_1 + \sum_{j=0}^{k-1} \norm{\pi_0}_1 \cdot \norm{\Pi}^j_\infty \cdot \norm{\Pi A - A P}_\infty \\
    &= \norm{\pi_0^{\transp} A - p_0^{\transp}}_1 + \norm{\pi_0}_1 \cdot \norm{\Pi A - A P}_\infty \cdot \sum_{j=0}^{k-1} \norm{\Pi}^j_\infty \\
    &= \norm{\pi_0^{\transp} A - p_0^{\transp}}_1 + \norm{\pi_0}_1 \cdot \norm{\Pi A - A P}_\infty \cdot \begin{cases}
      \frac{\norm{\Pi}^k_\infty - 1}{\norm{\Pi}_\infty - 1} & \textrm{ if } \norm{\Pi}_\infty \neq 1 \\
      k & \textrm{ otherwise}
    \end{cases}
  \end{align*}
  which proves \ref{thm:dtmc_bound_general} and \ref{thm:dtmc_bound_imprecise}.
\end{proof}

Some remarks are in order: in \cite{adaptformalagg}, the bounds look quite
different on first sight (compare with \cite[equations (17) and (18)]{adaptformalagg}).
However, they are actually a special case of the above bounds. In fact, we can define $\tau(\chi)$ for
$\chi \in \{1, \ldots, m\}$ in analogy to \cite{adaptformalagg} as follows:
$\tau(\chi) := \left(\abs{\Pi A - A P} \cdot \mathbf{1}_n\right)(\chi)$
(note that $\abs{\Pi A - A P} \cdot \mathbf{1}_n$ is a column vector with $m$ entries).
We can interpret $\tau(\chi)$ as the error caused by ``aggregate'' $\chi$,
even though the word ``aggregate'' does not have an intuitive meaning
in our very general setting for aggregations (there are no such things
as aggregates consisting of a certain group of states in our case).
In \cite{adaptformalagg}, the considered aggregations correspond almost exactly
to what was defined as weighted state space partitioning in \autoref{ssec:prelim_mc},
with the following exceptions: the $\alpha$ distributions were assumed to
be uniform distributions over the respective aggregates, and it was \emph{not}
assumed that $\Pi = A P \Lambda$. In this context, $\tau(\chi)$, written
as $\tau(\rho)$ for $\rho \in \Omega$ does indeed have the intuitive meaning
of bounding the error caused by the aggregate $\rho$ which corresponds
to a group of states in the original chain.

The advantage of the more general bounds in \autoref{thm:dtmc_bound} is that
we can calculate error bounds for a wider set of possible aggregation schemes
when compared to \cite{adaptformalagg}, as we are able to choose arbitrary values
for $\pi_0$, $\Pi$ and $A$. In the following, we often distinguish two types of error:
\begin{align*}
  \norm{e_k}_1 \leq \underbrace{\norm{\pi_0^{\transp} A - p_0^{\transp}}_1}_{\textrm{initial error}} + \underbrace{\sum_{j=0}^{k-1} \iprod{\abs{\pi_j}}{\abs{\Pi A - A P} \cdot \mathbf{1}_n}}_{\textrm{dynamic error}}
\end{align*}
The initial error is independent of the transition matrix $P$ and describes
the error made by approximating $p_0$ with an aggregated initial distribution $\pi_0$,
while the dynamic error describes the error made by approximating the dynamic
evolution of the Markov chain, given by the transition matrix $P$, with
an aggregated dynamics, given by the transition matrix $\Pi$.

\subsection{Error bounds for CTMCs}
\label{sec:ctmcerror}

We again want to give a bound on $\norm{e_t}_1$, where we set $e_t = \widetilde{p}_t - p_t$,
as before. We have
\begin{align*}
  \norm{e_t}_1
  &= \norm{\pi_0^{\transp} e^{\Theta t} A - p_0^{\transp} e^{Qt}}_1
\end{align*}
In \cite{adaptformalagg}, error bounds for aggregated CTMCs were obtained
by uniformising the Markov chain and then applying the error bounds for DTMCs.
We are now going to drop the detour via uniformisation and show how to directly bound the approximation error of the transient
distribution in continuous time with a very similar expression as in
the discrete-time case via the error matrix $\Theta A - A Q$ (instead of $\Pi A - A P$ for
DTMCs). The norm $\norm{\Theta A - A Q}_\infty$
can now be interpreted as a rate of maximal error growth (instead of error growth
per step). We again consider $A$, $\Theta$ and $\pi_0$ as arbitrary but fixed.

\begin{theorem}
  \label{thm:ctmc_bound}
  Let $e_t$ be the error vector at time $t$, i.e.
  \begin{align*}
    e_t^{\transp} = \widetilde{p}_t^{\transp} - p_t^{\transp} = \pi_0^{\transp} e^{\Theta t} A - p_0^{\transp} e^{Qt}
  \end{align*}
  Then, we have that
  \begin{align*}
    &\norm{e_t}_1 \textrm{ is absolutely continuous and almost everywhere (a.e.) differentiable, with} \\
    \frac{\mathrm{d}}{\mathrm{d}t}& \norm{e_t}_1 \leq \left< \abs{\pi_t}, \;\; \abs{\Theta A - A Q} \cdot \mathbf{1}_n \right> \textrm{ almost everywhere, and} \\
    \limsup_{u \to 0} \; &\frac{\norm{e_{t+u}}_1 - \norm{e_t}_1}{u} \leq \left< \abs{\pi_t}, \;\; \abs{\Theta A - A Q} \cdot \mathbf{1}_n \right> \textrm{ everywhere}
  \end{align*}
  and the following hold:
  \begin{enumerate}[(i)]
    \item \label{thm:ctmc_bound_precise} $\displaystyle \norm{e_t}_1 \leq \norm{\pi_0^{\transp} A - p_0^{\transp}}_1 + \int_0^t \left< \abs{\pi_u}, \;\; \abs{\Theta A - A Q} \cdot \mathbf{1}_n \right> \dx{u}$
    \item \label{thm:ctmc_bound_general}$\displaystyle \norm{e_t}_1 \leq \norm{\pi_0^{\transp} A - p_0^{\transp}}_1 + \norm{\pi_0}_1 \cdot \norm{\Theta A - A Q}_\infty \cdot \frac{e^{t\norm{\Theta}_\infty} - 1}{\norm{\Theta}_\infty}$
    whenever $\norm{\Theta}_\infty \neq 0$
    \item \label{thm:ctmc_bound_imprecise} if $\Theta$ is a generator and $\pi_0$ is a probability distribution, then
    $\displaystyle \norm{e_t}_1 \leq \norm{\pi_0^{\transp} A - p_0^{\transp}}_1 + t \cdot \norm{\Theta A - A Q}_\infty$
  \end{enumerate}
\end{theorem}

Before being able to prove the above theorem, we need another lemma.
\begin{lemma}
  \label{lem:absfctlim}
  Assume that $f: \bbR \to \bbR$ is differentiable in $0$, and that $f(0) = 0$.
  Then
  \begin{align*}
    \limsup_{u \to 0} \frac{\abs{f(u)}}{u} = \lim_{\substack{u \to 0\\u > 0}} \frac{\abs{f(u)}}{u} &= \abs{f'(0)} \qedhere
  \end{align*}
\end{lemma}

\begin{proof}
  We have:
  \begin{align*}
    \abs{f'(0)} &= \abs{\lim_{u \to 0} \frac{f(u)}{u}} = \lim_{u \to 0} \frac{\abs{f(u)}}{\abs{u}}
    = \lim_{\substack{u \to 0\\u > 0}} \frac{\abs{f(u)}}{u}
    = \limsup_{u \to 0} \frac{\abs{f(u)}}{u}
    \qedhere
  \end{align*}
\end{proof}

\begin{proof}[Proof of \autoref{thm:ctmc_bound}]
  First, note the following: every component of $e_t$ is continuously differentiable
  in $t$, as both $p_t$ and $\widetilde{p}_t$ are continuously differentiable with respect to
  $t$. Indeed, calculating the derivative of all components of $e_t$ simultaneously, we get
  \begin{align}
    \frac{\mathrm{d}}{\mathrm{d}t} (\widetilde{p}_t^{\transp} - p_t^{\transp}) &= \frac{\mathrm{d}}{\mathrm{d}t} \left(\pi_0^{\transp} e^{\Theta t} A - p_0^{\transp} e^{Qt}\right)
    = \pi_0^{\transp} e^{\Theta t}\Theta A - p_0^{\transp} e^{Qt}Q = \pi_t^{\transp} \Theta A - p_t^{\transp} Q \label{eq:error_derivative} \\
    \norm{\frac{\mathrm{d}}{\mathrm{d}t} (\widetilde{p}_t^{\transp} - p_t^{\transp})}_1 &= \norm{\pi_t^{\transp} \Theta A - p_t^{\transp} Q}_1
    \leq \norm{\pi_t^{\transp} \Theta A}_1 + \norm{p_t^{\transp} Q}_1 
    \overset{\textrm{\autoref{lem:1norm_mult}}}{\leq} \norm{\pi_0}_1 \norm{e^{\Theta t}}_\infty \norm{\Theta A}_\infty + \norm{p_t}_1 \norm{Q}_\infty \notag \\
    &\overset{\textrm{\autoref{lem:matexp_norm_bound}}}{\leq} \norm{\pi_0}_1 e^{t \norm{\Theta}_\infty} \norm{\Theta A}_\infty + \norm{Q}_\infty \notag
  \end{align}
  As every component of $e_t$ is continuously differentiable with bounded derivative, $\norm{e_t}_1$ is absolutely continuous
  and differentiable almost everywhere (see, e.g., \cite[Section 5.4 on page 108]{realanalysis}) on any bounded time interval.

  Noting that $e^{Qt} e^{Qu} = e^{Q(t+u)}$ by commutativity of $Qt$ and $Qu$, we see that
  \begin{align}
    \begin{split}
      \norm{e_{t+u}}_1
      &= \norm{\pi_0^{\transp} e^{\Theta (t+u)} A - p_0^{\transp} e^{Q(t+u)}}_1 \\
      &= \left\lVert\vphantom{\pi_0^{\transp} e^{\Theta t}\left(e^{\Theta u}A - Ae^{Qu}\right)}\right.
      \underbrace{\left(\pi_0^{\transp} e^{\Theta t} A - p_0^{\transp} e^{Qt}\right)}_{e_t^{\transp}} \underbrace{e^{Qu}}_{\textrm{stochastic}} + \left.\pi_0^{\transp} e^{\Theta t}\left(e^{\Theta u}A - Ae^{Qu}\right)\right\rVert_1 \\
      &\overset{\substack{\triangle\textrm{-inequ.~\&}\\\textrm{\autoref{lem:1norm_mult}}}}{\leq} \norm{e_t}_1 + \norm{\pi_t^{\transp} \left(e^{\Theta u}A - Ae^{Qu}\right)}_1
      \overset{\textrm{\autoref{lem:1norm_mult}}}{\leq} \begin{cases}
        \norm{e_t}_1 + \left< \abs{\pi_t}, \;\; \abs{e^{\Theta u}A - Ae^{Qu}} \cdot \mathbf{1}_n \right> \\
        \norm{e_t}_1 + \norm{\pi_t}_1 \cdot \norm{e^{\Theta u}A - Ae^{Qu}}_\infty
      \end{cases}
    \end{split}
    \label{eq:ctmc_altbound_intermed}
  \end{align}
  To obtain a more precise bound, we will continue by using the first inequality derived
  above. Note: we have that (the matrices $e^{Qu}$ and $Q$ commute)
  \begin{align*}
    \frac{\mathrm{d}}{\mathrm{d}u} \left(e^{\Theta u} A - A e^{Qu}\right) &= e^{\Theta u} \Theta A - A e^{Qu} Q \; \overset{u=0}{=} \; \Theta A - A Q 
  \end{align*}
  Hence:
  \begin{align*}
    &\norm{e_{t+u}}_1 - \norm{e_t}_1
    \leq \left< \abs{\pi_t}, \;\; \abs{e^{\Theta u}A - Ae^{Qu}} \cdot \mathbf{1}_n \right> \\[1em]
    \implies
    \limsup_{u \to 0} \; &\frac{\norm{e_{t+u}}_1 - \norm{e_t}_1}{u} \leq \\
    &\leq \limsup_{u \to 0} \frac{\left< \abs{\pi_t}, \;\; \abs{e^{\Theta u}A - Ae^{Qu}} \cdot \mathbf{1}_n \right>}{u}
    = \limsup_{u \to 0} \sum_{\chi = 1}^m \abs{\pi_t(\chi)} \cdot \sum_{i=1}^n \frac{\abs{\left(e^{\Theta u}A - Ae^{Qu}\right) (\chi, i)}}{u} \\
    &\leq \sum_{\chi = 1}^m \abs{\pi_t(\chi)} \cdot \sum_{i=1}^n \limsup_{u \to 0} \frac{\abs{\left(e^{\Theta u}A - Ae^{Qu}\right) (\chi, i)}}{u}
    \overset{\textrm{\autoref{lem:absfctlim}}}{=} \sum_{\chi = 1}^m \abs{\pi_t(\chi)} \cdot \sum_{i=1}^n \abs{\left(\Theta A - A Q\right)(\chi, i)} \\
    &= \left< \abs{\pi_t}, \;\; \abs{\Theta A - A Q} \cdot \mathbf{1}_n \right>
  \end{align*}
  This proves the bounds on the derivative given in \autoref{thm:ctmc_bound}
  and \ref{thm:ctmc_bound_precise} follows immediately.
  Going back to \eqref{eq:ctmc_altbound_intermed}, we can also derive a
  less precise bound using the lower inequality in \eqref{eq:ctmc_altbound_intermed}:
  we have that
  \begin{align*}
    \limsup_{u \to 0} \; &\frac{\norm{e_{t+u}}_1 - \norm{e_t}_1}{u} \leq \\
    &\leq \limsup_{u \to 0} \norm{\pi_t}_1 \cdot \frac{\norm{e^{\Theta u}A - Ae^{Qu}}_\infty}{u}
    = \limsup_{u \to 0} \norm{\pi_t}_1 \cdot \max_{\chi = 1, \ldots, m} \sum_{i = 1}^{n} \frac{\abs{\left(e^{\Theta u}A - Ae^{Qu}\right)(\chi, i)}}{u} \\
    &\leq \norm{\pi_t}_1 \cdot \max_{\chi = 1, \ldots, m} \sum_{i = 1}^{n} \limsup_{u \to 0} \frac{\abs{\left(e^{\Theta u}A - Ae^{Qu}\right)(\chi, i)}}{u} \\
    &\overset{\textrm{\autoref{lem:absfctlim}}}{=} \norm{\pi_t}_1 \cdot \max_{\chi = 1, \ldots, m} \sum_{i = 1}^{n} \abs{\left(\Theta A - A Q\right)(\chi, i)}
    = \norm{\pi_t}_1 \cdot \norm{\Theta A - A Q}_\infty
  \end{align*}
  Therefore,
  \begin{align*}
    \norm{e_t}_1 &\leq \norm{\pi_0^{\transp} A - p_0^{\transp}}_1 + \int_0^t \norm{\pi_u}_1 \cdot \norm{\Theta A - A Q}_\infty \dx{u} \\
    &\leq \norm{\pi_0^{\transp} A - p_0^{\transp}}_1 + \norm{\pi_0}_1 \cdot \left(\int_0^t \norm{e^{\Theta u}}_\infty \dx{u}\right) \cdot \norm{\Theta A - A Q}_\infty
  \end{align*}
  where we used that $\norm{\pi_t}_1 \overset{\textrm{\autoref{lem:1norm_mult}}}{\leq} \norm{\pi_0}_1 \cdot \norm{e^{\Theta t}}_\infty$.
  If $\Theta$ is a generator and if $\pi_0$ is a probability distribution, this
  inequality collapses to
  \begin{align*}
    \norm{e_t}_1 &\leq \norm{\pi_0^{\transp} A - p_0^{\transp}}_1 + t \cdot \norm{\Theta A - A Q}_\infty
  \end{align*}
  which proves \ref{thm:ctmc_bound_imprecise}. To prove \ref{thm:ctmc_bound_general}, we note that
  \begin{align*}
    \int_0^t \norm{e^{\Theta u}}_\infty \dx{u}
    &\overset{\textrm{\autoref{lem:matexp_norm_bound}}}{\leq} \int_0^t e^{u\norm{\Theta}_\infty} \dx{u}
    = \left[\frac{e^{u\norm{\Theta}_\infty}}{\norm{\Theta}_\infty}\right]_{u=0}^t
    = \frac{e^{t\norm{\Theta}_\infty} - 1}{\norm{\Theta}_\infty} \qedhere
  \end{align*}
\end{proof}

\subsection{When is the error bound zero?}
\label{sec:errzero}

\begin{corollary}
  \label{cor:exact_errbnd}
  Given a DTMC or CTMC, an arbitrary disaggregation matrix $A$, an arbitrary aggregated step matrix $\Pi$
  (respectively evolution matrix $\Theta$), and an arbitrary initial vector $\pi_0$, we have the following.
  If $\Pi A = A P$ (respectively $\Theta A = A Q$), then
  \begin{align*}
    \norm{\widetilde{p}_k - p_k}_1 \leq \norm{\pi_0^{\transp} A - p_0^{\transp}}_1
    \textrm{ or, for continuous time, }
    \norm{\widetilde{p}_t - p_t}_1 \leq \norm{\pi_0^{\transp} A - p_0^{\transp}}_1
  \end{align*}
\end{corollary}

\begin{proof}
  This follows directly from \autoref{thm:dtmc_bound} and \autoref{thm:ctmc_bound}.
\end{proof}

\autoref{cor:exact_errbnd} motivates the following definition:

\begin{definition}
  \label{def:exactagg}
  We call an aggregation of a DTMC (respectively CTMC), given by
  a disaggregation matrix $A$, an aggregated step matrix $\Pi$ (respectively
  an aggregated evolution matrix $\Theta$) and an aggregated initial vector
  $\pi_0$, \textbf{dynamic-exact}
  if $\Pi A = A P$ (respectively $\Theta A = A Q$).
  
  If we further have that $\pi_0^{\transp} A = p_0^{\transp}$,
  then we call the aggregation \textbf{exact}.
\end{definition}

If $A$, $\Pi$ and $\pi_0$ are an exact aggregation, then
$\widetilde{p}_k = p_k$ for all $k$. This was already shown in \autoref{cor:exact_errbnd},
but can also be proven easily via induction. Indeed,
$\Pi A = A P$ implies $A P^k = APP^{k-1} = \Pi A P P^{k-2} = \ldots = \Pi^k A$.
Hence, if $\pi_0^{\transp} A = p_0^{\transp}$, then we have
\begin{align}
  \label{eq:induc_exact}
  p_k^{\transp} &= p_0^{\transp}P^k = \pi_0^{\transp} A P^k = \pi_0^{\transp} \Pi^k A = \pi_k^{\transp} A = \widetilde{p}_k^{\transp}
\end{align}
If the aggregation is only dynamic-exact, then it only holds that
$\widetilde{p}_k^{\transp} = \widetilde{p}_0^{\transp} P^k$ as
\begin{align*}
  \widetilde{p}_0^{\transp} P^k &= \pi_0^{\transp} A P^k = \pi_0^{\transp} \Pi^k A = \pi_k^{\transp} A = \widetilde{p}_k^{\transp}
\end{align*}
Note that in this case, we might have $\widetilde{p}_0 \neq p_0$. This is the reason for the term ``dynamic-exact'': the
step dynamics are correctly represented by the aggregation (for certain initial vectors), but the initial distribution might
be wrongly approximated by such an initial vector. The corresponding statements also hold for CTMCs.

\begin{proposition}
  \label{prop:dynex_alwaysiniterr}
  Consider an aggregation of a DTMC (respectively CTMC) which is dynamic-exact,
  i.e. $\Pi A = A P$ (respectively $\Theta A = A Q$). Further assume that
  either $\pi_0^{\transp} A \geq p_0^{\transp}$ or $\pi_0^{\transp} A \leq p_0^{\transp}$,
  where the inequalities should hold entry-wise. Then:
  \begin{align*}
    \norm{\pi_0^{\transp} \Pi^k A - p_0^{\transp} P^k}_1 \overset{\textrm{def.}}{=} \norm{e_k}_1 = \norm{\pi_0^{\transp} A - p_0^{\transp}}_1
    \quad \textrm{ respectively } \quad
    \norm{\pi_0^{\transp} e^{\Theta t} A - p_0^{\transp} e^{Qt}}_1 \overset{\textrm{def.}}{=} \norm{e_t}_1 = \norm{\pi_0^{\transp} A - p_0^{\transp}}_1
  \end{align*}
  Furthermore, if $\pi_0^{\transp} A \geq p_0^{\transp}$, then $\widetilde{p}_k^{\transp} = \pi_0^{\transp} \Pi^k A \geq p_k^{\transp}$ for all $k$
  (and the same holds for $\leq$ and for CTMCs).
\end{proposition}

\begin{proof}
  Note that
  \begin{align*}
    \pi_0^{\transp} \Pi^k A - p_0^{\transp} P^k
    = \underbrace{\pi_0^{\transp} \Pi^k A - \pi_0^{\transp} A P^k}_{=0 \textrm{ as } \Pi A = A P} \, + \, \pi_0^{\transp} A P^k - p_0^{\transp} P^k
    = \left(\pi_0^{\transp} A - p_0^{\transp}\right) P^k
  \end{align*}
  which proves the claim for DTMCs as $P^k$ is stochastic and therefore preserves
  the $\norm{\cdot}_1$-norms as well as the non-negativity or non-positivity of vectors applied from the left if they have all non-negative or all non-positive
  entries. For CTMCs, we have
  \begin{align*}
    \pi_0^{\transp} e^{\Theta t} A - p_0^{\transp} e^{Qt}
    &= \pi_0^{\transp} e^{\Theta t} A - \pi_0^{\transp} A e^{Qt} + \pi_0^{\transp} A e^{Qt} - p_0^{\transp} e^{Qt} \\
    &= \underbrace{\pi_0^{\transp} \sum_{k=0}^\infty \frac{t^k\left(\Theta^k A - A Q^k\right)}{k!}}_{=0 \textrm{ as } \Theta A = A Q} \, + \,
    \left(\pi_0^{\transp} A - p_0^{\transp}\right) e^{Qt}
  \end{align*}
  which proves the claim for CTMCs as $e^{Qt}$ is stochastic as well.
\end{proof}

The above \autoref{prop:dynex_alwaysiniterr} allows us to find lower and upper bounds for the transient
distributions by choosing appropriate values for $\pi_0$ which
result in lower or upper bounds of $p_0$ (in the sense that $\pi_0^{\transp} A \geq p_0^{\transp}$
or $\pi_0^{\transp} A \leq p_0^{\transp}$) if we have a dynamic-exact aggregation.

The condition $\Pi A = A P$ has appeared in the literature before in connection
with Markov chains, but always in more restricted contexts, in particular in the context
of weighted state space partitioning as defined in \autoref{ssec:prelim_mc}
(where the state space $S$ of the original chain is partitioned into aggregates
by $\Omega$ and each state $s$ of the original chain is assigned weight $\alpha(s)$
within its aggregate). Indeed, equation (4) on page 135 of \cite{finmc} states that, if $\Pi = A P \Lambda$,
then $\Pi A = A P$ implies weak lumpability
of the DTMC. A DTMC is called weakly lumpable for a given partition $\Omega$
if there exists an initial distribution $p_0$ such that the process
$Y_k$, defined by $Y_k = \sigma \in \Omega \iff X_k \in \sigma$, is
a Markov chain. For such an initial distribution, the probabilities
$\pi_k(\sigma)$ are then exactly equal to $\Prb{Y_k = \sigma} = \Prb{X_k \in \sigma}$.
However, the concept of weak lumpability makes no statement about whether
the probability $\Prb{X_k = s}$ for $s \in \sigma$ can be accurately
derived from the knowledge of $\Prb{X_k \in \sigma}$.

Again under the condition that $\Pi = A P \Lambda$,
\cite[Definition 2.2]{markovbounds} defined the matrix $P$ to be $A$-lumpable
if $\Pi A = A P$. In the subsequent remarks, \cite{markovbounds} then
noted that, given $\pi_k$, an exact recovery of the probabilities $p_k(s)$
is possible provided that the initial distribution $p_0$ is compatible
with the $\alpha$ distributions. Actually, \cite[equation (2.4)]{markovbounds}
corresponds almost exactly to \eqref{eq:induc_exact}.

Furthermore, an exact aggregation is called backward bisimulation of type 2
in \cite[Definition 4.3]{bisimwgtautom}, yet again if $\Pi = A P \Lambda$.
\cite{bisimwgtautom} already remarks that the coarsest exact aggregation
cannot be found by the usual partition refinement approach, which we will
see later in \autoref{sec:almost_exlump}. Therefore, efficiently finding partitions $\Omega$
(and distributions $\alpha$) which are an exact aggregation for a given
DTMC or CTMC is currently still an open research problem, to
the best of our knowledge.

\subsection{Error bound for the approximated stationary distribution}
\label{ssec:stat_dist}

This paper focuses on error bounds for the approximated transient distributions
of a Markov chain. However, we will briefly cover a similar error bound
for how well we can approximate the stationary distribution with an aggregated
chain. For simplicity, we will assume in this section that the
DTMC given by $P$ and the CTMC given by $Q$ give rise to a unique
stationary distribution $p$. We first consider the discrete-time case.

As we allow arbitrary matrices for the aggregated step matrix $\Pi$,
we define an approximation for the stationary distribution $p$ of the original
chain as follows. A stationary vector $\pi$ of $\Pi$ simply is a left eigenvector
of the matrix $\Pi$ with eigenvalue $1$, i.e.~it holds that $\pi^{\transp} \Pi = \pi^{\transp}$.
Given such a stationary vector, we can approximate $p^{\transp}$ by the vector
$\pi^{\transp} A$, which is very similar to the approximation
$\widetilde{p}_k^{\transp} = \pi_k^{\transp}A$ of $p_k$. We then get
\begin{align*}
  \norm{\pi^{\transp} A P - \pi^{\transp} A}_1
  &= \norm{\pi^{\transp} A P - \pi^{\transp} \Pi A}_1
  \overset{\textrm{\autoref{lem:1norm_mult}}}{\leq} \norm{\pi}_1 \norm{\Pi A - A P}_\infty
\end{align*}
The leftmost term gives a measure of how far the approximated stationary
distribution $\pi^{\transp} A$ is from being stationary for the
matrix $P$. Note that depending on our choice of $\Pi$, there might not
exist a left eigenvector with eigenvalue $1$.
In the special case where $\Pi$ and $A$ are stochastic matrices, the
existence of $\pi$ is guaranteed, $\pi$ is an actual stationary distribution for
$\Pi$, and we will also have that $\pi^{\transp} A$
is a probability distribution. $\norm{\Pi A - A P}_\infty$ gives
us some sort of measure of how far this probability distribution is from
being stationary. Hence, we can see $\norm{\Pi A - A P}_\infty$ as not
only capturing how well the aggregated
chain reflects the original transient behavior, but also capturing
how well we can approximate the stationary behavior.

If $\Pi$ does not have a left eigenvector with eigenvalue $1$, we can
consider the case where we only have the following approximate statement:
$\norm{\pi^{\transp} \Pi - \pi^{\transp}}_1 \approx 0$.
In this case,
\begin{align*}
  \norm{\pi^{\transp} A P - \pi^{\transp} A}_1
  &= \norm{\pi^{\transp} A P - \pi^{\transp} \Pi A + (\pi^{\transp} \Pi - \pi^{\transp}) A}_1
  \leq \norm{\pi}_1 \norm{\Pi A - A P}_\infty + \norm{\pi^{\transp} \Pi - \pi^{\transp}}_1 \norm{A}_\infty
\end{align*}
For stochastic $A$ and a probability distribution $\pi$, this inequality reduces to
\begin{align*}
  \norm{\pi^{\transp} A P - \pi^{\transp} A}_1
  \leq \norm{\Pi A - A P}_\infty + \norm{\pi^{\transp} \Pi - \pi^{\transp}}_1
\end{align*}

For CTMCs, we have the corresponding result: if $\pi$ is a left
eigenvector of $\Theta$ with eigenvalue $0$, i.e.~$\pi^{\transp} \Theta = 0$, then
\begin{align*}
  \norm{\pi^{\transp} A Q}_1
  &= \norm{\pi^{\transp} A Q - \pi^{\transp} \Theta A}_1
  \overset{\textrm{\autoref{lem:1norm_mult}}}{\leq} \norm{\pi}_1 \norm{\Theta A - A Q}_\infty
\end{align*}
and if we only have that $\norm{\pi^{\transp} \Theta}_1 \approx 0$, then
\begin{align*}
  \norm{\pi^{\transp} A Q}_1 &\leq
  \norm{\pi}_1 \norm{\Theta A - A Q}_\infty + \norm{\pi^{\transp} \Theta}_1 \norm{A}_\infty
\end{align*}

\begin{corollary}
  \label{cor:stat_bnd}
  Consider a DTMC or CTMC with transition matrix $P$ or generator matrix $Q$,
  an aggregated step matrix $\Pi$ or an aggregated evolution matrix $\Theta$,
  and a disaggregation matrix $A$. Let $\pi \in \bbR^{m}$ be
  an approximate stationary vector for the aggregated model (i.e.~$\pi^{\transp} \Pi \approx \pi^{\transp}$
  or $\pi^{\transp} \Theta \approx 0$). Then, the following holds
  for $\pi^{\transp} A$, which is an approximation of the stationary
  distribution of the original chain:
  \begin{align*}
    \norm{\pi^{\transp} A P - \pi^{\transp} A}_1
    &\leq \norm{\pi}_1 \norm{\Pi A - A P}_\infty + \norm{\pi^{\transp} \Pi - \pi^{\transp}}_1 \norm{A}_\infty \quad \textrm{ and } \\
    \norm{\pi^{\transp} A Q}_1 &\leq
    \norm{\pi}_1 \norm{\Theta A - A Q}_\infty + \norm{\pi^{\transp} \Theta}_1 \norm{A}_\infty
  \end{align*}
  If $A$ is stochastic and $\pi$ is a probability distribution, then
  \begin{align*}
    \norm{\pi^{\transp} A P - \pi^{\transp} A}_1
    &\leq \norm{\Pi A - A P}_\infty + \norm{\pi^{\transp} \Pi - \pi^{\transp}}_1 \quad \textrm{ and } \quad
    \norm{\pi^{\transp} A Q}_1 \leq
    \norm{\Theta A - A Q}_\infty + \norm{\pi^{\transp} \Theta}_1
  \end{align*}
\end{corollary}

\begin{remark}
  The bounds given above only measure how close to stationarity $\pi^{\transp} A$ is,
  the distance of $\pi^{\transp} A$ to the actual stationary distribution $p$
  of the original chain can be both smaller or larger. Indeed, as shown by the following
  examples, it is both possible that $\norm{p^{\transp} - \pi^{\transp}A}_1 < \norm{\Pi A - A P}_\infty$
  and that $\norm{p^{\transp} - \pi^{\transp}A}_1 > \norm{\Pi A - A P}_\infty$
  if $p$ is a stationary distribution of $P$ and if $\pi$ is a stationary distribution
  for the stochastic matrix $\Pi$ (so we actually have $\pi^{\transp}\Pi = \pi^{\transp}$
  and not just $\pi^{\transp}\Pi \approx \pi^{\transp}$, and the reduced model is a Markov chain
  as $\Pi$ is stochastic).

  \begin{itemize}
    \item To see that $\norm{p^{\transp} - \pi^{\transp}A}_1 < \norm{\Pi A - A P}_\infty$ is possible,
    consider the following (here we used weighted state space partitioning and $\Pi = A P \Lambda$):
    \begin{align*}
      P &= \frac{1}{4} \cdot \begin{pmatrix}
        1 & 1 & 2 \\
        1 & 2 & 1 \\
        2 & 1 & 1
      \end{pmatrix}, \qquad
      A = \begin{pmatrix}
        \frac{1}{2} & \frac{1}{2} & 0 \\
        0 & 0 & 1
      \end{pmatrix}, \qquad
      \Pi = \frac{1}{8} \cdot \begin{pmatrix}
        5 & 3 \\
        6 & 2
      \end{pmatrix} \\
      \implies p &= \left(\frac{1}{3}, \frac{1}{3}, \frac{1}{3}\right)^{\transp},
      \qquad \pi = \left(\frac{2}{3}, \frac{1}{3}\right)^{\transp},
      \qquad \pi^{\transp} A = \left(\frac{1}{3}, \frac{1}{3}, \frac{1}{3}\right) \\
      \implies \norm{p^{\transp} - \pi^{\transp}A}_1 &= 0
      < \frac{1}{4} = \norm{\Pi A - A P}_\infty
    \end{align*}
    \item To see that $\norm{p^{\transp} - \pi^{\transp}A}_1 > \norm{\Pi A - A P}_\infty$
    is possible as well, consider the following (note that the only
    difference to the previous example is the choice of $A$):
    \begin{align*}
      P &= \frac{1}{4} \cdot \begin{pmatrix}
        1 & 1 & 2 \\
        1 & 2 & 1 \\
        2 & 1 & 1
      \end{pmatrix}, \qquad
      A = \begin{pmatrix}
        \frac{3}{4} & \frac{1}{4} & 0 \\
        0 & 0 & 1
      \end{pmatrix}, \qquad
      \Pi = \frac{1}{16} \cdot \begin{pmatrix}
        9 & 7 \\
        12 & 4
      \end{pmatrix} \\
      \implies p &= \left(\frac{1}{3}, \frac{1}{3}, \frac{1}{3}\right)^{\transp},
      \qquad \pi = \left(\frac{12}{19}, \frac{7}{19}\right)^{\transp},
      \qquad \pi^{\transp} A = \left(\frac{9}{19}, \frac{3}{19}, \frac{7}{19}\right) \\
      \implies \norm{p^{\transp} - \pi^{\transp}A}_1 &= \frac{20}{57} > 0.35
      > 0.344 > \frac{11}{32} = \norm{\Pi A - A P}_\infty
    \end{align*}
  \end{itemize}
  Thus, $\norm{\Pi A - A P}_\infty$ is not a bound on how far off the approximated stationary distribution,
  obtained via the aggregated chain, is from the actual stationary distribution
  of the chain. Only the result from \autoref{cor:stat_bnd} holds: $\norm{\Pi A - A P}_\infty$
  is a measure of how far from stationarity $\pi^{\transp}A$ is, in the sense
  that $\norm{\pi^{\transp} A P - \pi^{\transp} A}_1 \leq \norm{\Pi A - A P}_\infty$.
\end{remark}

\subsection{Tightness of the error bound}
\label{sec:erroptimal}

We now want to show
that the bounds given in \autoref{thm:dtmc_bound} and \autoref{thm:ctmc_bound} are tight in the following sense:
if $P$, $\Pi$ and $A$ are fixed and $\norm{\Pi A - A P}_\infty > 0$, then we can find initial vectors $p_0$ and $\pi_0$ such that
the error bounds are actually achieved. We show this only for matrices
$A$ having non-negative entries.

\begin{theorem}
  Consider an aggregation of a DTMC or CTMC, given by
  a disaggregation matrix $A$, and an aggregated step matrix $\Pi$ (respectively
  an aggregated evolution matrix $\Theta$). Note that we do not fix $\pi_0$ yet.
  If every entry of $A$ is non-negative, if every row of $A$ contains at
  least one strictly positive entry, and if $\norm{\Pi A - A P}_\infty > 0$ (respectively $\norm{\Theta A - A Q}_\infty > 0$),
  then, the following holds: there exists an initial distribution
  $p_0$ and an aggregated initial vector $\pi_0$ such that
  $\pi_0^{\transp} A = p_0^{\transp}$ and such that
  $\norm{\widetilde{p}_1 - p_1}_1 = \norm{\pi_0}_1 \cdot \norm{\Pi A - A P}_\infty$ or, for CTMCs,
  $\lim_{t \to 0, t > 0} \frac{1}{t}\norm{\widetilde{p}_t - p_t}_1 = \norm{\pi_0}_1 \cdot \norm{\Theta A - A Q}_\infty$.
\end{theorem}

\begin{proof}
  Let
  \begin{align*}
    \chi^{\ast} := \argmax_{\chi = 1, \ldots, m} \sum_{i=1}^n \abs{(\Pi A - A P)(\chi, i)}
    \quad \textrm{or, for CTMCs,} \quad
    \chi^{\ast} := \argmax_{\chi = 1, \ldots, m} \sum_{i=1}^n \abs{(\Theta A - A Q)(\chi, i)}
  \end{align*}
  If there are mutliple possible choices for $\chi^{\ast}$, we may choose
  any of them. We now set
  \begin{align*}
    \pi_0(\chi) = \begin{cases}
      0 & \textrm{ if } \chi \neq \chi^{\ast} \\
      \frac{1}{\sum_{i=1}^n \abs{A(\chi^\ast, i)}} & \textrm{ if } \chi = \chi^{\ast}
    \end{cases} \qquad \qquad
    p_0^{\transp} = \pi_0^{\transp}A
  \end{align*}
  Note that setting $p_0$ as above implies that $p_0$ is a probability distribution:
  every entry is non-negative, as both $\pi_0$ and $A$ contain only non-negative entries
  (the latter by assumption), and since
  \begin{align*}
    \sum_{i=1}^n p_0(i) = \sum_{i=1}^n \sum_{\chi=1}^m \pi_0(\chi) A(\chi, i)
    = \sum_{i=1}^n \frac{1}{\sum_{j=1}^n \abs{A(\chi^\ast, j)}} \cdot A(\chi^{\ast}, i) = 1
  \end{align*}
  As we assumed that every row of $A$ contains at least one strictly positive
  entry, we also know that $\sum_{i=1}^n \abs{A(\chi^\ast, i)} > 0$, ensuring that
  we do not divide by zero. But now, in the discrete-time case,
  \begin{align*}
    \norm{\widetilde{p}_1 - p_1}_1
    &= \norm{\pi_0^{\transp} \Pi A - p_0^{\transp} P}_1
    = \norm{\pi_0^{\transp} \Pi A - \pi_0^{\transp} A P}_1
    = \norm{\pi_0^{\transp} (\Pi A - A P)}_1
    = \pi_0(\chi^\ast) \sum_{i=1}^n \abs{(\Pi A - A P)(\chi^\ast, i)} \\
    &= \norm{\pi_0}_1 \cdot \norm{\Pi A - A P}_\infty
  \end{align*}
  where the last equality holds by choice of $\chi^\ast$. Looking at the
  continous-time case, we find that (with the choice of $\pi_0$ and $p_0$
  as given above),
  \begin{align*}
    \left.\frac{\textrm{d}}{\textrm{d}t}\right|_{t = 0} (\widetilde{p}_t^{\transp} - p_t^{\transp})
    &\overset{\textrm{\eqref{eq:error_derivative}}}{=} \pi_0^{\transp}\Theta A - p_0^{\transp}Q
    = \pi_0^{\transp}\Theta A - \pi_0^{\transp}AQ
    = \pi_0^{\transp}(\Theta A - A Q)
  \end{align*}
  Hence, noting that $\norm{\widetilde{p}_0 - p_0}_1 = 0$, we obtain
  \begin{align*}
    \lim_{\substack{t \to 0\\t>0}} \frac{1}{t}\norm{\widetilde{p}_t - p_t}_1 &= \sum_{i=1}^n \lim_{\substack{t \to 0\\t>0}} \frac{\abs{(\widetilde{p}_t - p_t)(i)}}{t}
    \;\;\overset{\textrm{\autoref{lem:absfctlim}}}{=}\;\; \sum_{i=1}^n \abs{\left(\pi_0^{\transp}(\Theta A - A Q)\right)(i)}
    = \norm{\pi_0^{\transp} (\Theta A - A Q)}_1\\
    &= \ldots \textrm{ (as in the discrete-time case) } = \norm{\pi_0}_1 \cdot \norm{\Theta A - A Q}_\infty \qedhere
  \end{align*}
\end{proof}

\subsection{Reduction to arbitrarily low dimension}

As shown in the following proposition, we can actually find dynamic-exact
aggregations of any arbitrarily low dimension.

\begin{proposition}
  \label{prop:schur}
  Consider a DTMC or CTMC on a state space of size $n$. For every $m < n$,
  there either exist matrices $A \in \bbR^{m \times n}$ and $\Pi \in \bbR^{m \times m}$ (respectively $\Theta$)
  or matrices $A \in \bbR^{(m+1) \times n}$ and $\Pi \in \bbR^{(m+1) \times (m+1)}$ (respectively $\Theta$)
  where $A$ has full rank and such that
  \begin{align*}
    \Pi A = A P \quad \textrm{ or, for CTMCs, } \quad \Theta A = A Q
  \end{align*}
\end{proposition}

\begin{proof}
  This follows from the existence of the real Schur decomposition of $P^{\transp}$
  (see \cite[Theorem 7.4.1]{matrixcomputations}).
  Indeed, there exists an orthogonal matrix $U \in \bbR^{n \times n}$ and
  a quasi upper triangular matrix $T \in \bbR^{n \times n}$ such that
  $P^{\transp} = U T U^{\transp}$. Here, quasi upper triangular means that
  \begin{align*}
    T = \begin{pmatrix}
      T_{11} & T_{12} & \ldots & T_{1k} \\
      0 & T_{22} & \ldots & T_{2k} \\
      \vdots &  & \ddots & \vdots \\
      0 & 0 & \ldots & T_{kk}
    \end{pmatrix}
    \textrm{ where } T_{ii} \in \bbR^{1\times 1} \textrm{ or } T_{ii} \in \bbR^{2\times 2} \;\;\;\; \forall i \in \{1,\ldots,k\}
  \end{align*}
  so $T$ is almost upper triangular with the exception that there might be non-zero values
  on the diagonal one below the main diagonal. The $1 \times 1$ blocks on
  the diagonal correspond to the real eigenvalues of the matrix $P^{\transp}$, and the $2 \times 2$ blocks
  to the complex eigenvalues (see \cite[Section 7.4.1]{matrixcomputations} for details). Hence, we have that
  \begin{align*}
    P = U T^{\transp} U^{\transp}
    \implies
    U^{\transp} P = \begin{pmatrix}
      T_{11}^{\transp} & 0 & \ldots & 0 \\
      T_{12}^{\transp} & T_{22}^{\transp} & \ldots & 0 \\
      \vdots &  & \ddots & \vdots \\
      T_{1k}^{\transp} & T_{2k}^{\transp} & \ldots & T_{kk}^{\transp}
    \end{pmatrix} U^{\transp}
  \end{align*}
  We can now simply choose $\Pi$ to correspond to the first $m$ rows and first $m$ columns
  of $T^{\transp}$ if we do not cut one of the $T_{ij}$ blocks in half by
  doing this. Otherwise, we can choose $\Pi$ to correspond to the first $m+1$ rows and first $m+1$ columns
  of $T^{\transp}$. If we then choose $A$ to be the first $m$ or $m+1$ rows
  of $U^{\transp}$ (and all columns), it does hold that $A P = \Pi A$, as required.
  Note that $A$ will have full rank as the rows of $A$ are
  orthonormal vectors.
  The same method can be applied to obtain $A$ and $\Theta$ for a generator
  matrix $Q$.
\end{proof}

\begin{remark}
  In \autoref{prop:schur}, it was stated that $A$ can be chosen to have
  full rank. $\Pi$ can actually also be chosen to have full rank if it is calculated
  using the Schur decomposition, and if the dimension of the reduced state
  space is at most the rank of $P$ (this can be seen easily as the submatrix consisting of the first $m$ rows
  and columns of $T^{\transp}$ in the proof of \autoref{prop:schur} will have
  full rank if the blocks on the diagonal correspond to non-zero eigenvalues, where
  we note that the order in which the eigenvalue blocks appear on the diagonal
  can be chosen arbitrarily in the Schur decomposition with an appropriate algorithm).

  It makes sense to choose full-rank matrices -- otherwise, there would
  still be redundant information in the reduced model. Indeed, if we
  consider arbitrary matrices $A$ and $\Pi$ which do not necessarily have
  full rank, the following trivial aggregation will always be dynamic-exact
  (for simplicity, we assume that $P$ gives rise to a unique stationary
  distribution):
  \begin{align*}
    A = \begin{pmatrix}
      \hpipe & p^{\transp} & \hpipe \\
      & \vdots & \\
      \hpipe & p^{\transp} & \hpipe
    \end{pmatrix}, \qquad
    \Pi = \begin{pmatrix}
      \frac{1}{m} & \ldots & \frac{1}{m} \\
      \vdots & \ddots & \vdots \\
      \frac{1}{m} & \ldots & \frac{1}{m}
    \end{pmatrix} \qquad
    \textrm{where } p \textrm{ is the stationary distribution of } P
  \end{align*}
  In this case, $AP = A$ by choice of $A$, and it is easy to see that
  we also have $\Pi A = A$. In fact, choosing $A = p^{\transp}$ and
  $\Pi = ( 1 ) \in \bbR^{1 \times 1}$ always results in a dynamic-exact aggregation
  with only a single aggregate state. This aggregation does furthermore
  respect the restrictions of weighted state space partitioning, so we
  can always find a dynamic-exact aggregation with $m=1$ aggregate when using
  weighted state space partitioning.
\end{remark}

At this point, one might ask why we should consider dynamic-exact aggregations
with more than a single aggregate or even why we should consider aggregations
which are not dynamic-exact. The biggest disadvantage of a dynamic-exact
aggregation with a low number of aggregates is that the initial error
$\norm{\pi_0^{\transp} A - p_0^{\transp}}_1$ increases as $m$ decreases,
generally speaking, since with decreasing $m$, we have fewer degrees of
freedom for choosing our approximate initial distribution. The other disadvantage
is the computational cost of finding the Schur decomposition or the
stationary distribution. It takes
$\mathcal{O}(n^3)$ time to perform a Schur decomposition (see \cite[Section 7.4]{matrixcomputations}), which can
quickly get out of hand for large state spaces where we are actually
most interested in reducing the state space. As for finding the
stationary distribution: if we are actually interested in the transient
distributions, then the dynamic-exact aggregation with $A = p^{\transp}$ and
$\Pi = ( 1 ) \in \bbR^{1 \times 1}$ will not be helpful at all, and
for stationary analysis, we are interested in aggregation for speeding
up the computation of the stationary distribution, so this dynamic-exact
aggregation is obviously not helpful, either.

While we can, in theory, find dynamic-exact aggregations with reduced state spaces of arbitrarily low
dimension, this is not possible for exact aggregations. Consider a reduction to
a state space of dimension $m=1$. The matrix $\Pi$ will reduce to a scalar value,
and the matrix $A$ to a single row. For $\Pi A = A P$ to hold,
we would therefore need that the row of $A$ is a left eigenvector of $P$.
If we also require that $\pi_0^{\transp} A = p_0^{\transp}$ (note that $\pi_0$ also
reduces to a scalar value for $m=1$) which we need for an exact aggregation,
then we find that the row of $A$ must furthermore be a multiple of the
initial distribution $p_0$. But a left eigenvector of $P$ which is a multiple
of a probability distribution necessarily is a stationary distribution of
the Markov chain, if normed appropriately. Hence, an exact aggregation to
a state space of dimension $1$ exists if, and only if, $p_0$ is a stationary
distribution of the Markov chain.
It might be that a result along the lines of ``an exact aggregation of dimension $k$ exists
if and only if we can find $k$ left eigenvectors of $P$ such that the initial
distribution $p_0$ is a linear combination of these eigenvectors'' holds.
This merits further investigation.

To conclude, we give yet another example:
\begin{align*}
  P = \begin{pmatrix}
    0 & 1 & 0 \\ 0 & 0 & 1 \\ 1 & 0 & 0
  \end{pmatrix}, \qquad
  p_0 = \begin{pmatrix}
    1 \\ 0 \\ 0
  \end{pmatrix}
\end{align*}
This defines a DTMC which cycles along its three states. It is relatively easy
to show that no exact aggregations with reduced state spaces of dimension $1$ or $2$
exist here, but the Schur decomposition approach from the proof of \autoref{prop:schur}
would yield a dynamic-exact aggregation with $1$ or $2$ aggregated states.
One possible Schur decomposition is the following:
\begin{align*}
  P = \underbrace{
    \begin{pmatrix}[1.2]
      0 & \sqrt{\frac{2}{3}} & \frac{1}{\sqrt{3}} \\
      \frac{1}{\sqrt{2}} & -\frac{1}{\sqrt{6}} & \frac{1}{\sqrt{3}} \\
      -\frac{1}{\sqrt{2}} & -\frac{1}{\sqrt{6}} & \frac{1}{\sqrt{3}}
    \end{pmatrix}}_{U}
  \underbrace{\begin{pmatrix}[1.2]
    -\frac{1}{2} & -\frac{\sqrt{3}}{2} & 0 \\
    \frac{\sqrt{3}}{2} & -\frac{1}{2} & 0 \\
    0 & 0 & 1
  \end{pmatrix}}_{T^{\transp}}
  U^{\transp}
\end{align*}
Here, the eigenvalues were ordered such that the complex conjugate block
of the eigenvalues $-\frac{1}{2}\pm \frac{\sqrt{3}}{2}\ivar$ appears first.
We thus have the following dynamic-exact aggregation:
\begin{align*}
  A = \begin{pmatrix}[1.2]
    0 & \frac{1}{\sqrt{2}} & -\frac{1}{\sqrt{2}} \\
    \sqrt{\frac{2}{3}} & -\frac{1}{\sqrt{6}} & -\frac{1}{\sqrt{6}}
  \end{pmatrix}, \qquad
  \Pi = \begin{pmatrix}[1.2]
    -\frac{1}{2} & -\frac{\sqrt{3}}{2} \\
    \frac{\sqrt{3}}{2} & -\frac{1}{2}
  \end{pmatrix}
\end{align*}
Note, however, that this dynamic-exact aggregation is not so useful for
the given initial distribution $p_0$. We will obtain the initial error
\begin{align*}
  \norm{\pi_0^{\transp} A - p_0^{\transp}}_1
  &= \abs{\pi_0(2) \cdot \sqrt{\frac{2}{3}} - 1}
  + \abs{\pi_0(1) \cdot \frac{1}{\sqrt{2}} - \pi_0(2) \cdot \frac{1}{\sqrt{6}}}
  + \abs{- \pi_0(1) \cdot \frac{1}{\sqrt{2}} - \pi_0(2) \cdot \frac{1}{\sqrt{6}}} \\
  &= \frac{1}{\sqrt{6}}\left(\abs{2\pi_0(2)-\sqrt{6}} + \abs{\sqrt{3}\pi_0(1) - \pi_0(2)} + \abs{\sqrt{3}\pi_0(1) + \pi_0(2)}\right)
\end{align*}
We see that $\pi_0(1) = 0$ (sometimes among other choices) always minimizes the error, regardless of the
choice of $\pi_0(2)$ (due to the last two absolute values in the sum above).
Hence, we set $\pi_0(1) = 0$ and get
\begin{align*}
  \norm{\pi_0^{\transp} A - p_0^{\transp}}_1 &= \frac{2}{\sqrt{6}}\left(\abs{\pi_0(2)-\frac{\sqrt{6}}{2}} + \abs{\pi_0(2)}\right)
\end{align*}
Any choice of $\pi_0(2) \in \left[0, \frac{\sqrt{6}}{2}\right]$ minimizes the above expression,
resulting in $\norm{\pi_0^{\transp} A - p_0^{\transp}}_1 = 1$. We could e.g.~choose
$\pi_0 = (0,0)^{\transp}$ to achieve the bound, resulting in $\widetilde{p}_k = (0,0,0)^{\transp}$
for all $k$ which is obviously not a very helpful approximation of $p_k$.

\subsection{Summary of results}

\begin{table}[H]
  \begin{center}
    \begin{tabular}{|c|c|}
      \hline
      \inlinehalfcell{\textbf{DTMCs}} & \inlinehalfcell{\textbf{CTMCs}} \\
      \inlinehalfcell{$\pi_0$ probability distribution, $\Pi$ and $A$ stochastic, $\pi$ stationary probability distribution for $\Pi$}
      & \inlinehalfcell{$\pi_0$ probability distribution, $\Theta$ generator, $A$ stochastic, $\pi$ stationary probability distribution for $\Theta$} \\ \hline
      \multicolumn{2}{|c|}{\inlinefullcell{transient error bounds (precise)}} \\
      \mathhalfcell{\begin{gather*}
        \norm{e_k}_1 \leq \\ \norm{\pi_0^{\transp} A - p_0^{\transp}}_1 + \sum_{j=0}^{k-1} \iprod{\abs{\pi_j}}{\abs{\Pi A - A P} \cdot \mathbf{1}_n}
      \end{gather*}}
      & \mathhalfcell{\begin{gather*}
        \norm{e_t}_1 \leq  \\ \norm{\pi_0^{\transp} A - p_0^{\transp}}_1 + \int_0^t \left< \abs{\pi_u}, \;\; \abs{\Theta A - A Q} \cdot \mathbf{1}_n \right> \dx{u}
      \end{gather*}} \\ \hline
      \multicolumn{2}{|c|}{\inlinefullcell{transient error bounds (imprecise)}} \\
      \inlinehalfcell{$\displaystyle \norm{e_k}_1 \leq \norm{\pi_0^{\transp} A - p_0^{\transp}}_1 + k \cdot \norm{\Pi A - A P}_\infty$}
      & \inlinehalfcell{$\displaystyle \norm{e_t}_1 \leq \norm{\pi_0^{\transp} A - p_0^{\transp}}_1 + t \cdot \norm{\Theta A - A Q}_\infty$} \\ \hline
      \multicolumn{2}{|c|}{\inlinefullcell{stationary error bounds}} \\
      \inlinehalfcell{$\displaystyle \norm{\pi^{\transp} A P - \pi^{\transp} A}_1 \leq \norm{\Pi A - A P}_\infty$}
      & \inlinehalfcell{$\displaystyle \norm{\pi^{\transp} A Q}_1 \leq \norm{\Theta A - A Q}_\infty$} \\ \hline
    \end{tabular}
  \end{center}
  \caption{Overview of error bounds}
  \label{tab:resum}
\end{table}

\autoref{tab:resum} lists the most important error bounds
(from \autoref{thm:dtmc_bound}, \autoref{thm:ctmc_bound} and \autoref{cor:stat_bnd}) for the case
where $\Pi$ and $A$ are stochastic and $\Theta$ is a generator, which is
most relevant in practical applications. $\norm{\Pi A - A P}_\infty$ and
$\norm{\Theta A - A Q}_\infty$ can be seen as measuring how well the
aggregated model reflects the transient and stationary behavior of the
full Markov chain, and $\norm{\pi_0^{\transp} A - p_0^{\transp}}_1$ measures
the error made when approximating the initial distribution, which is
relevant for the faithfulness of the approximation of the transient behavior as well.

\section{Lumpability and aggregatability}
\label{sec:lump}

In this section, we will put the error bounds presented in the previous
section in context by comparing them with concepts from the
related literature. As most publications focus on weighted state space
partitioning as defined in \autoref{ssec:prelim_mc}, we will focus on this
special case in this section, i.e.~the state space $S$ of the original chain is partitioned into aggregates
by $\Omega$, each state $s$ of the original chain is assigned weight $\alpha(s)$
within its aggregate, we set $\Pi = A P \Lambda$, $\Theta = A Q \Lambda$, $\pi_0^{\transp} = p_0^{\transp} \Lambda$,
and the rows of $A$ are the $\alpha$ distributions. In what is to come,
we will index $P$ by pairs of states $r, s \in S$, i.e.~we will write $P(r, s)$,
and we will index $\Pi$ by pairs of aggregates $\rho, \sigma \in \Omega$, i.e.~$\Pi(\rho, \sigma)$.

\subsection{An alternative characterization of dynamic-exactness}

We first compare the notion of exactness and dynamic-exactness as defined
in \autoref{def:exactagg} with a similar notion from \cite{mcaggreactnet}
for aggregations via weighted state space partitioning. Recall the following
notation: $\omega : S \to \Omega$ maps a state $s$ to its aggregate $\omega(s)$.
\cite[Theorem 2]{mcaggreactnet} and \cite[Theorem 9]{mcaggreactnet}
state that the distributions $\widetilde{p}_k$ (respectively $\widetilde{p}_t$)
are equal to $p_k$ (respectively $p_t$) under the following conditions:
\begin{enumerate}[(i)]
  \item \label{it:cond1}$\displaystyle \forall s, s' \in S \textrm{ s.t. } \omega(s) = \omega(s'): \forall \rho \in \Omega:
  \frac{\sum_{r \in \rho} \alpha(r) P(r,s)}{\alpha(s)}
  = \frac{\sum_{r \in \rho} \alpha(r) P(r,s')}{\alpha(s')}$
  \item \label{it:initrespect}the initial distribution $p_0$ is compatible with the $\alpha$ distributions
  (compare with \autoref{def:compat}; this is called $p_0$ respects the $\alpha$ distributions in \cite{mcaggreactnet}).
  Recall that this is the case when $p_0^{\transp} \Lambda A = p_0^{\transp}$, or
  $\pi_0^{\transp} A = p_0^{\transp}$ as $p_0^{\transp} \Lambda = \pi_0^{\transp}$
  if we consider weighted state space partitioning.
\end{enumerate}

The next proposition, \autoref{prop:cond1}, shows that \ref{it:cond1}
is equivalent to $\norm{\Pi A - A P}_\infty = 0$ (respectively $\norm{\Theta A - A Q}_\infty = 0$) and thus by
\autoref{def:exactagg} equivalent to dynamic-exactness. It follows that conditions
\ref{it:cond1} and \ref{it:initrespect} together are equivalent to
exactness, and hence imply (see \autoref{cor:exact_errbnd}) that the
transient distributions $\widetilde{p}_k$ (respectively $\widetilde{p}_t$)
agree with $p_k$ (respectively $p_t$), which is exactly the statement
of \cite[Theorem 2]{mcaggreactnet} and \cite[Theorem 9]{mcaggreactnet},
but obtained via a different proof.

\begin{proposition}
  \label{prop:cond1}
  Given an irreducible DTMC, a partition $\Omega$ of its state space, and arbitrary probability distributions $\alpha_\sigma$
  with support on $\sigma \in \Omega$, let $\Pi = A P \Lambda$ (where the row vectors of $A$
  correspond to the distributions $\alpha_\sigma$). Then:
  \begin{gather*}
    \norm{\Pi A - A P}_\infty = 0 \\
    \iff \\
    \forall s \in S: \alpha(s) > 0 \textrm{ and } \\
    \forall s, s' \in S \textrm{ s.t. } \omega(s) = \omega(s'): \forall \rho \in \Omega:
    \frac{\sum_{r \in \rho} \alpha(r) P(r,s)}{\alpha(s)}
    = \frac{\sum_{r \in \rho} \alpha(r) P(r,s')}{\alpha(s')}
  \end{gather*}
  Furthermore, the assumption $\Pi = A P \Lambda$ is only needed for the
  $\impliedby$ direction, the other direction $\implies$ in the equation
  above holds for an arbitrary choice of $\Pi$.
  The same statement holds for irreducible CTMCs with $\Pi$ replaced by $\Theta$ and
  $P$ replaced by $Q$.
\end{proposition}

\begin{proof}
  We have that
  \begin{align}
    \label{eq:tau_zero}
    \begin{split}
      \norm{\Pi A - A P}_\infty = 0
      &\iff \forall \rho \in \Omega: \forall s \in S: (\Pi A)(\rho, s) = (A P)(\rho, s) \\
      &\iff \forall \rho \in \Omega: \forall s \in S: \sum_{\sigma \in \Omega}\Pi(\rho, \sigma) A(\sigma, s) = \sum_{r \in S}A(\rho, r) P(r, s) \\
      &\iff \forall \rho \in \Omega: \forall s \in S: \Pi(\rho, \omega(s)) \alpha(s) = \sum_{r \in \rho} \alpha(r) P(r, s) \\
      &\iff \forall \rho, \sigma \in \Omega: \forall s \in \sigma: \;\; \alpha(s)\Pi(\rho, \sigma) = \sum_{r \in \rho} \alpha(r) P(r,s)
    \end{split}
  \end{align}
  We first show by contradiction: $\norm{\Pi A - A P}_\infty = 0$ implies
  $\forall s \in S: \alpha(s) > 0$.
  Call $S_0 := \{s_0 \in S : \alpha(s_0) = 0\} \neq S$
  (we have $S_0 \neq S$ because the $\alpha_\sigma$ are probability distributions,
  so there must be at least one state $s_{\neg 0}$ with $\alpha(s_{\neg 0}) > 0$).
  If $S_0 \neq \varnothing$, then,
  by our assumption of irreducibility, there must be some state $s \in S_0$ and
  some state $r \in S \setminus S_0$ (hence $\alpha(r) > 0$) with $P(r, s) > 0$.
  In particular, we have that $\sum_{r' \in \omega(r)} \alpha(r') P(r', s) > 0$
  while $\alpha(s) \Pi(\omega(r),\omega(s)) = 0$. Thus, since the last statement
  of \eqref{eq:tau_zero} does not hold in this case (it is violated for
  $\rho = \omega(r)$, $\sigma = \omega(s)$ and $s$), it follows that $\norm{\Pi A - A P}_\infty > 0$.
  $S_0 \neq \varnothing$ is thus impossible when $\norm{\Pi A - A P}_\infty = 0$, and hence, writing $\circledast$ for $\forall s \in S: \alpha(s) > 0$:
  \begin{align*}
    \norm{\Pi A - A P}_\infty = 0
    &\iff \circledast \textrm{ and } \forall \rho, \sigma \in \Omega: \forall s \in \sigma: \;\; \alpha(s)\Pi(\rho, \sigma) = \sum_{r \in \rho} \alpha(r) P(r,s) \\
    &\iff \circledast \textrm{ and } \forall \rho, \sigma \in \Omega: \forall s \in \sigma: \;\; \Pi(\rho, \sigma) = \frac{\sum_{r \in \rho} \alpha(r) P(r,s)}{\alpha(s)} \\
    &\overset{\circledcirc}{\iff} \circledast \textrm{ and } \forall \rho, \sigma \in \Omega: \forall s, s' \in \sigma: \;\; \frac{\sum_{r \in \rho} \alpha(r) P(r,s)}{\alpha(s)} = \frac{\sum_{r \in \rho} \alpha(r) P(r,s')}{\alpha(s')}
  \end{align*}
  So far, we did not use the assumption $\Pi = A P \Lambda$.
  We now show $\circledcirc$ and note that we only need $\Pi = A P \Lambda$
  for the direction $\impliedby$:
  \begin{itemize}
    \item $\implies$: this is immediately clear since $\Pi(\rho, \sigma)$ does not depend on $s$.
    \item $\impliedby$: if we have
    \begin{align*}
      \frac{\sum_{r \in \rho} \alpha(r) P(r,s)}{\alpha(s)} = \frac{\sum_{r \in \rho} \alpha(r) P(r,s')}{\alpha(s')}
      \;\; \forall s, s' \in \sigma
    \end{align*}
    then, it holds that, for any $s \in \sigma$,
    \begin{align*}
      \frac{\sum_{r \in \rho} \alpha(r) P(r, s)}{\alpha(s)}
      = \sum_{s' \in \sigma} \alpha(s') \cdot \frac{\sum_{r \in \rho} \alpha(r) P(r, s')}{\alpha(s')}
      = \sum_{r \in \rho} \alpha(r) \sum_{s' \in \sigma} P(r, s') = (AP\Lambda)(\rho, \sigma) = \Pi(\rho, \sigma)
    \end{align*}
    where the first equality holds since the sum on the right hand side consists
    of a weighted sum of terms which are equal, and since the weights sum up to $1$.
  \end{itemize}
  Note that the above proof also works for the continuous-time case.
\end{proof}

\autoref{prop:cond1} also allows us to justify the choice $\Pi = A P \Lambda$
(or $\Theta = A Q \Lambda$) for weighted state space partitioning:

\begin{corollary}
  \label{cor:error0_pi_APLambda}
  Given an irreducible DTMC (or CTMC),
  assume that the partition $\Omega$ is fixed, and that the distributions
  $\alpha$ are fixed as well (but arbitrary). Further assume that there
  exists some (arbitrary) matrix $\Pi \in \bbR^{m \times m}$ (respectively, $\Theta \in \bbR^{m \times m}$)
  such that $\norm{\Pi A - A P}_\infty = 0$ (respectively $\norm{\Theta A - A Q}_\infty = 0$).
  Then, it actually holds that $\Pi = A P \Lambda$ (respectively $\Theta = A Q \Lambda$),
  i.e.~$\Pi = A P \Lambda$ is the unique minimizer of
  $\norm{\Pi A - A P}_\infty$ \textbf{if} the minimum is $0$.
\end{corollary}

\begin{proof}
  This actually already follows from what we have proved in \autoref{prop:cond1}:
  if there is some $\Pi$ such that $\norm{\Pi A - A P}_\infty = 0$,
  then, by \autoref{prop:cond1}, it follows that
  \begin{align}
    \label{eq:min0achieved}
    \begin{gathered}
      \forall s \in S: \alpha(s) > 0 \textrm{ and } \\
      \forall s, s' \in S \textrm{ s.t. } \omega(s) = \omega(s'): \forall \rho \in \Omega:
      \frac{1}{\alpha(s)}\sum_{r \in \rho} \alpha(r) P(r,s)
      = \frac{1}{\alpha(s')}\sum_{r \in \rho} \alpha(r) P(r,s')
    \end{gathered}
  \end{align}
  Here, we used
  that the $\implies$ direction in \autoref{prop:cond1} does not need
  the assumption $\Pi = A P \Lambda$ but holds for general $\Pi$. But in
  fact, if we take a closer look at \eqref{eq:tau_zero} in the proof of
  \autoref{prop:cond1}, we see that furthermore,
  \begin{align*}
    \norm{\Pi A - A P}_\infty = 0 \implies \forall \rho, \sigma \in \Omega:
    \forall s \in \sigma: \Pi(\rho, \sigma) = \frac{1}{\alpha(s)}\sum_{r \in \rho} \alpha(r) P(r,s)
  \end{align*}
  As $\alpha(s) > 0$ by \eqref{eq:min0achieved}, $\Pi$ is therefore uniquely determined
  if $\norm{\Pi A - A P}_\infty = 0$.
  We now apply the reverse direction $\impliedby$ of \autoref{prop:cond1} with an a priori potentially different
  matrix $\Pi$, namely $\Pi = A P \Lambda$. \autoref{prop:cond1} tells
  us that if \eqref{eq:min0achieved} holds and if we choose $\Pi = A P \Lambda$,
  then $\norm{\Pi A - A P}_\infty = 0$ which concludes the proof by showing
  that our unique $\Pi$ is indeed the matrix $A P \Lambda$. Note again
  that this proof can also be applied to CTMCs.
\end{proof}

\begin{remark}
  We can restate \autoref{prop:cond1} slightly differently: if we set $\Pi = A P \Lambda$,
  then $\norm{\Pi A - A P}_\infty = 0$ for DTMCs
  if, and only if,
  \begin{align*}
    \frac{\alpha(s)}{\alpha(s')}
    = \frac{\sum_{r \in \rho} \alpha(r) P(r,s)}{\sum_{r \in \rho} \alpha(r) P(r,s')}
  \end{align*}
  for any two states $s$ and $s'$ in the same aggregate $\sigma$ and any aggregate
  $\rho$ such that the probability to go from $\rho$ to $s'$ is positive.
  
  Note that for CTMCs, $\sum_{r \in \rho} \alpha(r) Q(r,s)$ can be negative
  if $s \in \rho$. However, if, in addition, $s' \in \rho$ and if $\norm{\Theta A - A Q}_\infty = 0$,
  then $\sum_{r \in \rho} \alpha(r) Q(r,s')$ must be negative as well by \autoref{prop:cond1},
  because $\alpha(s) > 0$ and $\alpha(s') > 0$. Usually, $\sum_{r \in \rho} \alpha(r) Q(r,s)$
  can be interpreted as the rate at which probability mass is flowing from $\rho$ to $s$,
  under the assumption that $s \notin \rho$ and that the mass within $\rho$ is distributed
  according to $\alpha$. Now, if $s \in \rho$, this is no longer the case.
  Instead,
  \begin{align*}
    \sum_{r \in \rho} \alpha(r) Q(r,s) = \underbrace{\alpha(s) Q(s,s)}_{\textrm{negative of the rate at which mass is exiting }s} + \underbrace{\sum_{\substack{r \in \rho\\r \neq s}} \alpha(r) Q(r,s)}_{\textrm{rate at which mass is flowing from }\rho\setminus\{s\}\textrm{ to }s}
  \end{align*}
  again under the assumption that the mass within $\rho$ is distributed
  according to $\alpha$.
\end{remark}

With the help of \autoref{prop:cond1}, we can now apply other results
from \cite{mcaggreactnet} (we still assume that we aggregate via weighted
state space partitioning and that $\Pi = A P \Lambda$, respectively $\Theta = A Q \Lambda$):
\begin{itemize}
  \item By \cite[Theorem 7]{mcaggreactnet}, if $\norm{\Pi A - A P}_\infty = 0$
  and if the DTMC is aperiodic and irreducible, then we have $\norm{\pi_k^{\transp} - p_k^{\transp}\Lambda}_1 \to 0$
  for $k \to \infty$ (i.e.~the approximate aggregate probabilities $\pi_k(\sigma)$
  converge to the exact aggregate probability $\Prb{X_k \in \sigma}$) and
  $\frac{p_k(s)}{\pi_k(\omega(s))} \to \alpha(s)$. For periodic chains,
  we have to consider $\frac{1}{k}\sum_{i=1}^k p_i(s)$ instead of $p_k(s)$,
  see \cite[Theorem 6]{mcaggreactnet}. In addition, the
  stationary distribution $p$ of the irreducible DTMC satisfies $p^{\transp} = \pi^{\transp} A$
  where $\pi$ is the stationary distribution of the aggregated
  chain in both periodic and aperiodic cases, which also follows from \autoref{cor:stat_bnd}.
  \item By \cite[Theorem 11]{mcaggreactnet}, if $\norm{\Theta A - A Q}_\infty = 0$
  and if the CTMC is irreducible, then we have $\norm{\pi_t^{\transp} - p_t^{\transp}\Lambda}_1 \to 0$
  for $t \to \infty$ and
  $\frac{p_t(s)}{\pi_t(\omega(s))} \to \alpha(s)$. In addition, the
  stationary distribution $p$ of the CTMC satisfies $p^{\transp} = \pi^{\transp} A$
  where $\pi$ is the stationary distribution of the aggregated
  chain, which again follows from \autoref{cor:stat_bnd} as well.
\end{itemize}

The notion of exact and dynamic-exact aggregation has thus already appeared
in the literature, but only in the setting of weighted state space partitioning,
and \cite{mcaggreactnet} is one of the few papers considering this concept.
Most papers consider different types of lumpability and aggregatability, which
often imply exactness or dynamic-exactness.

\subsection{Dynamic-exactness compared to other lumpability concepts}

The following definition was given in \cite[Definition 1]{exactordlump}:
\begin{definition}
  \label{def:ordlump}
  A partition $\Omega = \{\Omega_1, \ldots, \Omega_m\}$ of the state space
  of a DTMC is called \textbf{ordinarily lumpable} if
  \begin{align}
    \label{eq:ordinarylump}
    \forall r, r' \in S \textrm{ s.t.~} \omega(r) = \omega(r'):
    \forall \sigma \in \Omega: \qquad
    \sum_{s \in \sigma} P(r, s) = \sum_{s \in \sigma} P(r', s)
  \end{align}
  That is, for any two states in the same aggregate, the summed outgoing
  probabilities to any other aggregate must be identical. For CTMCs, a
  partition is called ordinarily lumpable if \eqref{eq:ordinarylump}
  holds with $P$ replaced by $Q$.
\end{definition}

Equivalently, a DTMC is ordinarily lumpable w.r.t.~a given partition if
there exists a (stochastic) matrix $\Pi \in \bbR^{m \times m}$ (which is then uniquely
defined) such that $\Lambda \Pi = P \Lambda$.
For an ordinarily lumpable partition, it is easy to show that
$\pi_k(\sigma) = \sum_{s \in \sigma} \widetilde{p}_k(s) = \sum_{s \in \sigma} p_k(s)$
if $\Pi = AP\Lambda$, for any $\sigma \in \Omega$, all $k$,
and any initial distribution $p_0$
(and independent of the choice of $\alpha$ as long as the distributions
$\alpha$ are consistent with $\Omega$), i.e.~the probability $\Prb{X_k \in \sigma}$
is exactly equal to $\pi_k(\sigma)$. A simple induction suffices to prove this,
see \cite[Theorem 5]{exactordlump}.
The same holds in the continuous-time case.

Note that ordinary lumpability is called strong lumpability in \cite{finmc}, which considers
only DTMCs. As before, define the process $Y_k$ by $Y_k = \sigma \in \Omega \iff X_k \in \Omega$.
By \cite[Theorem 6.3.2]{finmc}, ordinary lumpability is equivalent to
the following: for every initial distribution $p_0$, $Y_k$ is a Markov chain
(whose transition probabilities do not depend on the choice of $p_0$).

\cite[Definition 1]{exactordlump} also defines exact lumpability:
\begin{definition}
  \label{def:exlump}
  A partition $\Omega = \{\Omega_1, \ldots, \Omega_m\}$ of the state space
  of a DTMC is called \textbf{exactly lumpable} if
  \begin{align}
    \label{eq:exactlump}
    \forall s, s' \in S \textrm{ s.t.~} \omega(s) = \omega(s'):
    \forall \rho \in \Omega: \qquad
    \sum_{r \in \rho} P(r, s) = \sum_{r \in \rho} P(r, s')
  \end{align}
  That is, for any two states in the same aggregate, the summed incoming
  probabilities from any other aggregate must be identical. For CTMCs, a
  partition is called exactly lumpable if \eqref{eq:exactlump}
  holds with $P$ replaced by $Q$.

  Furthermore, a partition $\Omega$ is called \textbf{strictly lumpable}
  if it is both ordinarily and exactly lumpable.
\end{definition}

We can again give an equivalent definition in terms of matrix products:
a DTMC is exactly lumpable w.r.t.~a given partition if
there exists a (not necessarily stochastic) matrix $\widetilde{\Pi} \in \bbR^{m \times m}$ (which is then uniquely
defined) such that $\widetilde{\Pi} \Lambda^{\transp} = \Lambda^{\transp} P$. Note that
we could also use the following defining equality instead: exact lumpability
is equivalent to the existence of a (unique stochastic) matrix $\Pi \in \bbR^{m \times m}$
such that $\Pi A = A P$ where the rows of $A$ correspond to uniform
distributions over the respective aggregates (see \autoref{prop:exlump_impl_dynex}).

If a partition is ordinarily (exactly) lumpable for a CTMC, then the partition
is also ordinarily (exactly) lumpable for any uniformisation of the CTMC, regardless
of the uniformisation rate, see the remarks after \cite[Definition 1]{exactordlump}.

\cite[Definition 2.1]{probalgmcagg} also defines lumpability. Note that this definition
of lumpability agrees with the definition of ordinary lumpability given above.
\cite[Definition 2.1]{probalgmcagg} further defines deflatability and aggregatability:

\begin{definition}
  \label{def:deflat}
  A partition $\Omega = \{\Omega_1, \ldots, \Omega_m\}$ of the state space
  of a DTMC, together with distributions $\alpha_\sigma \in \bbR^n$ with support
  on $\sigma \in \Omega$, is called \textbf{deflatable} if
  \begin{align}
    \label{eq:deflatable}
    \forall r \in S: \forall s \in S: \qquad
    P(r, s) = \alpha(s) \cdot \sum_{s' \in \omega(s)} P(r, s')
  \end{align}
  In words: the probability to go from $r$ to $s$ only depends on $r$ and the
  aggregated state $\omega(s)$ as well as a factor which depends on $s$, but not on $r$.
  Another description: after a jump into a partition element $\sigma$,
  the particular target state can be chosen according to the probability measure
  $\alpha_\sigma$, independently of where the jump started.

  The partition $\Omega$, together with distributions $\alpha$, is further
  called \textbf{aggregatable} if it is deflatable and if $\Omega$ is ordinarily lumpable.
\end{definition}

\autoref{def:deflat} cannot be extended to CTMCs easily.

\begin{proposition}
  \label{prop:exlump_impl_dynex}
  Assume a partition $\Omega$ is given for an irreducible DTMC or CTMC.
  When setting $\Pi = AP\Lambda$ (respectively $\Theta = AQ\Lambda$) and
  $\alpha(s) = \frac{1}{\abs{\omega(s)}}$, then:
  \begin{align*}
    \norm{\Pi A - A P}_\infty = 0 \;\; \iff \;\; \Omega \textrm{ is exactly lumpable}
  \end{align*}
\end{proposition}

\begin{proof}
  By \autoref{prop:cond1},
  $\norm{\Pi A - A P}_\infty = 0$ is equivalent to (in the following, we use that $\alpha(s) = \frac{1}{\abs{\omega(s)}}$)
  \begin{align*}
    \overbrace{\forall s, s' \in S \textrm{ s.t.~} \omega(s) = \omega(s'): \forall \rho \in \Omega:}^{\forall \ldots}
    &\;\;
    \alpha(s') \sum_{r \in \rho} \alpha(r) P(r, s) = \alpha(s) \sum_{r \in \rho} \alpha(r) P(r, s') \\
    \hskip-0.5em\iff \; \forall \ldots:&
    \;\;
    \frac{1}{\abs{\omega(s')}} \sum_{r \in \rho} \frac{1}{\abs{\rho}} P(r, s) = \frac{1}{\abs{\omega(s)}} \sum_{r \in \rho} \frac{1}{\abs{\rho}} P(r, s') \\
    \hskip-0.5em\iff \; \forall \ldots:&
    \;\;
    \sum_{r \in \rho} P(r, s) = \sum_{r \in \rho} P(r, s') \\
    \hskip-0.5em\iff &\;\; \Omega \textrm{ is exactly lumpable}
  \end{align*}
  The same calculation holds for CTMCs.
\end{proof}

\begin{proposition}
  \label{prop:deflat_impl_dynex}
  Consider an irreducible DTMC, a partition $\Omega$ of its state space
  and distributions $\alpha_{\sigma} \in \bbR^n$ with support on $\sigma \in \Omega$.
  If $\Omega$ and the distributions $\alpha$ are deflatable, then $\norm{\Pi A - A P}_\infty = 0$.
\end{proposition}

\begin{proof}
  Again by \autoref{prop:cond1},
  $\norm{\Pi A - A P}_\infty = 0$ is equivalent to
  \begin{align*}
    &\overbrace{\forall s, s' \in S \textrm{ s.t.~} \omega(s) = \omega(s'):
    \forall \rho \in \Omega:}^{\forall \ldots} \;
    \alpha(s') \sum_{r \in \rho} \alpha(r) P(r, s) = \alpha(s) \sum_{r \in \rho} \alpha(r) P(r, s') \\
    \overset{\textrm{deflatability}}{\iff} \; &\forall \ldots: \;\;
    \alpha(s') \sum_{r \in \rho} \alpha(r) \alpha(s) \sum_{\widehat{s} \in \omega(s)} P(r, \widehat{s}) = \alpha(s) \sum_{r \in \rho} \alpha(r) \alpha(s') \sum_{\widehat{s} \in \omega(s')} P(r, \widehat{s}) \\
    \iff \; &\forall \ldots: \;\;
    \alpha(s')\alpha(s) \sum_{r \in \rho} \alpha(r) \sum_{\widehat{s} \in \omega(s)} P(r, \widehat{s}) = \alpha(s)\alpha(s') \sum_{r \in \rho} \alpha(r) \sum_{\widehat{s} \in \omega(s')} P(r, \widehat{s})
  \end{align*}
  The statement in the last line ist true since $\omega(s) = \omega(s')$ by
  the assumption on $s$ and $s'$.
\end{proof}

We next show that none of the lumpability concepts above are necessary conditions
for $\norm{\Pi A - A P}_\infty = 0$. Hence, except for \cite{mcaggreactnet},
a large part of the literature has treated stricter than necessary
conditions in order for dynamic-exact aggregation to be possible.
None of the definitions of ordinary and exact lumpability as
well as deflatability take into account the initial distribution, so none of
these conditions are sufficient for an exact aggregation. Exact lumpability and deflatability
only imply dynamic-exactness. At the same time, all three concepts (ordinary \& exact lumpability, deflatability) are still useful
since they are easier to check computationally, and can thus be relevant for practical applications.
This will be discussed in more detail in the next section.

\begin{proposition}
  \label{prop:lump_nimpl_dynex}
  There are partitions $\Omega$ of the state space of a DTMC and probability
  distributions $\alpha_\sigma$ with support on $\sigma \in \Omega$
  which are dynamic-exact (when $\Pi = AP\Lambda$),
  but where $\Omega$ is neither ordinary lumpable, nor exactly lumpable, nor are $\Omega$
  and the distributions $\alpha$ deflatable.
\end{proposition}

\begin{proof}
  We consider the state space $S = \{1, 2, 3\}$, the aggregation $\Omega = \{\{1\}, \{2,3\}\}$
  and $\alpha(1) = 1, \alpha(2) = \frac{1}{4}, \alpha(3) = \frac{3}{4}$ as well as the
  DTMC given by the following transition matrix:
  \begin{align*}
    P &= \begin{pmatrix}[1.2]
      0 & \frac{1}{4} & \frac{3}{4} \\
      0 & \frac{1}{2} & \frac{1}{2} \\
      \frac{4}{9} & \frac{1}{18} & \frac{1}{2}
    \end{pmatrix} \\
    \implies \qquad \Pi &= \begin{pmatrix}
      0 & 1 \\
      \frac{1}{3} & \frac{2}{3}
    \end{pmatrix},
    \qquad
    A = \begin{pmatrix}
      1 & 0 & 0 \\
      0 & \frac{1}{4} & \frac{3}{4}
    \end{pmatrix},
    \qquad
    \Lambda = \begin{pmatrix}
      1 & 0 \\
      0 & 1 \\
      0 & 1
    \end{pmatrix}
  \end{align*}
  It is easy to check that $\Pi A = A P$ holds, i.e.~this aggregation
  is dynamic-exact and $\norm{\Pi A - A P}_\infty = 0$. Indeed, we
  have
  \begin{align*}
    \Pi A = \begin{pmatrix}[1.2]
      0 & \frac{1}{4} & \frac{3}{4} \\
      \frac{1}{3} & \frac{1}{6} & \frac{1}{2}
    \end{pmatrix} = A P
  \end{align*}
  We now show that none of the stated properties hold for $\Omega$ and $\alpha$:
  \begin{itemize}
    \item ordinary lumpability: since $\omega(2) = \omega(3)$, by \eqref{eq:ordinarylump}
    in \autoref{def:ordlump}, we would need
    \begin{align*}
      \underbrace{P(2, 2) + P(2, 3)}_{= 1} = \sum_{s \in \{2, 3\}} P(2, s) = \sum_{s \in \{2, 3\}} P(3, s) = \underbrace{P(3, 2) + P(3, 3)}_{=\frac{5}{9}}
    \end{align*}
    which is clearly not true.
    \item exact lumpability: since $\omega(2) = \omega(3)$, by \eqref{eq:exactlump}
    in \autoref{def:exlump}, we would need
    \begin{align*}
      \underbrace{P(2, 2) + P(3, 2)}_{= \frac{5}{9}} = \sum_{r \in \{2, 3\}} P(r, 2) = \sum_{r \in \{2, 3\}} P(r, 3) = \underbrace{P(2, 3) + P(3, 3)}_{=1}
    \end{align*}
    which is clearly not true.
    \item deflatability: since $\omega(2) = \omega(3)$, by \eqref{eq:deflatable}
    in \autoref{def:deflat}, we would need
    \begin{align*}
      \frac{1}{2} = P(2, 2) = \alpha(2) \sum_{s \in \{2, 3\}} P(2, s) = \alpha(2) \cdot 1 \implies \alpha(2) = \frac{1}{2} \\
      \frac{1}{18} = P(3, 2) = \alpha(2) \sum_{s \in \{2, 3\}} P(3, s) = \alpha(2) \cdot \frac{5}{9} \implies \alpha(2) = \frac{1}{10} 
    \end{align*}
    so the given $\Omega$ and $\alpha$ are not deflatable,
    and there is even no other choice of $\alpha$ such that $\Omega$ and $\alpha$
    would be deflatable.
  \end{itemize}
  In fact, there is no aggregation $\Omega$ with $\Omega \neq \{S\}$ and
  $\Omega \neq \{\{1\}, \{2\}, \{3\}\}$ (the trivial partitions) which is
  ordinary or exactly lumpable, or for which deflatable $\alpha$ distributions
  exist.
\end{proof}

We have now seen some concepts which imply dynamic-exactness but not
vice versa. However, as mentioned already briefly, dynamic-exactness does
imply a type of lumpability: weak lumpability.

\begin{definition}
  \label{def:weaklump}
  Consider a DTMC or CTMC and a partition $\Omega$ of its state space.
  Denote the state of the chain at time $k$ or $t$ by $X_k$ or $X_t$, respectively.
  Further define the processes $Y_k := \omega(X_k)$ and $Y_t := \omega(X_t)$.
  The Markov chain is said to be \textbf{weakly lumpable} w.r.t.~$\Omega$
  if there exists an initial distribution $p_0$ for $X_k$ or $X_t$ such
  that the process $Y_k$ or $Y_t$ is itself a (time-homogeneous) Markov chain on $\Omega$.
\end{definition}

\begin{proposition}
  \label{prop:dynex_impl_weaklump}
  Consider a DTMC or CTMC, a partition $\Omega$ of its state space and
  distributions $\alpha_\sigma$ with support on $\sigma \in \Omega$ such
  that $\norm{\Pi A - A P}_\infty = 0$ or $\norm{\Theta A - A Q}_\infty = 0$
  (where $\Pi = A P \Lambda, \Theta = A Q \Lambda$). Then, the Markov chain
  is weakly lumpable w.r.t.~$\Omega$, and any initial distribution which
  is compatible with the $\alpha$ distributions will result in the process
  $Y_k$ or $Y_t$ being a Markov chain.
\end{proposition}

\begin{proof}
  The discrete-time case was already proven in \cite[Theorem 6.4.4]{finmc}.
  
  We look at the continuous-time case. Let $p_0$ such that $p_0^{\transp} \Lambda A = p_0^{\transp}$
  (i.e.~$p_0$ is compatible with the $\alpha$ distributions). We claim that
  $Y_t$ is a Markov chain with generator $\Theta$ in this case. To prove
  this claim, it is sufficient to show that
  \begin{align}
    \label{eq:yt_gen_theta}
    \underbrace{p_0^{\transp} e^{Q t} \Lambda}_{\textrm{dist.~of }Y_t} = \underbrace{p_0^{\transp} \Lambda}_{\textrm{dist.~of }Y_0} \cdot \; e^{\Theta t}
  \end{align}
  since the above equation shows that the distribution of $Y_t$ can be calculated
  by applying the generator $\Theta$ to the initial distribution of $Y_0$. If this
  holds for arbitrary $t$, then we can conclude that $Y_t$ is indeed a
  Markov chain with generator $\Theta$. We proceed to prove that \eqref{eq:yt_gen_theta}
  does indeed hold. Note that we have
  that $A Q^k = A Q Q^{k-1} = \Theta A Q Q^{k-2} = \ldots = \Theta^k A$ by assumption since
  $\norm{\Theta A - A Q}_\infty = 0$.
  \begin{align*}
    p_0^{\transp} e^{Q t} \Lambda
    &= p_0^{\transp}\Lambda A \left(\sum_{k \geq 0} \frac{t^k Q^k}{k!}\right) \Lambda
    = \sum_{k \geq 0} \frac{t^k}{k!} \cdot p_0^{\transp}\Lambda AQ^k\Lambda
    = \sum_{k \geq 0} \frac{t^k}{k!} \cdot p_0^{\transp}\Lambda \Theta^k A\Lambda
    = p_0^{\transp}\Lambda \left(\sum_{k \geq 0} \frac{t^k \Theta^k}{k!}\right)
    = p_0^{\transp}\Lambda e^{\Theta t}
  \end{align*}
  where we used that $A \Lambda = I$.
\end{proof}

\begin{proposition}
  \label{prop:weaklump_nimpl_dynex}
  There exists a DTMC which is weakly lumpable w.r.t.~a partition $\Omega$,
  but no stochastic matrix $\Pi$ (not necessarily $\Pi = A P \Lambda$) and distributions $\alpha_\sigma$ with support on $\sigma \in \Omega$
  exist such that $\norm{\Pi A - A P}_\infty = 0$.
\end{proposition}

\begin{proof}
  Consider
  \begin{align*}
    P = \begin{pmatrix}[1.2]
      0 & \frac{1}{2} & \frac{1}{2} \\
      \frac{1}{4} & \frac{1}{4} & \frac{1}{2} \\
      \frac{1}{2} & \frac{1}{4} & \frac{1}{4} \\
    \end{pmatrix}, \qquad \Omega = \{\{1, 2\}, \{3\}\}
  \end{align*}
  Note that
  \begin{align*}
    P \Lambda = P \begin{pmatrix}
      1 & 0 \\ 1 & 0 \\ 0 & 1
    \end{pmatrix} = \begin{pmatrix}[1.2]
      \frac{1}{2} & \frac{1}{2} \\
      \frac{1}{2} & \frac{1}{2} \\
      \frac{3}{4} & \frac{1}{4}
    \end{pmatrix}
    = \Lambda \begin{pmatrix}[1.2]
      \frac{1}{2} & \frac{1}{2} \\
      \frac{3}{4} & \frac{1}{4}
    \end{pmatrix} =: \Lambda \Pi
  \end{align*}
  Hence, this DTMC is ordinarily lumpable w.r.t.~$\Omega$.
  By \cite[Theorem 6.4.4]{finmc} (or by equation (3) on page 135 of \cite{finmc}), the DTMC is therefore also weakly lumpable w.r.t.~$\Omega$.
  In order for $\norm{\Pi A - A P}_\infty = 0$ to hold (for some matrix $A$ consistent with $\Omega$ and for some matrix $\Pi$, not
  necessarily the $\Pi$ given above), we would need
  \begin{align*}
    \Pi A = \begin{pmatrix}
      \Pi(1, 1) & 1 - \Pi(1, 1) \\
      \Pi(2, 1) & 1 - \Pi(2, 1)
    \end{pmatrix} \begin{pmatrix}
      \alpha(1) & \alpha(2) & 0 \\
      0 & 0 & 1
    \end{pmatrix}
    &= \begin{pmatrix}
      \Pi(1, 1) \alpha(1) &  \Pi(1, 1) \alpha(2) & 1 - \Pi(1, 1) \\
      \Pi(2, 1) \alpha(1) &  \Pi(2, 1) \alpha(2) & 1 - \Pi(2, 1)
    \end{pmatrix} \\
    &\overset{!}{=} \begin{pmatrix}
      \frac{1}{4}\alpha(2) & \frac{1}{2}\alpha(1)+\frac{1}{4}\alpha(2) & \frac{1}{2}\alpha(1)+\frac{1}{2}\alpha(2) \\
      \frac{1}{2} & \frac{1}{4} & \frac{1}{4}
    \end{pmatrix}
    = AP
  \end{align*}
  Looking at the lower right entry, we find that $\Pi(2, 1) = \frac{3}{4}$ would need to hold,
  and looking at the lower left entry, we then see that $\alpha(1) = \frac{2}{3}$ and hence $\alpha(2) = \frac{1}{3}$.
  But then, we have:
  \begin{align*}
    \Pi A = \begin{pmatrix}
      \Pi(1, 1) & 1 - \Pi(1, 1) \\
      \frac{3}{4} & \frac{1}{4}
    \end{pmatrix} \begin{pmatrix}
      \frac{2}{3} & \frac{1}{3} & 0 \\
      0 & 0 & 1
    \end{pmatrix}
    &= \begin{pmatrix}
      \Pi(1, 1) \cdot \frac{2}{3} &  \Pi(1, 1) \cdot \frac{1}{3} & 1 - \Pi(1, 1) \\
      \frac{1}{2} &  \frac{1}{4} & \frac{1}{4}
    \end{pmatrix} \\
    &\overset{!}{=} \begin{pmatrix}[1.2]
      \frac{1}{12} & \frac{5}{12} & \frac{1}{2} \\
      \frac{1}{2} & \frac{1}{4} & \frac{1}{4}
    \end{pmatrix}
    = AP
  \end{align*}
  Looking at the upper right entry, we now also see that we would need
  $\Pi(1, 1) = \frac{1}{2}$, which yields
  \begin{align*}
    \Pi A = \begin{pmatrix}[1.2]
      \frac{1}{2} & \frac{1}{2} \\
      \frac{3}{4} & \frac{1}{4}
    \end{pmatrix} \begin{pmatrix}
      \frac{2}{3} & \frac{1}{3} & 0 \\
      0 & 0 & 1
    \end{pmatrix}
    &= \begin{pmatrix}[1.2]
      \frac{1}{3} & \frac{1}{6} & \frac{1}{2} \\
      \frac{1}{2} &  \frac{1}{4} & \frac{1}{4}
    \end{pmatrix}
    \neq \begin{pmatrix}[1.2]
      \frac{1}{12} & \frac{5}{12} & \frac{1}{2} \\
      \frac{1}{2} & \frac{1}{4} & \frac{1}{4}
    \end{pmatrix}
    = AP
  \end{align*}
  Thus, it is indeed impossible in this example to choose a stochastic
  matrix $\Pi$ and a matrix $A$ with $\alpha$ distributions as rows
  leading to $\norm{\Pi A - A P}_\infty = 0$ for the weakly lumpable $\Omega = \{\{1, 2\}, \{3\}\}$.
  As $\Omega$ is also ordinarily lumpable, we further see that ordinary
  lumpability does not imply dynamic-exactness, either.
\end{proof}

\subsection{Choosing the aggregated transition matrix when using weighted state space partitioning}
\label{ssec:median}

In the current \autoref{sec:lump}, we mentioned repeatedly that we choose $\Pi = A P \Lambda$
(respectively $\Theta = A Q \Lambda$) when we aggregate a Markov chain
using weighted state space partitioning, which was justified by \autoref{cor:error0_pi_APLambda}.
We now want to take a closer look
at whether this is indeed the best choice, or if we can achieve a lower error
bound by using a different matrix for $\Pi$ or $\Theta$. In the following,
we will focus on the discrete time case, but as we only consider the problem
of finding a matrix $\Pi$ such that $\norm{\Pi A - A P}_\infty$ is minimal
(when $A$ and $P$ are fixed), the calculations below also hold for
CTMCs.

We start by introducing the median-based scheme to determine $\Pi$. This
scheme might achieve lower error bounds than setting $\Pi = A P \Lambda$ and can
be applied in the context of weighted state space partitioning.
\cite{adaptformalagg} already introduced this method for the special
case where the distributions $\alpha_\sigma$ are
uniform distributions on the aggregates.
The median-based scheme works as follows: assume
that the partition $\Omega$ and the distributions $\alpha_\sigma$ are fixed
(hence $A$ and $\Lambda$ are fixed, we now want to determine $\Pi$).
We then set $\Pi(\rho,\sigma)$ to the weighted median of the (multi-)set
\begin{align}
  \label{eq:median1}
  a_s &= \frac{1}{\alpha(s)} \sum_{r \in \rho} \alpha(r) P(r, s) \qquad (s \in \sigma \textrm{ with } \alpha(s) > 0)
\end{align}
weighted with the weights $\alpha(s)$ for every state $s \in \sigma$. The weighted median
of the (multi-)set $\{ a_s \;|\; s \in \sigma \textrm{ with } \alpha(s) > 0\}$ is the element $a_{s_{\textrm{med}}}$ with the property that
\begin{align}
  \label{eq:median2}
  \sum_{\substack{s \in \sigma: \alpha(s) > 0\\ a_s < a_{s_{\textrm{med}}}}} \alpha(s) \leq \frac{1}{2}
  \quad &\textrm{and} \quad
  \sum_{\substack{s \in \sigma: \alpha(s) > 0\\ a_s > a_{s_{\textrm{med}}}}} \alpha(s) \leq \frac{1}{2}
\end{align}
It can be that there are multiple elements in the (multi-)set $\{ a_s \;|\; s \in \sigma \textrm{ with } \alpha(s) > 0\}$
which satisfy \eqref{eq:median2}. In such a case, the weighted median is non-unique,
and any of the elements satisfying \eqref{eq:median2} may be chosen.
For a fixed partition $\Omega$ and fixed $\alpha_\sigma$ distributions,
setting $\Pi(\rho,\sigma)$ to this weighted median minimizes the error bound
$\norm{\Pi A - A P}_\infty$ as we will see in \autoref{prop:median_scheme}.

\begin{remark}
  Using the median-based scheme can result in matrices $\Pi$ and $\Theta$
  which are no longer stochastic or a generator. Therefore, the aggregated
  transient distributions $\pi_k$ and $\pi_t$ will also no longer
  necessarily be probability distributions in this case.
\end{remark}

\begin{proposition}
  \label{prop:median_scheme}
  Given a DTMC (or CTMC),
  assume that the partition $\Omega$ is fixed, and that the distributions
  $\alpha$ are fixed as well (but arbitrary). Then, setting $\Pi$ (or $\Theta$)
  according to \eqref{eq:median1} and \eqref{eq:median2} minimizes $\left(\abs{\Pi A - A P} \cdot \mathbf{1}_n\right)(\rho)$
  ($= \tau(\rho)$ in \cite{adaptformalagg}; $\left(\abs{\Theta A - A Q} \cdot \mathbf{1}_n\right)(\rho)$ for CTMCs)
  for every $\rho \in \Omega$ among all matrices $\Pi \in \bbR^{m \times m}$
  (or $\Theta \in \bbR^{m \times m}$).
  In particular, $\norm{\Pi A - A P}_\infty$ (or $\norm{\Theta A - A Q}_\infty$)
  is minimized.
\end{proposition}

\begin{proof}
  We write out the definition of $\left(\abs{\Pi A - A P} \cdot \mathbf{1}_n\right)(\rho)$
  (also compare with \eqref{eq:tau_zero}):
  \begin{align*}
    \left(\abs{\Pi A - A P} \cdot \mathbf{1}_n\right)(\rho)
    &= \sum_{s \in S} \abs{(\Pi A)(\rho, s) - (A P)(\rho, s)}
    = \sum_{\sigma \in \Omega}
      \sum_{s \in \sigma} \abs{
        \alpha(s)\Pi(\rho, \sigma)
        - \sum_{r \in \rho} \alpha(r) P(r,s)} \\
    &= \sum_{\sigma \in \Omega}
    \underbrace{\sum_{s \in \sigma} \alpha(s) \cdot \abs{
        \Pi(\rho, \sigma)
        - \frac{1}{\alpha(s)} \sum_{r \in \rho} \alpha(r) P(r,s)}}_{=: \tau(\rho,\sigma)}
  \end{align*}
  Minimizing $\left(\abs{\Pi A - A P} \cdot \mathbf{1}_n\right)(\rho)$ (with $\alpha$ and $\Omega$ fixed) amounts to minimizing
  $\tau(\rho,\sigma)$ separately for each $\sigma \in \Omega$, since
  we can choose a different value for $\Pi(\rho, \sigma)$ for every $\sigma$.
  We thus want to set $\Pi(\rho, \sigma)$ by solving the following optimization problem:
  \begin{align*}
    \Pi(\rho, \sigma) :=
    \argmin_{x \in \bbR}
    \sum_{s \in \sigma} \alpha(s) \cdot \left|
      x - \vphantom{\frac{1}{\alpha(s)} \sum_{r \in \rho}}\right.
      \underbrace{\frac{1}{\alpha(s)} \sum_{r \in \rho} \alpha(r) P(r,s)}_{a_s}
      \left.\vphantom{\frac{1}{\alpha(s)} \sum_{r \in \rho}}\right|
  \end{align*}
  Note: it is easy to see that we can assume $\alpha(s) > 0$ for all $s \in S$
  w.l.o.g.\ in this proof.
  We show that the weighted median of the (multi-)set $\{a_s \;|\; s\in\sigma\}$
  with weights $\alpha(s), \; s \in \sigma$ does indeed solve the given minimization problem.
  For ease of notation and w.l.o.g., we write $\{a_s \;|\; s\in\sigma\} = \{a_1, a_2, \ldots, a_k\}$,
  we assume $a_1 < a_2 < \ldots < a_k$ (we can merge elements $a_i$ and $a_{i + 1}$ if
  $a_i = a_{i + 1}$ into a new element whose weight is the sum of the
  two weights) and we denote the corresponding weights by $\alpha(1), \ldots, \alpha(k)$.
  Consider the function $f : \bbR \to \bbR, f(x) = \sum_{i = 1}^k \alpha(i) |x - a_i|$.
  Note that this function is piecewise linear: for $x \in (a_i, a_{i+1})$, $f$ is
  a linear combination of linear functions and thus linear. On the other hand,
  when $x = a_i$ for some $i$, the slope of $f$ might change.
  
  By this piecewise linearity and since $\lim_{x \to \pm \infty} f(x) = \infty$, we know
  that $f$ must take a global minimum for some $x$ in the set $\{a_1, \ldots, a_k\}$.
  Denote the value of $x$ for which the minimum is taken by $a_{i^\ast}$.
  Assume for a contradiction that $a_{i^\ast}$ is not the weighted median.
  W.l.o.g.~we consider the case where $\sum_{i > i^\ast} \alpha(i) > \frac{1}{2}$,
  i.e.~the case where the total weight of
  the elements which are bigger than $a_{i^\ast}$ is bigger than $\frac{1}{2}$.
  Note that we also assumed above (w.l.o.g.) that no two elements in the set
  $\{a_s \;|\; s\in\sigma\}$ are identical, i.e.~we have $a_{i^\ast + 1} > a_{i^\ast}$.

  We proceed to show that $f(a_{i^\ast + 1}) < f(a_{i^\ast})$, and
  hence the minimum is not taken at $a_{i^\ast}$ and our assumption
  that $a_{i^\ast}$ is not the weighted median must have been wrong.
  To see that this is true, consider the following:
  \begin{align*}
    f(a_{i^\ast + 1}) &= \sum_{i = 1}^k \alpha(i) |a_{i^\ast + 1} - a_i| \\
    &= \sum_{i = 1}^{i^\ast} \alpha(i) |a_{i^\ast + 1} - a_i|
    + \underbrace{\alpha(i^\ast + 1) |a_{i^\ast + 1} - a_{i^\ast + 1}|}_{0}
    + \sum_{i = i^\ast + 2}^{k} \alpha(i) |a_{i^\ast + 1} - a_i| \\
    &= \sum_{i = 1}^{i^\ast} \alpha(i) \left(|a_{i^\ast} - a_i| + |a_{i^\ast + 1} - a_{i^\ast}|\right)
    + \sum_{i = i^\ast + 2}^{k} \alpha(i) \left(|a_{i^\ast} - a_i| - |a_{i^\ast + 1} - a_{i^\ast}|\right) \\
    &= \underbrace{\sum_{i = 1}^k \alpha(i) |a_{i^\ast} - a_i|}_{f(a_{i^\ast})} + |a_{i^\ast + 1} - a_{i^\ast}|
    \underbrace{\left(\vphantom{\sum_{i = 1}^{i^\ast}}\right.
      \underbrace{\sum_{i = 1}^{i^\ast} \alpha(i)}_{< \frac{1}{2}} - \underbrace{\sum_{i = i^\ast + 1}^{k} \alpha(i)}_{> \frac{1}{2}}
    \left.\vphantom{\sum_{i = 1}^{i^\ast}}\right)}_{< 0}
  \end{align*}
  Note that the additional term (if compared to the previous line)
  in the leftmost sum on the last line is compensated by the additional
  term in the rightmost sum.

  The proof above also works for the continuous-time case.
\end{proof}

So we can achieve a lower error bound by using the median-based scheme
instead of $\Pi = A P \Lambda$, at the cost of $\Pi$ not necessarily being
stochastic anymore. Therefore, $\Pi = A P \Lambda$ could still be considered as the natural choice
for $\Pi$, because this guarantees stochasticity, and as this choice
is at least optimal when we can achieve an error bound of $0$.
As a final side remark, note that \autoref{prop:cond1} implies that all
elements in the multiset $a_s$ from equation \eqref{eq:median1} are equal
when there is some $\Pi$ such that $\norm{\Pi A - A P}_\infty = 0$.

\subsection{Comparison with more general state space reduction schemes}

Consider again the example
\begin{align*}
  P = \begin{pmatrix}[1.2]
    0 & \frac{1}{2} & \frac{1}{2} \\
    \frac{1}{4} & \frac{1}{4} & \frac{1}{2} \\
    \frac{1}{2} & \frac{1}{4} & \frac{1}{4} \\
  \end{pmatrix}
\end{align*}
which was already given in the proof of \autoref{prop:weaklump_nimpl_dynex}.
There, it was already shown that $\Omega = \{\{1,2\},\{3\}\}$ is not
a dynamic-exact aggregation when using weighted state space partitioning. With similar calculations, we can see that
$\{\{1,3\},\{2\}\}$ and $\{\{1\},\{2,3\}\}$ are no dynamic-exact aggregations, either.
Hence, it is impossible to reduce the state space of this Markov chain
to two states with a dynamic-exact aggregation under weighted state space
partitioning. However, we actually have
\begin{align*}
  A &= \begin{pmatrix}[1.2]
    \frac{7}{\sqrt{213}} & \frac{8}{\sqrt{213}} & \frac{10}{\sqrt{213}} \\
    \frac{44}{\sqrt{3621}} & -\frac{41}{\sqrt{3621}} & \frac{2}{\sqrt{3621}}
  \end{pmatrix} \approx \begin{pmatrix}
    0.48 & 0.55 & 0.69 \\
    0.73 & -0.68 & 0.03
  \end{pmatrix}, \quad
  \Pi = \begin{pmatrix}
    1 & 0 \\
    \frac{1}{4\sqrt{17}} & -\frac{1}{4}
  \end{pmatrix} \approx \begin{pmatrix}
    1 & 0 \\
    0.06 & -0.25
  \end{pmatrix} \\
  \implies \Pi A &= A P
\end{align*}
The matrices above were computed using the Schur decomposition as demonstrated
in the proof of \autoref{prop:schur}.
So, we can, in fact, find a reduction to a state space of dimension 2 which
is dynamic-exact if we allow general matrices $\Pi$ and $A$ and do not restrict
ourselves to weighted state space partitioning.
The more general view of state space reduction actually allows us to
find dynamic-exact aggregations for (almost) any desired number of aggregates,
as shown in \autoref{prop:schur}, but the above example demonstrates that
this is not true for weighted state space partitioning. We now further fix
the initial vector $p_0 = (\frac{19}{30}, 0, \frac{11}{30})^{\transp}$.
If we use the above dynamic-exact aggregation, we find that the optimal
choice for $\pi_0$ is (this can be calculated using a linear program):
\begin{align*}
  \pi_0 = \left(\frac{779}{90 \sqrt{213}}, \frac{152 \sqrt{17}}{90 \sqrt{213}}\right)^{\transp} \approx (0.59, 0.48)^{\transp}
  \textrm{ resulting in }
  \pi_0^{\transp} A = \left(\frac{19}{30}, 0, \frac{19}{45}\right)
  \textrm{ and }
  \norm{\pi_0^{\transp} A - p_0^{\transp}}_1 = \frac{1}{18} \approx 0.06
\end{align*}
Using the given dynamic-exact aggregation, we can thus, by \autoref{thm:dtmc_bound} \ref{thm:dtmc_bound_general},
calculate an approximation of the transient distribution $p_k$ at any time
$k$ with an error in the $\norm{\cdot}_1$-norm of at most $\frac{1}{18}$.
As an example,
\begin{align*}
  \widetilde{p}_{1} = \left(\frac{19}{90}, \frac{19}{45}, \frac{19}{45}\right)^{\transp}
  &\approx (0.21, 0.42, 0.42)^{\transp}, \quad
  p_1 = \left(\frac{11}{60},\frac{49}{120},\frac{49}{120}\right)^{\transp} 
  \approx (0.18, 0.41, 0.41)^{\transp} \\
  \widetilde{p}_{100} \approx \left(\frac{133}{450},\frac{76}{225},\frac{19}{45}\right)^{\transp}
  &\approx (0.30, 0.34, 0.42)^{\transp}, \quad
  p_{100} \approx \left(\frac{7}{25},\frac{8}{25},\frac{2}{5}\right)^{\transp} 
  = (0.28, 0.32, 0.4)^{\transp} = \textrm{stat.~dist.~of }P \\
  &\widetilde{p}_{n} \overset{n\to\infty}{\longrightarrow}
  \left(\frac{133}{450},\frac{76}{225},\frac{19}{45}\right)^{\transp}, \quad
  p_n \overset{n\to\infty}{\longrightarrow} \left(\frac{7}{25},\frac{8}{25},\frac{2}{5}\right)^{\transp} = \textrm{stat.~dist.~of }P
\end{align*}
Here, the error after one step is $\norm{\widetilde{p}_{1} - p_1}_1 = \frac{1}{18} \approx 0.06$.
Note that $\pi_0^{\transp} A \geq p_0^{\transp}$, and thus, by \autoref{prop:dynex_alwaysiniterr},
$\widetilde{p}_k \geq p_k$ for all $k$, and the error is $\frac{1}{18}$ for every $k$.
If we look instead at the best possible weighted state space partitioning
with $2$ aggregates, $\Pi = A P \Lambda$ and with $\pi_0$ a probability distribution,
we find that the choice minimizing the error bound after one step, i.e.~the
sum of the initial error and $\norm{\Pi A - A P}_\infty$, is the following
(which we found by brute force search with a precision of $4$ digits after the
decimal point):
\begin{align*}
  A &\approx \begin{pmatrix}
    1 & 0 & 0 \\
    0 & 0.46 & 0.54
  \end{pmatrix}, \quad
  \Pi \approx \begin{pmatrix}
    0 & 1 \\
    0.38 & 0.62
  \end{pmatrix}, \quad
  \pi_0 = \left(\frac{19}{30}, \frac{11}{30}\right)^{\transp} \\
  \implies \norm{\Pi A - A P}_\infty &\approx 0.07, \quad
  \norm{\pi_0^{\transp} A - p_0^{\transp}}_1 \approx 0.34 \\[0.5em]
  \widetilde{p}_1 &\approx (0.14, 0.40, 0.46)^{\transp}, \quad
  \widetilde{p}_{100} \approx (0.28, 0.34, 0.39)^{\transp}
\end{align*}
Here, $\norm{\widetilde{p}_1 - p_1}_1 \approx 0.10$.
Therefore, the dynamic-exact aggregation achieved a slightly lower error in the
approximation of $p_1$ when compared to the best possible weighted state
space partitioning. However, we see that, as the transient distributions
approach stationarity, the actual error of the latter becomes smaller
(around $0.03$ at step $100$),
even though the first row of $A$ in the dynamic-exact aggregation is actually
a multiple of the stationary distribution of $P$ (and it is thus possible
to have $\pi_k^{\transp}A$ equal the stationary distribution).

This example shows that it
can be useful to consider the more abstract view of aggregation in some
cases, achieving a lower approximation error. On the other hand, the advantage of weighted state space partitioning
is that we obtain an intuitive reduction to a Markov chain with fewer states
which we can interpret in a meaningful way, while the more abstract view
only results in an abstract lower-dimensional state space and the time
evolution on this reduced state space is represented by an arbitrary linear map,
and not necessarily by a Markov chain.

\subsection{Summary of results}

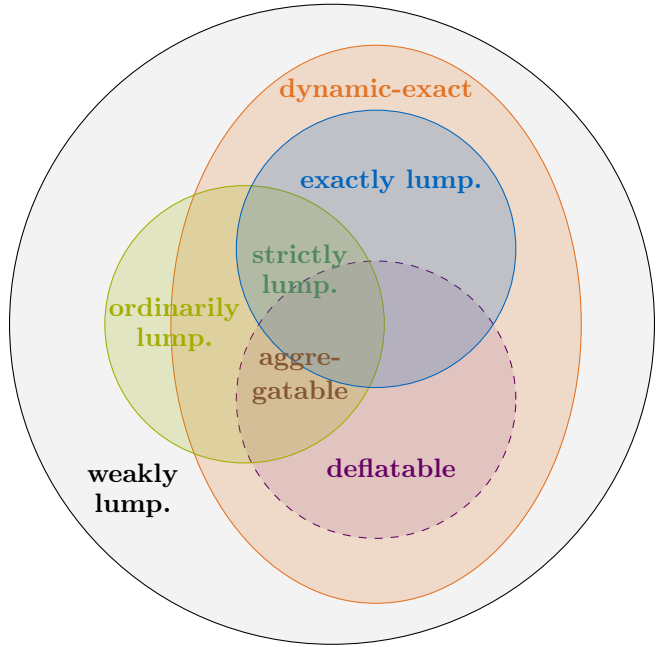
\begin{figure}[H]
  \begin{center}
    \begin{tikzpicture}
      \node[circle,inner sep=3cm,draw=black,fill=black,fill opacity=0.05] at (-0.58,-1) {};
      \draw[tumOrange,fill=tumOrange,fill opacity=0.2] (0,-1) ellipse (2.7cm and 3.7cm);
      \node[circle,inner sep=1.3cm,draw=commentpurple,dashed,fill=commentpurple,fill opacity=0.1] at (0,-2) {};
      \node[circle,inner sep=1.3cm,draw=tumGreen,fill=tumGreen,fill opacity=0.2] at (-1.73,-1) {};
      \node[circle,inner sep=1.3cm,draw=sectionblue,fill=sectionblue,fill opacity=0.2] at (0,0) {};
      \node[tumOrange] at (0,2.1) {\textbf{dynamic-exact}};
      \node[tumBlue] at (0.2,0.9) {\textbf{exactly lump.}};
      \node[commentpurple] at (0.2,-2.9) {\textbf{deflatable}};
      \node[tumGreen] at (-2.65,-0.8) {\textbf{ordinarily}};
      \node[tumGreen] at (-2.65,-1.2) {\textbf{lump.}};
      \node at (-3.2,-3) {\textbf{weakly}};
      \node at (-3.2,-3.4) {\textbf{lump.}};
      \node[sectionblue!50!tumGreen] at (-1,-0.1) {\textbf{strictly}};
      \node[sectionblue!50!tumGreen] at (-1,-0.5) {\textbf{lump.}};
      \node[tumGreen!50!commentpurple] at (-1,-1.5) {\textbf{aggre-}};
      \node[tumGreen!50!commentpurple] at (-1,-1.9) {\textbf{gatable}};
    \end{tikzpicture}
  \end{center}
  \caption{A Venn diagram summarizing the relation between the different types of lumpability}
  \label{fig:lump_types}
\end{figure}

In \autoref{fig:lump_types}, the relation between the different notions
of lumpability is shown with a Venn diagram. Deflatability is only defined
for DTMCs, which is why it is only depicted with a dashed circle. For the
other properties, the diagram is true for both the discrete- and the continuous-time
case. The whole diagram should be understood as comparing properties of
state space partitions, and we look at dynamic-exactness in the restricted
context of weighted state space partitioning here. The inclusion of
the exactly lumpable and deflatable sets in the dynamic-exact set was
proven in \autoref{prop:exlump_impl_dynex} and \autoref{prop:deflat_impl_dynex},
and the inclusion of the dynamic-exact set in the weakly lumpable set
was show in \autoref{prop:dynex_impl_weaklump}. On the other hand,
\autoref{prop:lump_nimpl_dynex} and \autoref{prop:weaklump_nimpl_dynex}
demonstrated that these inclusions are real inclusions and not equalities.
The example in the proof of \autoref{prop:weaklump_nimpl_dynex} also
showed that the ordinarily lumpable set is not included in the dynamic-exact
set. The fact that all the other depicted non-empty intersections of the various
properties exist is left for the reader to check, but very simply examples
do usually suffice (e.g.~to see that exactly lumpable and deflatable partitions
exist which are not ordinarily lumpable).

\section{Finding suitable aggregations}
\label{sec:findagg}

We will now consider how we can, in practice, find suitable aggregations
which achieve a low error bound. In contrast to the Schur
decomposition approach, which was already presented in \autoref{prop:schur},
the algorithms below will all use the weighted state space partition approach
and we will mostly stick to the notation used in the previous \autoref{sec:lump}.
The goal of these algorithms is to choose a partition $\Omega$ of the state space $S$ of
a DTMC or CTMC in such a way that the error bound $\norm{\Pi A - AP}_\infty$
(and potentially also the initial error) is small, as this will result in a good approximation
$\widetilde{p}_k$ of the transient distribution $p_k$. At the same
time, we would like $\abs{\Omega} = m \ll n = \abs{S}$ in order to
reduce the computational effort required to calculate $\widetilde{p}_k$.
An ideal algorithm would receive a parameter $\varepsilon$ as input
and determine the partition $\Omega$ with the fewest aggregates
satisfying $\norm{\pi_0^{\transp} A - p_0^{\transp}}_1 + k \cdot \norm{\Pi A - A P}_\infty < \varepsilon$,
where $k$ is the time point for which we would like to compute an approximation
of the transient distribution. This is motivated by
\autoref{thm:dtmc_bound} and \autoref{thm:ctmc_bound} and would
guarantee an error of at most $\varepsilon$.

Solving the above problem exactly will in general result in a runtime exceeding
the time needed to simply compute $p_k$ exactly for the original chain.
We will therefore consider different ways of choosing an $\Omega$ which
is somehow close to the optimal solution. The presented algorithms can
be used for general irreducible
Markov chains, without any assumption on the structure of the chain.
However, they can only be expected to perform well if a certain structure \emph{is} present:
\autoref{sec:errzero} and \autoref{sec:lump} characterize
settings in which the error bounds are low, and which can be identified
with comparatively little computational effort, in contrast to finding
a partition which satisfies exactness as defined in \autoref{def:exactagg}. The algorithms
presented in this section
will thus try to identify aggregates which are close to fulfilling
conditions such as exact lumpability or aggregatability. We will start
with some remarks on the time needed to calculate the exact transient
distribution of a Markov chain, in order to be able to assess the speed-up
resulting from aggregation.

\subsection{No aggregation}

If we perform no aggregation and compute exact transient distributions,
then we get the following runtimes:
\begin{itemize}
  \item For DTMCs, computing $p_k$ amounts to $k$ vector-matrix multiplications
  (each of the form $p_{i+1}^{\transp} = p_i^{\transp} P$) of a vector of length $n$ with a matrix
  of size $n \times n$. Each such multiplication has a runtime of order
  $\calO(n^2)$, and since we need $k$ of those, the total runtime amounts
  to $\calO(kn^2)$. Alternatively, we can compute $P^k$ and then multiply
  with $p_0$. This will need around $\calO(n^3\log(k))$ time (or $\calO(M(n)\log(k))$ where $M(n)$ is the
  time needed to multiply two square matrices of size $n$).
  \item For CTMCs, computing the exact transient distribution $p_t$ at
  time $t$ is not feasible in general. Instead, using uniformisation and
  truncation (see \cite[Section 2.5]{adaptformalagg} and \cite{foxglynn}), the transient
  distribution can be computed up to a pre-defined error $\varepsilon$
  (we will neglect the dependence of the runtime on $\varepsilon$, which
  we assume to be fixed).
  The runtime of this computation is $\calO(qtn^2)$
  where $q := \max_{s \in S} \abs{Q(s,s)}$.
\end{itemize}

\subsection{Almost aggregatability}

In \autoref{def:deflat}, aggregatability for DTMCs was defined, and we
have seen in \autoref{prop:deflat_impl_dynex} that deflatability
actually already implies that the aggregation is dynamic-exact and hence $\norm{\Pi A - A P}_\infty = 0$. The first algorithm,
taken from \cite{probalgmcagg}, will thus try to identify a partition
which is close to an aggregatable partition (or ``almost aggregatable'',
see \cite[Definition 2.8]{probalgmcagg}). As the definition of deflatability
only applies to DTMCs, this partitioning algorithm only works for DTMCs
a priori.

We quickly summarize the algorithm from \cite{probalgmcagg} here for
better understanding of some of the choices available when implementing it.
Note that in \cite{probalgmcagg}, the transpose of the
transition matrix $P$ is considered, and all transient distributions are
column vectors instead of row vectors. In the following
summary, we will stick to our notation, and hence present a transposed
version of the results in \cite{probalgmcagg}.

We first consider the case of an aggregatable matrix $P$, which means
that $\Lambda \Pi A = P$ by \cite[Proposition 2.6]{probalgmcagg},
where the partition corresponding to $\Lambda$ together with the
$\alpha$ distributions contained in $A$ is aggregatable and where $\Pi = A P \Lambda$.
We will write $\widetilde{P} := \Lambda \Pi A$ (so we now consider the
case where $\widetilde{P} = P$).
The idea is to find a connection
between the singular value decompositions (SVD for short, also see \cite[Section 2.4]{matrixcomputations}) of $\Pi$ and $P$ which allows
for identification of the aggregates by analysing the singular value
decomposition of $P$, without knowledge of $\Pi$ or $\Omega$. For an
almost aggregatable $P$, the changes in the singular value decomposition
compared to the closest aggregatable matrix will be small, and the algorithm
can thus also be applied if a matrix is only almost aggregatable. For
a detailed justification, we refer to \cite{probalgmcagg}.

In order to understand the connection between the singular value decomposition
and $\Omega$, we go back to assuming that $P$ is aggregatable and that $\Omega$ and the $\alpha$
distributions are given. Consider matrices
\begin{align*}
  D_A = \begin{pmatrix}
    \norm{\alpha_{\Omega_1}}_2 & & \\
    & \ddots & \\
    & & \norm{\alpha_{\Omega_m}}_2
  \end{pmatrix} \qquad \qquad
  D_{\Lambda} &= \begin{pmatrix}
    \sqrt{\abs{\Omega_1}} & & \\
    & \ddots & \\
    & & \sqrt{\abs{\Omega_m}}
  \end{pmatrix}
\end{align*}
Let $\widehat{U} \widehat{\Sigma} \widehat{V}$ be the singular value decomposition
of $D_\Lambda \Pi D_A$ (note: we do \emph{not} transpose $\widehat{V}$), i.e.~$\widehat{U} \widehat{\Sigma} \widehat{V} = D_\Lambda \Pi D_A$,
$\widehat{U}$ and $\widehat{V}$ are orthogonal (their rows and columns are orthonormal vectors), and $\widehat{\Sigma}$ contains
the singular values on its diagonal, ordered from largest to smallest value.
Then, the singular value decomposition $U\Sigma V$ of $P$ can be written as
\begin{align*}
  U = \left(\begin{array}{c|c} \Lambda D_{\Lambda}^{-1} \widehat{U} & U^{(2)} \end{array}\right) \in \bbR^{n \times n} \qquad
  \Sigma = \left(\begin{array}{c|c}
    \widehat{\Sigma} & 0 \\ \hline
    0 & 0
  \end{array}\right) \in \bbR^{n \times n} \qquad
  V = \left(\begin{array}{c} \widehat{V} D_A^{-1} A \\ \hline V^{(2)} \end{array}\right) \in \bbR^{n \times n}
\end{align*}
where $U^{(2)}$ and $V^{(2)}$ are such that $U$ and $V$ are orthogonal.
This is verified by computing:
\begin{align*}
  U\Sigma V = \Lambda \overbrace{D_{\Lambda}^{-1} \underbrace{\widehat{U} \widehat{\Sigma} \widehat{V}}_{D_{\Lambda} \Pi D_A} D_A^{-1}}^{\Pi} A = \widetilde{P} = P
\end{align*}
Details can be found in \cite{probalgmcagg}, where orthogonality of $U$ and $V$ is shown as well. Note that we used the matrices $D_{\Lambda}$ and $D_A$ only in order to
scale the singular values of $\Pi$ such that they agree with the singular values
of $P$, they have no other relevance. For simplicity, we will assume that
the singular value decomposition of $P$ is unique if the singular values
in $\Sigma$ are ordered by size and if we ignore $U^{(2)}$ and $V^{(2)}$ which
are anyway not used anymore. This is true, for example, when $\Pi$ has full rank (i.e.~rank $m$) and all non-zero
singular values of $P$ have multiplicity one.

We now consider the first $m$ rows
of $V$ (which are the first $m$ right-singular vectors of $P$), i.e.~$\widehat{V} D_A^{-1} A \in \bbR^{m \times n}$. We
call the $i$-th column of this submatrix
$v(i) = \left( V(1, i), \ldots, V(m, i) \right)^{\transp} \in \bbR^m$. Note that $i$ corresponds
to state $i$ of the Markov chain, and hence $\omega(i)$ is the corresponding
aggregate. Using that the upper $m$ rows of $V$ are equal to $\widehat{V} D_A^{-1} A$, we can write
\begin{align*}
  v(i) = \norm{\alpha_{\omega(i)}}_2^{-1} \alpha(i) \cdot \underbrace{\left( \widehat{V}(1, \omega(i)), \ldots, \widehat{V}(m, \omega(i)) \right)^{\transp}}_{\omega(i)\textrm{-th column of }\widehat{V}} \in \bbR^m
\end{align*}
This implies that for two states $r, s$ in the same aggregate, $v(r)$ and $v(s)$
will point in the same direction -- in fact,
\begin{align}
  \label{eq:svd_vecmult}
  \forall r, s \in S \textrm{ s.t.~} \omega(r) = \omega(s):
  \qquad \bbR^m \ni v(r) &= \underbrace{\frac{\alpha(r)}{\alpha(s)}}_{\in \bbR} \cdot \, v(s) \in \bbR^m
\end{align}
Furthermore, by orthogonality of the columns of $\widehat{V}$, $v(r)$ and $v(s)$ will
be orthogonal if $\omega(r) \neq \omega(s)$. This structure of the submatrix
can be exploited to recover $\Omega$ (and even $\alpha$). For almost aggregatable
$P$, we can only expect \eqref{eq:svd_vecmult} to hold approximately. The number
of aggregates $m$ is furthermore unknown. Instead, we choose an integer
$l$ (details on how to choose $l$ follow later) and then consider the first $l$ rows of $V$ (instead of the first $m$ rows).
This leads to the following three possible algorithms to compute $\Omega$ if
$l$ is already fixed:
\begin{itemize}
  \item \textbf{SVDsgn}: Proposed as a very simple algorithm in \cite{probalgmcagg}
  with only limited practical applicability due to its numerical instability.
  The aggregates are recovered by putting two states $r$ and $s$ into the
  same aggregate if the sign structure of the vectors $v(r)$ and $v(s)$
  is identical. By \eqref{eq:svd_vecmult} and by orthogonality of the
  vectors $v$ for states in different aggregates, this yields the correct
  partition $\Omega$ if $P$ is aggregatable. However, for almost aggregatable $P$,
  perturbed values in the vectors $v$ can lead to the sign of an entry changing,
  resulting in instability.
  \item \textbf{SVDseba}: Proposed as a more stable algorithm in \cite{probalgmcagg}
  via a combination with \cite{seba}. The \textbf{s}parse \textbf{e}igen\textbf{b}asis \textbf{a}pproximation
  algorithm proposed in \cite{seba} is applied to the first $l$ rows of $V$.
  This results in an approximate sparse basis of the space spanned by the
  first $l$ rows of $V$. In the case of aggregatability of $P$, the space
  spanned by the first $m$ rows of $V$ is spanned by the vectors $\alpha_{\Omega_1}, \ldots, \alpha_{\Omega_m}$
  as a consequence of \eqref{eq:svd_vecmult}. Therefore, the sparse
  basis obtained from applying the algorithm of \cite{seba} should approximately
  correspond to the $\alpha$ distributions for almost aggregatable $P$.
  \item \textbf{SVDdir}: A new proposal presented in this paper, with
  the intention to fully exploit \eqref{eq:svd_vecmult}. States are clustered such that
  the distance between the corresponding $v$ vectors within a cluster should be low, where the
  distance is measured as follows: the shorter vector is projected onto
  the longer vector, and we then measure the euclidean distance between the shorter original
  and the projected vector (see \autoref{alg:vdist} below). The intention behind this is to measure
  whether two vectors $v(r)$ and $v(s)$ point in approximately the same direction,
  as should be the case for $r$ and $s$ in the same aggregate by \eqref{eq:svd_vecmult}.
  The distance between vectors is not measured as the angle between vectors
  because this approach would suffer from the same numerical instability as
  SVDsgn: if the entries of a $v$ vector are close to $0$, small perturbations
  can lead to huge changes in the angle of the vector. An additional ordering
  of the $v$ vectors by length is applied to increase stability. We give
  the implementation details in \autoref{sec:svddir}.
\end{itemize}
With a given $l$ corresponding to the actual number of aggregates, the algorithms above can recover $\Omega$ from the
singular value decomposition of an almost aggregatable matrix $P$.
However, the number of aggregates is in general not known in advance.
In order to choose $l$, it makes sense to analyse the spectrum of singular
values of $P$. One way would be to identify a gap in this spectrum and
set $l$ to the number of singular values above the gap. We chose a different
approach which delivered better results in experiments. Call the singular
values $\gamma_1 \geq \ldots \geq \gamma_n \geq 0$. Given a threshold parameter $\varepsilon$, $l$ is set to the smallest value such that
\begin{align}
  \label{eq:svd_cutoff}
  \sum_{i=1}^l \gamma_i &\geq (1 - \varepsilon) \sum_{i=1}^n \gamma_i
\end{align}
i.e.~the first $l$ singular values sum to at least $1 - \varepsilon$
times the sum of all singular values. If the computation of all singular
values is too expensive, an alternative is to simply choose a fixed
number $l$ (or try multiple different $l$) which is small enough such that
the computation of the first $l$ singular values and corresponding right-singular
vectors is still affordable (there are algorithms which allow one to compute
only a part of the singular value decomposition, often referred to as
truncated SVD; compare with \cite[Section 10.4]{matrixcomputations}).

After having calculated $\Omega$ using one of the above SVD algorithms,
it is also necessary to specify the corresponding $\alpha$ distributions.
This can be done by comparing the length of the vectors $v(s)$
for all states $s$ in an aggregate. However, we saw better results in
our experiments with the following simple approach for setting $\alpha$
once the aggregation is fixed:
\begin{align}
  \label{eq:alphaprop}
  \alpha(s) &= \frac{\sum_{r \in S} P(r, s)}{\sum_{r \in S} \sum_{s' \in \omega(s)} P(r, s')}
\end{align}
$\alpha(s)$ is the same as the probability of being
in state $s$, conditioned on being in the aggregate of $s$, after the
Markov chain took a single step, starting with a uniform distribution.
Intuitively, the distributions $\alpha_\sigma$
should be approximations of this type of conditional probabilities, with
the exception that we do not necessarily start with a uniform distribution.
Whenever $\Omega$ is fixed
and $\alpha$ is chosen according to \eqref{eq:alphaprop}, we will refer
to these $\alpha$ distributions as proportional $\alpha$.
Note: proportional $\alpha$ is well-defined if the chain is irreducible.
It is easy to see that for a deflatable partition $\Omega$, choosing $\alpha$ as
in \eqref{eq:alphaprop} will result in the corresponding $\alpha$ distributions
which make the partition $\Omega$ deflatable.

The runtime of the SVD algorithms depends on the variant chosen. For
the simple SVDsgn, the computational cost is dominated by the singular
value decomposition, which needs $\calO(n^3)$ time (see \cite[Figure 8.6.1]{matrixcomputations}).
We saw that computing $p_k$ without aggregation results in a complexity
of $\calO(kn^2)$. Asymptotically (for large $k$ and $n$), applying the SVD algorithm therefore only makes sense if
$k \gg n$. As noted in \cite{probalgmcagg}, we can reduce the complexity
of the SVD approach by random sampling of the entries of $P$ under some
assumptions, which makes the algorithm more attractive. See \cite{probalgmcagg}
for details. In our experiments, we also used a singular value decomposition
algorithm for sparse matrices which only computes a partial singular value decomposition,
offered by the SciPy python package\footnote{see the SciPy documentation at \url{https://docs.scipy.org/doc/scipy/reference/generated/scipy.sparse.linalg.svds.html}}.

As mentioned before, the three variants of the SVD algorithm can only
be applied to DTMCs. They are expected to perform well for almost aggregatable
$P$, and this is confirmed in the experiments. However, in different settings,
identifying almost exactly lumpable partitions may provide better results, with the additional benefit
that this is also possible for CTMCs. This is discussed in \autoref{sec:almost_exlump}.

\subsubsection{Implementation of SVDdir}
\label{sec:svddir}

We briefly summarize the implementation details of SVDdir for completeness. First,
we define the function which computes a distance which indicates whether two vectors point in approximately the same
direction, which should also work if vectors with entries close to zero
are perturbed randomly.

\begin{algorithm}[H]
  {\raggedright
  \hspace*{\algorithmicindent} \textbf{Input:} $\textrm{vectors } v_1, v_2 \in \bbR^l \textrm{ s.t.~} \norm{v_1}_2 \geq \norm{v_2}_2$ \\
  \hspace*{\algorithmicindent} \textbf{Output:} $\textrm{a real number } \geq 0 \textrm{, which roughly measures a distance between vector angles}$ \\~}
	\begin{algorithmic}[1]
    \State $v_{\textrm{proj}} \gets \textrm{projection of } v_2 \textrm{ onto } v_1$
    \State \Return $\norm{v_{\textrm{proj}} - v_2}_2$
	\end{algorithmic}
	\caption{vdist (measuring whether vectors point in a similar direction)}
	\label{alg:vdist}
\end{algorithm}

We can now look at the implementation details of SVDdir, given in
\autoref{alg:svddir}. After calculating the singular value decomposition of $P$
and after determining the number of rows of $V$ to consider with the
aid of \eqref{eq:svd_cutoff}, the algorithm iterates over all states $s \in S$
in descending order of the length of $v(s)$ (in line 13 of \autoref{alg:svddir}).
This ordering is used to
increase stability: the larger the entries of a vector, the less its
angle varies under small perturbations of the vector components.
Hence, the first assignments of states to aggregates are based on longer
vectors whose angles are most reliable.

The assignment to aggregates proceeds as follows: when processing state $s$,
the minimum distance $d_{\textrm{min}}$ of $v(s)$ (according to vdist, see \autoref{alg:vdist})
to the vectors $v(s_1), \ldots, v(s_i)$ is calculated, where $s_1, \ldots, s_i$
are all states which have already been processed. If this distance
is smaller than the parameter $\delta$ (i.e.~if $d_{\textrm{min}} < \delta$),
then $s$ is assigned to the aggregate of $s_{i^\ast}$ where the distance
$\textrm{vdist}(v(s), v(s_{i^\ast}))$ is minimal. In another effort
to increase numerical stability, the distance $d_{\textrm{min}}$ is first
only calculated for the states in $S_{\textrm{reliable}} = \{s' \in \{s_1, \ldots, s_i\} : \norm{v(s')}_2 > 2\delta\}$,
and an assignment to an aggregate of one of the states in $S_{\textrm{reliable}}$
is performed if $d_{\textrm{min}} < \delta$ on this set, and only if this
is not possible, we calculate distances to all previously processed states.
This additional step follows the same reasoning as before: the angle of
short perturbed vectors is uncertain, so the aggregate assignment is
more certain if we assign $s$ to an aggregate which contains a state
with a longer vector pointing in a similar direction as $v(s)$.
If no previously processed state $s'$ satisfies $\textrm{vdist}(v(s), v(s')) < \delta$,
then $s$ is assigned to a new aggregate.

As input for \autoref{alg:svddir}, $\varepsilon$ should be chosen according
to how close the resulting aggregation should be to an aggregatable partition.
For $\varepsilon = 0$ (and $\delta = 0$), the algorithm finds aggregatable
partitions. For $\varepsilon = 1$, all states are assigned to a single aggregate.
We used $\delta = 0.05$ in all our experiments. Using a constant $\delta$ regardless of the size
of the state space makes sense insofar as that the columns of $V$ are vectors
of unit length. However, since we crop the columns to dimension $l$ (which
results in shorter vectors in general), a $\delta$ depending on $l$ might also be
a good choice. We experimented with different $\delta$ values, and the results
do not seem to be very sensitive to which $\delta$ is chosen (aggregations
of similar quality are found for different $\delta$ values when adapting
$\varepsilon$ appropriately).
An in-depth analysis of the numerical stability of \autoref{alg:svddir} is
still missing, but we observed at least a convincing performance in our
experiments.

\begin{algorithm}[H]
  {\raggedright
  \hspace*{\algorithmicindent} \textbf{Input:} $\textrm{a Markov chain, defined via its transition matrix } P \textrm{ on state space }S = \{1, \ldots, n\}\textrm{,}$ \\
  \hspace*{\algorithmicindent} \hphantom{\textbf{Input:}} $\textrm{the parameter }\varepsilon\textrm{, and the parameter } \delta$ \\
  \hspace*{\algorithmicindent} \textbf{Output:} $\textrm{an aggregation function } \omega$ \\
  \hspace*{\algorithmicindent} \hphantom{\textbf{Output:}} $\textrm{whose corresponding partition is close to an aggregatable partition}$\\~}
	\begin{algorithmic}[1]
    \State $\omega \gets \left((s \in S) \mapsto 1\right)$ \Comment{aggregation function}
    \State $U, \Sigma, V \gets \textrm{singular value decomposition of }P$ \Comment{\parbox[t]{6cm}{s.t.~$P = U\Sigma V$, \emph{not} $P = U\Sigma V^{\mathsf{T}}$, with ordered singular values in $\Sigma$}}
    \State $\gamma_1, \ldots, \gamma_n \gets \textrm{values on diagonal of }\Sigma$ \Comment{$\gamma_1 \geq \ldots \geq \gamma_n \geq 0$ are the singular values of $P$}
    \State $l \gets 0$ \Comment{number of considered rows of $V$}
    \While{$\sum_{i=1}^l \gamma_i < (1 - \varepsilon) \sum_{i=1}^n \gamma_i$} \Comment{determine $m$ according to \eqref{eq:svd_cutoff}}
      \State $l \gets l + 1$
    \EndWhile
    \ForAll{$s \in \{1, \ldots, n\}$}
      \State $v(s) \gets (V(1, s), \ldots, V(l, s))^{\transp} \in \bbR^m$ \Comment{extract cropped columns from $V$}
    \EndFor
    \State $m \gets 0$ \Comment{current number of aggregates}
    \State $(s_1, \ldots, s_n) \;\gets\;$\parbox[t]{4.5cm}{permutation of $(1, \ldots, n)$ s.t.\ $\norm{v(s_1)}_2 \geq \ldots \geq \norm{v(s_n)}_2$}
    \ForAll{$i \in \{1, \ldots, n\} \textrm{ (in order)}$} \Comment{go through $v(s_i)$ by descending length}
      \State $S_{\textrm{reliable}} \gets \{ s \in \{s_1, \ldots, s_{i-1}\} : \norm{v(s)}_2 > 2\delta \}$
      \State $S_{\lnot\textrm{reliable}} \gets \{s_1, \ldots, s_{i-1}\} \setminus S_{\textrm{reliable}}$
      \State $d_{\textrm{min}} \gets \min_{s \in S_{\textrm{reliable}}} \textrm{vdist}(v(s_i), v(s))$ \Comment{see \hyperref[alg:vdist]{Algo.~\ref{alg:vdist}}, $d_{\textrm{min}} = \infty$ for $S_{\textrm{reliable}} = \varnothing$}
      \If{$d_{\textrm{min}} < \delta$}
        \State $\omega(s_i) \gets \omega\left(\argmin_{s \in S_{\textrm{reliable}}} \textrm{vdist}(v(s_i), v(s))\right)$ \Comment{\parbox[t]{4.5cm}{$s_i$ added to aggregate which contains closest vector}}
      \Else
        \State $d_{\textrm{min}} \gets \min_{s \in S_{\lnot\textrm{reliable}}} \textrm{vdist}(v(s_i), v(s))$ \Comment{try again with less reliable vectors}
        \If{$d_{\textrm{min}} < \delta$}
          \State $\omega(s_i) \gets \omega\left(\argmin_{s \in S_{\lnot\textrm{reliable}}} \textrm{vdist}(v(s_i), v(s))\right)$
        \Else
          \State $m \gets m + 1$
          \State $\omega(s_i) \gets m$ \Comment{$s_i$ added to new aggregate}
        \EndIf
      \EndIf
    \EndFor
    \State \Return $\omega$
	\end{algorithmic}
	\caption{Calculating almost aggregatable partitions with SVDdir}
	\label{alg:svddir}
\end{algorithm}

We conclude this section with a short runtime analysis of \autoref{alg:svddir}.
Line 2 (the singular value decomposition) takes $\calO(n^3)$ time, as stated
before (see \cite[Figure 8.6.1]{matrixcomputations}). Lines 1 and 3 to 11 are
negligible in comparison. Line 12 (sorting the vectors by descending length)
takes time $\calO(n \log(n))$ after having computed all vector lengths
in time $\calO(nl)$; both runtimes are smaller than $\calO(n^3)$.
In line 13, we loop over all $n$ states. Determining $S_{\textrm{reliable}}$ (lines 14 and 15)
can be done in time $\calO(1)$ by successively adding each processed state
with suitable length to $S_{\textrm{reliable}}$. Line 16 takes $\calO(nl)$ time
at most (we need to call vdist at most $n$ times, and vdist is applied
to vectors of dimension $l$). The same holds for line 20. The remaining part
of the loop body is negligible. We arrive at $\calO(nl)$ for the loop body,
giving a total runtime of the loop in lines 13 to 28 of $\calO(n^2l)$ which
is also smaller than $\calO(n^3)$. The overall runtime is thus still dominated
by the singular value decomposition with $\calO(n^3)$.

For larger state spaces, we only compute partial singular value decompositions
to circumvent the $\calO(n^3)$ runtime, resulting in the state clustering
time (the loop from lines 13 to 28) being the dominant factor with a runtime
of $\calO(n^2l)$.

\subsection{Almost exact lumpability}
\label{sec:almost_exlump}

By \autoref{prop:exlump_impl_dynex}, if $\Omega$ is exactly lumpable,
if the $\alpha$ distributions are uniform distributions, and if
$\Pi = A P \Lambda$ (or $\Theta = A Q \Lambda$), then the aggregation
is dynamic-exact. However,
we cannot expect an exactly lumpable partition to exist for a general Markov chain. It might be,
though, that a partition exists which is close to being exactly lumpable.
This motivates the following definition:

\begin{definition}
  \label{def:almost_exlump}
  We call a partition $\Omega$ $\varepsilon$-\textbf{almost exactly lumpable}
  if:
  \begin{align*}
    \forall s, s' \in S \textrm{ s.t.~} \omega(s) &= \omega(s'):
    \qquad
    \sum_{\rho \in \Omega} \abs{\sum_{r \in \rho} P(r, s) - \sum_{r \in \rho} P(r, s')} \leq \varepsilon \qedhere
  \end{align*}
\end{definition}

The incoming probabilities to two states in the same aggregate from another
aggregate are not required to be identical anymore as in \autoref{def:exlump},
but they are close to being identical. For the SVD partitioning algorithm, which
also takes a (different) parameter $\varepsilon$ used for cutting off the
smallest singular values (see \eqref{eq:svd_cutoff}), it seems to be difficult to derive
a bound on the $\norm{\Pi A - A P}_\infty$ factors depending on the input parameter $\varepsilon$.
For $\varepsilon$-almost exactly lumpable partitions, we have at least the
following result:

\begin{proposition}
  \label{prop:almost_exlump_tau}
  Given a partition $\Omega$ of the state space of a DTMC or CTMC which is $\varepsilon$-almost exactly lumpable,
  let $\alpha(s) = \frac{1}{\abs{\omega(s)}}$ and $\Pi = A P \Lambda$ or $\Theta = A Q \Lambda$. Then
  \begin{align*}
    \norm{\Pi A - A P}_\infty &\leq \abs{\Omega} \cdot \frac{\max_{\rho \in \Omega} \abs{\rho}}{\min_{\rho \in \Omega} \abs{\rho}} \cdot \varepsilon \qedhere
  \end{align*}
\end{proposition}

Hence, given an algorithm which takes $\varepsilon$ as input and outputs an
$\varepsilon$-almost exactly lumpable partition, we can choose some $\varepsilon$
which guarantees a desired error bound $\norm{\Pi A - A P}_\infty$ in advance without actually
running the algorithm and calculating $\norm{\Pi A - A P}_\infty$ for the resulting partition.
In practice, the bound in \autoref{prop:almost_exlump_tau} always seemed to
be much larger than the actual value of $\norm{\Pi A - A P}_\infty$, though.

\begin{proof}[Proof of \autoref{prop:almost_exlump_tau}]
  We have
  \begin{align}
    \label{eq:almost_exlump_tau}
    \begin{split}
      \norm{\Pi A - A P}_\infty
      &\overset{\textrm{cf.~\eqref{eq:tau_zero}}}{=} \max_{\rho \in \Omega} \sum_{\sigma \in \Omega} \sum_{s \in \sigma} \abs{
        \alpha(s)\Pi(\rho, \sigma)
        - \sum_{r \in \rho} \alpha(r) P(r,s)} \\
      &\overset{\Pi = AP\Lambda}{=} \; \max_{\rho \in \Omega} \sum_{\sigma \in \Omega} \sum_{s \in \sigma} \abs{
        \frac{1}{\abs{\sigma}} \sum_{r \in \rho} \frac{1}{\abs{\rho}} \sum_{s'\in\sigma} P(r, s')
        - \sum_{r \in \rho} \frac{1}{\abs{\rho}} P(r,s)} \\
      &= \max_{\rho \in \Omega} \sum_{\sigma \in \Omega} \sum_{s \in \sigma} \abs{
        \frac{1}{\abs{\sigma}} \frac{1}{\abs{\rho}} \sum_{s'\in\sigma} \sum_{r \in \rho} P(r, s')
        -  \frac{1}{\abs{\sigma}} \frac{1}{\abs{\rho}} \sum_{s'\in\sigma} \sum_{r \in \rho} P(r,s)} \\
      &= \max_{\rho \in \Omega} \sum_{\sigma \in \Omega} \frac{1}{\abs{\sigma}} \frac{1}{\abs{\rho}} \sum_{s \in \sigma}
      \abs{\sum_{s' \in \sigma} \sum_{r \in \rho} \left(P(r, s') - P(r, s)\right)} \\
      &\leq \max_{\rho \in \Omega} \sum_{\sigma \in \Omega} \frac{1}{\abs{\sigma}} \frac{1}{\abs{\rho}} \sum_{s \in \sigma}
      \sum_{s' \in \sigma} \abs{\sum_{r \in \rho} \left(P(r, s') - P(r, s)\right)} \\
      &\leq \sum_{\sigma \in \Omega} \frac{1}{\abs{\sigma}} \sum_{s \in \sigma}
      \sum_{s' \in \sigma} \max_{\rho \in \Omega} \frac{1}{\abs{\rho}} \abs{\sum_{r \in \rho} P(r, s') - \sum_{r \in \rho} P(r, s)} \\
      &\leq \sum_{\sigma \in \Omega} \frac{1}{\abs{\sigma}} \sum_{s \in \sigma}
      \sum_{s' \in \sigma}  \frac{1}{\min_{\rho \in \Omega} \abs{\rho}} \underbrace{\max_{\rho \in \Omega} \abs{\sum_{r \in \rho} P(r, s') - \sum_{r \in \rho} P(r, s)}}_{\leq \varepsilon} \\
      &\leq \sum_{\sigma \in \Omega} \frac{\abs{\sigma}}{\min_{\rho \in \Omega} \abs{\rho}} \cdot \varepsilon
      \leq \abs{\Omega} \cdot \frac{\max_{\rho \in \Omega} \abs{\rho}}{\min_{\rho \in \Omega} \abs{\rho}} \cdot \varepsilon
    \end{split}
  \end{align}
  The same calculation holds for continuous time with $P$ replaced by $Q$ and $\Pi$ replaced by $\Theta$.
  Also note that the bound can be improved to
  \begin{align}
    \label{eq:almost_exlump_tau_im}
    \norm{\Pi A - A P}_\infty &\leq \abs{\Omega} \cdot \frac{\max_{\rho \in \Omega} \abs{\rho} - 1}{\min_{\rho \in \Omega} \abs{\rho}} \cdot \varepsilon
  \end{align}
  with the same calculation by noting that the double sum $\sum_{s \in \sigma} \sum_{s' \in \sigma}$ in \eqref{eq:almost_exlump_tau}
  actually sums over elements which are zero for $s = s'$.
\end{proof}

\begin{remark}
  The bound given in \autoref{prop:almost_exlump_tau} cannot be significantly improved.
  We can give a series of examples of $\varepsilon$-almost exactly lumpable
  partitions $\Omega$ with $\abs{\Omega} = m$, $\abs{S} = n = 2m$, $\varepsilon = \frac{1}{2m} = \frac{1}{n}$
  and $\norm{\Pi A - A P}_\infty = \abs{\Omega} \cdot \frac{\varepsilon}{2} = \frac{1}{4}$.
  Consider the state space $S = \{1, \ldots, n\}$ with $n = 2m$, the partition $\Omega = \{\{1, 2\}, \{3, 4\}, \ldots, \{2m-1, 2m\}\}$
  and $\alpha(s) = \frac{1}{2} = \frac{1}{\abs{\omega(s)}}$ for all $s$. Define
  $P(1, 2s) = \frac{2 + \varepsilon n}{2n}$ and $P(1, 2s - 1) = \frac{2 - \varepsilon n}{2n} > 0$
  for $s \in \{1, \ldots, m\}$, and $P(r, s) = \frac{1}{n}$ for $r \geq 2$ and $s \in S$. This partition is $\varepsilon$-almost
  exactly lumpable. Indeed, consider two states $s', s''$ in the same
  aggregate. Then, w.l.o.g., we have $s' = 2s$ and $s'' = 2s - 1$ for some
  $s \in \{1, \ldots, m\}$. Hence
  \begin{align*}
    \sum_{\rho \in \Omega} \abs{\sum_{r \in \rho} P(r, s') - \sum_{r \in \rho} P(r, s'')}
    &= \sum_{\rho \in \Omega} \abs{\sum_{r \in \rho} P(r, 2s) - \sum_{r \in \rho} P(r, 2s - 1)} \\
    &= \abs{\frac{2 + \varepsilon n}{2n} + \frac{1}{n} - \frac{2 - \varepsilon n}{2n} - \frac{1}{n}} + \underbrace{\sum_{\rho \in \Omega \setminus \{\{1,2\}\}} \abs{\sum_{r \in \rho} \frac{1}{n} - \sum_{r \in \rho} \frac{1}{n}}}_{0}
    = \frac{2 \varepsilon n}{2n} = \varepsilon
  \end{align*}
  Now, consider $r, s \in \{1, \ldots, m\}$. Setting $\rho = \{2r - 1, 2r\}, \sigma = \{2s - 1, 2s\}$, we have
  \begin{align*}
    \Pi(\rho, \sigma) &=
    \frac{1}{2} \cdot \left(P(2r - 1, 2s - 1) + P(2r - 1, 2s)\right)
    + \frac{1}{2} \cdot \left(P(2r, 2s - 1) + P(2r, 2s)\right)
    = \frac{2}{n}
  \end{align*}
  Hence, it holds that
  \begin{align*}
    \norm{\Pi A - A P}_\infty &= \max_{\rho \in \Omega}
    \sum_{\sigma \in \Omega} \sum_{s \in \sigma}
    \abs{\alpha(s) \Pi(\rho, \sigma) - \sum_{r \in \rho} \alpha(r) P(r, s)} \\
    &= \max_{\rho \in \Omega} \sum_{s = 1}^{m}
    \left(\abs{\alpha(2s - 1) \cdot \frac{2}{n} - \sum_{r \in \rho} \alpha(r) P(r, 2s - 1)}
    + \abs{\alpha(2s) \cdot \frac{2}{n} - \sum_{r \in \rho} \alpha(r) P(r, 2s)}\right) \\
    &= \sum_{s = 1}^{m}
    \left(\abs{\frac{1}{2} \cdot \frac{2}{n} - \sum_{r \in \{1,2\}} \frac{1}{2} P(r, 2s - 1)}
    + \abs{\frac{1}{2} \cdot \frac{2}{n} - \sum_{r \in \{1,2\}} \frac{1}{2} P(r, 2s)}\right) \\
    &= \sum_{s = 1}^{m} \left(\abs{\frac{1}{2} \cdot \frac{2}{n} - \frac{4 - \varepsilon n}{4n}}
    + \abs{\frac{1}{2} \cdot \frac{2}{n} - \frac{4 + \varepsilon n}{4n}}\right)
    = \sum_{s = 1}^{m} \frac{\varepsilon}{2} = m \cdot \frac{\varepsilon}{2} = \abs{\Omega} \cdot \frac{\varepsilon}{2}
  \end{align*}
  Therefore, we cannot drop the dependence on $\abs{\Omega}$ in the bound
  given in \autoref{prop:almost_exlump_tau}. Actually, the improved bound
  given in \eqref{eq:almost_exlump_tau_im} is tight in this case as we
  have $\min_{\rho \in \Omega} \abs{\rho} = \max_{\rho \in \Omega} \abs{\rho} = 2$.
\end{remark}

We now develop an algorithm which finds an $\varepsilon$-almost exactly lumpable
partition as a counterpart to the SVD approach for almost aggregatable partitions.
The algorithm works for both DTMCs and CTMCs (we give the DTMC version,
but for CTMCs, $P$ only has to be replaced by $Q$). For a given $\varepsilon$,
the algorithm should find a partition which is as coarse as possible and
still satisfies $\varepsilon$-almost exact lumpability. Note that in general,
there is no unique coarsest $\varepsilon$-almost exactly lumpable partition.
However, there always is a unique coarsest \emph{exactly lumpable} partition
which may be found by successive refinement of the partition $\Omega = \{S\}$.
We can thus hope to get good results by using a successive refinement
algorithm for $\varepsilon$-\emph{almost exact lumpability} as well. For
completeness and better understanding, we first show that there always is
a unique coarsest \emph{exactly lumpable} partition (which was already
shown in \cite{exactperfequiv}, where the partition refinement approach is
mentioned on p.~269). This justifies
the structure of the algorithm below for practical purposes.

\begin{lemma}
  \label{lem:exact_lump_coarse}
  Assume that the partition $\Omega$ satisfies exact lumpability.
  Let $\widetilde{\Omega}$ be a different partition of the state space
  which satisfies:
  \begin{align*}
    \forall \widetilde{\sigma} \in \widetilde{\Omega}:
    \exists \sigma_1, \ldots, \sigma_k \in \Omega : \widetilde{\sigma} = \bigcup_{i=1}^k \sigma_i
  \end{align*}
  i.e.~$\widetilde{\Omega}$ is coarser than $\Omega$ in the sense that
  the aggregated states in partition $\Omega$ are subsets of the aggregated
  states in $\widetilde{\Omega}$.

  Then, we have for DTMCs:
  \begin{equation}
    \label{eq:coarser_exlump}
    \forall s, s' \in S \textrm{ s.t.~} \omega(s) = \omega(s'):
    \forall \widetilde{\sigma} \in \widetilde{\Omega}: \qquad
    \sum_{r \in \widetilde{\sigma}} P(r, s) = \sum_{r \in \widetilde{\sigma}} P(r, s')
  \end{equation}
  The same holds for CTMCs.
\end{lemma}

\begin{remark}
  The partition $\widetilde{\Omega} = \{S\}$ is always coarser than any
  exactly lumpable partition $\Omega$, hence \autoref{lem:exact_lump_coarse}
  implies that
  \begin{align*}
    \forall s, s' \in S \textrm{ s.t.~} \omega(s) &= \omega(s'): \qquad
    \sum_{r \in S} P(r, s) = \sum_{r \in S} P(r, s')
  \end{align*}
  Also note that \autoref{lem:exact_lump_coarse} does \emph{not} imply
  that $\widetilde{\Omega}$ is exactly lumpable as well since in
  \eqref{eq:coarser_exlump}, the aggregation function $\omega$
  corresponds to the partition $\Omega$ and not to $\widetilde{\Omega}$.
\end{remark}

\begin{proof}[Proof of \autoref{lem:exact_lump_coarse}]
  This is easy to see. Let $\widetilde{\sigma} \in \widetilde{\Omega}$ be arbitrary
  and $\sigma_1, \ldots, \sigma_k \in \Omega$ s.t.~$\widetilde{\sigma} = \bigcup_{i=1}^k \sigma_i$.
  Then, for $s, s' \in S$ s.t.~$\omega(s) = \omega(s')$, we have
  \begin{align*}
    \sum_{r \in \widetilde{\sigma}} P(r, s)
    &= \sum_{i=1}^k \sum_{r \in \sigma_i} P(r, s)
    = \sum_{i=1}^k \sum_{r \in \sigma_i} P(r, s')
    = \sum_{r \in \widetilde{\sigma}} P(r, s')
  \end{align*}
  The same calculation holds for CTMCs.
\end{proof}

\begin{proposition}
  For every DTMC or CTMC, there exists a unique coarsest exactly lumpable
  partition.
\end{proposition}

\begin{proof}
  The partition where every aggregate contains exactly one state is always
  exactly lumpable, so we know that an exactly lumpable partition exists.

  Let us call $\Omega^{(1)} = \{S\}$.  We will now construct a sequence
  of partitions $\Omega^{(1)}, \ldots, \Omega^{(k)}$ such that
  $\Omega^{(i+1)}$ is finer than $\Omega^{(i)}$ as follows.
  Assume we already have constructed $\Omega^{(i)}$. We now
  construct $\Omega^{(i+1)}$. For $s, s' \in S$ with $s \neq s'$, we
  \begin{enumerate}[(i)]
    \item \label{item:coarsest_lump_step1} assign $s$ and $s'$ to different aggregates in $\Omega^{(i+1)}$
    if they already belong to different aggregates in $\Omega^{(i)}$,
    \item \label{item:coarsest_lump_step2} assign $s$ and $s'$ to different aggregates in $\Omega^{(i+1)}$
    if they belong to the same aggregate in $\Omega^{(i)}$ but there exists
    some $\sigma \in \Omega^{(i)}$ such that
    \begin{align*}
      \sum_{r \in \sigma} P(r, s) \neq \sum_{r \in \sigma} P(r, s')
    \end{align*}
    \item assign $s$ and $s'$ to the same aggregate in $\Omega^{(i+1)}$ otherwise.
  \end{enumerate}
  We stop the construction as soon as $\Omega^{(i+1)} = \Omega^{(i)} =: \Omega^{(k)}$. Note that this
  needs to happen at some point, at the latest when every aggregate in $\Omega^{(i)}$ 
  contains only one state.

  Now,
  \begin{itemize}
    \item if $\Omega^{(i+1)} = \Omega^{(i)}$, then $\Omega^{(i+1)}$ is exactly lumpable
    since step \ref{item:coarsest_lump_step2} was not applied in the last iteration
    of the construction. Hence, by definition, $\Omega^{(i+1)} = \Omega^{(i)} = \Omega^{(k)}$ is exactly lumpable.
    \item for any exactly lumpable partition $\Omega$, we have
    that $\Omega$ is finer than (or equal to) $\Omega^{(i)}$ for all $i = 1, \ldots, k$. We prove
    this by induction. $\Omega^{(1)} = \{S\}$ is coarser than any partition, so the
    statement holds for $i = 1$. The step from $i$ to $i+1$ is done as follows:
    if $\Omega^{(i)}$ is coarser than the exactly lumpable $\Omega$, then, by \autoref{lem:exact_lump_coarse},
    we have that (where the partition function $\omega$ corresponds to
    the partition $\Omega$, and not to $\Omega^{(i)}$)
    \begin{align*}
      \forall s, s' \in S \textrm{ s.t.~} \omega(s) &= \omega(s'):
      \forall \sigma \in \Omega^{(i)}: \qquad
      \sum_{r \in \sigma} P(r, s) = \sum_{r \in \sigma} P(r, s')
    \end{align*}
    Hence, when calculating $\Omega^{(i+1)}$,  step \ref{item:coarsest_lump_step2} will only assign $s$ and
    $s'$ to different aggregates in $\Omega^{(i+1)}$ if $\omega(s) \neq \omega(s')$,
    i.e.~if they also belong to different aggregates in $\Omega$.
    The same holds for step \ref{item:coarsest_lump_step1} because
    $\Omega^{(i)}$ is coarser than $\Omega$. Therefore,
    states $s$ and $s'$ belonging to the same aggregate in $\Omega$
    will also be assigned to the same aggregate in $\Omega^{(i+1)}$.
    So $\Omega^{(i+1)}$ is also coarser than $\Omega$.
  \end{itemize}
  The two statements above imply that $\Omega^{(k)}$ is an exactly lumpable
  partition which is coarser than any other exactly lumpable partition, which
  concludes the proof. The same proof can be applied for CTMCs.
\end{proof}

The following \autoref{alg:ealmostexlump} for finding $\varepsilon$-\emph{almost exactly lumpable}
partitions resulted from the successive refinement technique shown in the proof above. It does not necessarily
find a partition with as few aggregates as possible, but has performed
well in experiments.

The idea of the algorithm is as follows: we start with the initial
partition $\Omega = \{S\}$ (represented in \autoref{alg:ealmostexlump} by the aggregation function
$\omega : S \to \bbN$ which maps every state to aggregate $1$).
$\Omega$ is then successively refined. At every refinement step,
for every aggregate $\sigma \in \Omega$
and for all states $s \in \sigma$, we construct vectors of incoming probabilities
\begin{align*}
  \textrm{inc}(s) = \left(
    \sum_{r \in \Omega_1} P(r, s) \,, \;\; \ldots \,, \;\;
    \sum_{r \in \Omega_m} P(r, s)
  \right)^{\transp} \in \bbR^{m}
\end{align*}
where $m$ is the current number of aggregates in $\Omega$. For an
$\varepsilon$-almost exactly lumpable partition, it needs to hold
that the entries of the vectors $\textrm{inc}(s)$ and $\textrm{inc}(s')$
are close together for $s,s'$ in the same aggregate $\sigma$. Actually, by \autoref{def:almost_exlump}, we have
that the current partition $\Omega$ is $\varepsilon$-almost
exactly lumpable if, and only if, $\norm{\textrm{inc}(s) - \textrm{inc}(s')}_1 \leq \varepsilon$.
If this is not the case, the algorithm therefore proceeds with
the refinement by partitioning the states in $\sigma \in \Omega$
into smaller aggregates such that $\norm{\textrm{inc}(s) - \textrm{inc}(s')}_1 \leq \varepsilon$
for two states $s, s'$ in the same aggregate in the resulting refined partition.

The procedure stops when an $\varepsilon$-almost
exactly lumpable partition is found, at the latest when every
aggregate consists of a single state. The refinement step
amounts to clustering points in $\bbR^{m}$ such that the maximal
$\norm{\cdot}_1$-distance between any pair of points in a cluster
is at most $\varepsilon$. To do this, we usually apply the hierarchical
clustering utilities offered by the SciPy Python package. However, to speed up
computations, we sometimes switch to a greedy clustering algorithm
proceeding as follows: we iterate over all the vectors which we want
to cluster. Every cluster is assigned a so-called anchor vector, and when processing
a new vector, we check if it has $\norm{\cdot}_1$-distance of at most
$\frac{\varepsilon}{2}$ to any of the anchor vectors of the clusters formed by
previously processed vectors. If this is the case, we assign the vector
to the cluster of the first such anchor vector which we find. Otherwise,
a new cluster with the current vector as anchor vector is created.

\begin{algorithm}[H]
  {\raggedright
  \hspace*{\algorithmicindent} \textbf{Input:} $\textrm{a Markov chain, defined via its transition matrix } P \textrm{ on state space }S\textrm{,}$ \\
  \hspace*{\algorithmicindent} \hphantom{\textbf{Input:}} $\textrm{and the parameter }\varepsilon\textrm{ (a generator matrix }Q\textrm{ can be used instead of }P\textrm{)}$ \\
  \hspace*{\algorithmicindent} \textbf{Output:} $\textrm{an aggregation function } \omega$ \\
  \hspace*{\algorithmicindent} \hphantom{\textbf{Output:}} $\textrm{whose corresponding partition is }\varepsilon\textrm{-almost exactly lumpable}$\\~}
	\begin{algorithmic}[1]
    \State $\omega^{(1)} \gets \left((s \in S) \mapsto 1\right)$ \Comment{aggregation function}
    \State $i \gets 1$ \Comment{iteration counter}
    \State $m \gets 1$ \Comment{number of aggregates}
    \Repeat
      \State $m_{\textrm{old}} \gets m$ \Comment{saves number of old aggregates}
      \State $m \gets 0$ \Comment{counts number of new aggregates}
      \ForAll{$j \in \{1, \ldots, m_{\textrm{old}}\}$} \Comment{loop over old/target aggregates}
        \ForAll{$s \in \{r \in S: \omega^{(i)}(r) = j\}$} \Comment{loop over states in same aggregate}
          \State $\textrm{inc}(s) \gets \vec{0} \in \bbR^{m_{\textrm{old}}}$
          \ForAll{$k \in \{1, \ldots, m_{\textrm{old}}\}$} \Comment{loop over potential splitter aggregates}
            \State $\textrm{inc}(s)_k \gets \sum_{r \in S: \omega^{(i)}(r) = k} P(r, s)$ \Comment{incoming probability from agg.~$k$ to state $s$}
          \EndFor
        \EndFor
        \State $C \gets \textrm{cluster}(\{r \in S: \omega^{(i)}(r) = j\}, \textrm{inc}, \varepsilon)$ \Comment{see below}
        \ForAll{$\sigma \in C$} \Comment{loop over clusters}
          \ForAll{$s \in \sigma$}
            \State{$\omega^{(i + 1)}(s) \gets m + 1$} \Comment{states in $\sigma$ are assigned to the same aggregate}
          \EndFor
          \State $m \gets m + 1$ \Comment{increment aggregate number}
        \EndFor
      \EndFor
      \State $i \gets i + 1$
    \Until{$m_{\textrm{old}} = m$} \Comment{stop when no aggregates were split}
    \State \Return $\omega^{(i)}$
	\end{algorithmic}
	\caption{Calculating almost exactly lumpable partitions}
	\label{alg:ealmostexlump}
\end{algorithm}

The method $\textrm{cluster}(T, f, \varepsilon)$ takes a subset of states
$T \subseteq S$, a function $f : T \to \bbR^k$ and a parameter $\varepsilon > 0$
as input. The output is a partition $C$ of $T$ such that for any
cluster $\sigma \in C$ and any two states $s, s' \in \sigma$, we have that
$\norm{f(s) - f(s')}_1 \leq \varepsilon$. Of course, the method should
try to return as few clusters as possible, but our Python
implementation does not guarantee an optimal solution. The method
\texttt{scipy.cluster.hierarchy.fclusterdata} is usually used to calculate the clustering.
We will see that this method performs well in our experiments, but we also sometimes
switch to the mentioned greedy strategy.

To conclude this section, we will briefly discuss the runtime of
\autoref{alg:ealmostexlump}. Denote by $m$ the number of aggregates
returned by the algorithm (which is not known in advance). The outer
loop (lines 4 to 23) runs through at most $m$ iterations. The loops
in lines 7 to 8 lead to $n$ executions of the inner loop on lines 10 to
12. Lines 10 to 12, in turn, run in time $\calO(n)$ since the loop in
line 10 iterates over all aggregates, and line 11 then calculates a sum
over all states in the respective aggregate. Therefore, the loops in
lines 7 to 13 contribute a runtime of $\calO(n^2)$ per iteration of the
outer loop.

The runtime of line 14 depends on the clustering algorithm which is used.
In our implementation, \texttt{scipy.cluster.hierarchy.fclusterdata} runs in time $\calO(mn^2)$
because it gets at most $n$ vectors, one per state, as
input\footnote{for the runtime, see the SciPy documentation at \url{https://docs.scipy.org/doc/scipy/reference/generated/scipy.cluster.hierarchy.linkage.html}}
and because the vectors are of dimension at most $m$. Since
line 14 is executed within the loop on line 7 (which does at most
$m$ iterations), this contributes a runtime of $\calO(m^2n^2)$ per iteration of the
outer loop (larger than the $\calO(n^2)$ of lines 7 to 13). Lines 15 to 20 run in $\calO(n)$, so these are faster than
line 14 and do not add to the runtime. As the outer loop runs at most $m$ times,
we arrive at a total runtime of $\calO(m^3n^2)$.

Comparing with the runtimes for exact calculation of transient distributions
$p_k$ (respectively $p_t$),
we see that, asymptotically for large $k$ and $n$, applying \autoref{alg:ealmostexlump} makes sense if
$m^3 \ll k$ (respectively $m^3 \ll qt$ where $q$ is the maximal
exit rate of all states in the CTMC). However, the runtime bound of
$\calO(m^3n^2)$ seems to be rarely achieved in practice, so \autoref{alg:ealmostexlump}
can still be useful if we do not have $m^3 \ll k$.

If we implement the greedy clustering instead of using \texttt{scipy.cluster.hierarchy.fclusterdata},
then the clustering takes time $\calO(m^2n)$ (because only distances
to cluster anchor vectors are computed, instead of distances to all vectors),
reducing the overall runtime to $\calO(m^4n)$ (again, in practice, the algorithm
seemed to be significantly faster than that -- for example, we usually
observed that the outer loop of \autoref{alg:ealmostexlump} was executed
a number of times which was significantly lower than $m$, e.g.~by a factor of 100).

\subsection{Alternating least deviations}

In this section, we quickly analyse an approach to choose $A$, $\Pi$ and $\pi_0$ in
such a way that the bound given in \autoref{thm:dtmc_bound} \ref{thm:dtmc_bound_imprecise}
is small. A similar approach would apply to CTMCs -- however, the DTMC approach
did not perform well compared to the other methods, so we only treat it briefly. Recall that:
\begin{gather}
  \notag\textrm{If } P, \Pi \textrm{ are stochastic, } p_0, \pi_0 \textrm{ are probability distributions, and } A \textrm{ is an arbitrary matrix, then } \\
  \label{eq:dtmc_imprecise_bnd}\norm{e_k}_1 \leq \norm{\pi_0^{\transp} A - p_0^{\transp}}_1 + k \cdot \norm{\Pi A - A P}_\infty \\
  \notag\textrm{where } e_k^{\transp} = \pi_0^{\transp} \Pi^k A - p_0^{\transp} P^k \textrm{.}
\end{gather}
In this subsection, we will actually consider the more abstract view of
aggregation, with the restriction that $\Pi$ should be stochastic and that
$\pi_0$ should be a probability distribution. The matrices
$A$, $\Pi$ and $\pi_0$ define an abstract aggregation, and we do not have the
partition $\Omega$ of the state space into aggregates.

In order to obtain a good approximation of the transient distributions,
we would like to minimize the right hand side of \eqref{eq:dtmc_imprecise_bnd}.
We try to solve this optimization problem approximately via an alternating
least absolute deviations scheme, similar to the alternating least squares method
which is (e.g.) presented in \cite{approxaggregaltleastsqu}. The procedure
is as follows:
\begin{enumerate}
  \item Determine the parameter $m$ (the dimension of the aggregated state space).
  A viable approach might be to try different $m$ and see if the resulting
  aggregation have error bounds which are low enough.
  \item Initialize $A$, e.g.~randomly.
  \item \label{it:fixedAopt} Fix $A$ and choose $\Pi, \pi_0$ such that \eqref{eq:dtmc_imprecise_bnd} is
  minimal with respect to all possible choices of $\Pi, \pi_0$.
  \item \label{it:fixedPiopt} Fix $\Pi, \pi_0$ and choose $A$ such that \eqref{eq:dtmc_imprecise_bnd} is
  minimal with respect to all possible choices of $A$. Go to step \ref{it:fixedAopt}.
  \item Stop when the improvement obtained in the last iteration is smaller than
  a given threshold or when a maximal iteration count is reached.
\end{enumerate}
Steps \ref{it:fixedAopt} and \ref{it:fixedPiopt} can be solved with
linear programming (we will skip the details here). In our experiments,
even on very small Markov chains (less than $100$ states), the quality of
the result seemed to depend very much on the initial values chosen for
$A$. At the same time, the runtime is quite high because linear
programs need to be solved in every iteration. As we observed a tendency
to get stuck in local minima, the high runtime did not seem to justify using
such an approach. However, it could still be that there exists some efficient
way to initialize $A$ which we did not find and which would result in this alternating
least deviations algorithm finding a good aggregation.

We also applied a descent algorithm
which was directly trying to minimize \eqref{eq:dtmc_imprecise_bnd}. Again,
the result depended very much on the initial values. We therefore suspect
that there are too many local minima of the function given by the right
hand side of \eqref{eq:dtmc_imprecise_bnd} to use a simple function minimization
algorithm, and we did not further investigate this approach.

\subsection{Experiments}
\label{ssec:experiments}

In this section, we will compare the performance of SVDsgn, SVDseba,
SVDdir (\autoref{alg:svddir}), and \autoref{alg:ealmostexlump} on a selection
of Markov chains.
By performance comparison, we mean comparing the error bounds given
by $\norm{\Pi A - A P}_\infty$ (or $\norm{\Theta A - A Q}_\infty$) resulting from the aggregations returned by the
different algorithms -- the lower, the better. We will not compare initial
errors, as neither the SVD algorithms nor \autoref{alg:ealmostexlump}
take into account the initial distribution. The only approach presented
here which took into account initial distributions was the alternating
least deviations procedure, which we don't treat in this section due to
its infeasibility for large state spaces, and due to the resulting error
bounds which were very high. The development of efficient algorithms taking into
account both initial and dynamical error is left for future work.

We first look at the setting for which the SVD algorithm variants were
designed: almost aggregatable Markov chains. These are easy to generate randomly
and we can compare the performance of the different algorithms. Afterwards, we will
see some of the examples used in \cite{adaptformalagg} as well as
an example derived from a stochastic process algebra
model which allows for an exactly lumpable partition,
so we will also see a setting for which \autoref{alg:ealmostexlump}
was designed. By default, we will calculate the $\alpha$ distributions
as in \eqref{eq:alphaprop}, and $\Pi$ (or $\Theta$)
will be set to $\Pi = A P \Lambda$ (or $\Theta = A Q \Lambda$).

For some of the models, we used faster versions of SVDdir (\autoref{alg:svddir}) and \autoref{alg:ealmostexlump}
as already briefly mentioned before: for the fast variant of SVDdir, instead of using \eqref{eq:svd_cutoff}
to determine the number of singular values considered, we simply pass different fixed
numbers of singular values to be calculated to the algorithm which allows us
to bound the runtime (as we only need to compute a partial decomposition), and
which will result in different aggregations depending on the number passed.
The fast variant of \autoref{alg:ealmostexlump} switches to the greedy
clustering method described in the paragraph just before \autoref{alg:ealmostexlump}
when $d^2m > 4 \cdot 10^6$, where $d$ is the number of vectors (of incoming probabilities) to be clustered
and $m$ is the current number of aggregates (other choices than $4 \cdot 10^6$ are of course possible, depending
on how much one wants to bound the runtime). In line 14 of \autoref{alg:ealmostexlump},
the algorithm will thus dynamically decide which clustering method to use
depending on whether \texttt{scipy.cluster.hierarchy.fclusterdata} is expected
to have a high runtime. The figures using the fast algorithm variants
will have captions specifying that SVDdir (fast) or \autoref{alg:ealmostexlump} (fast) were used.

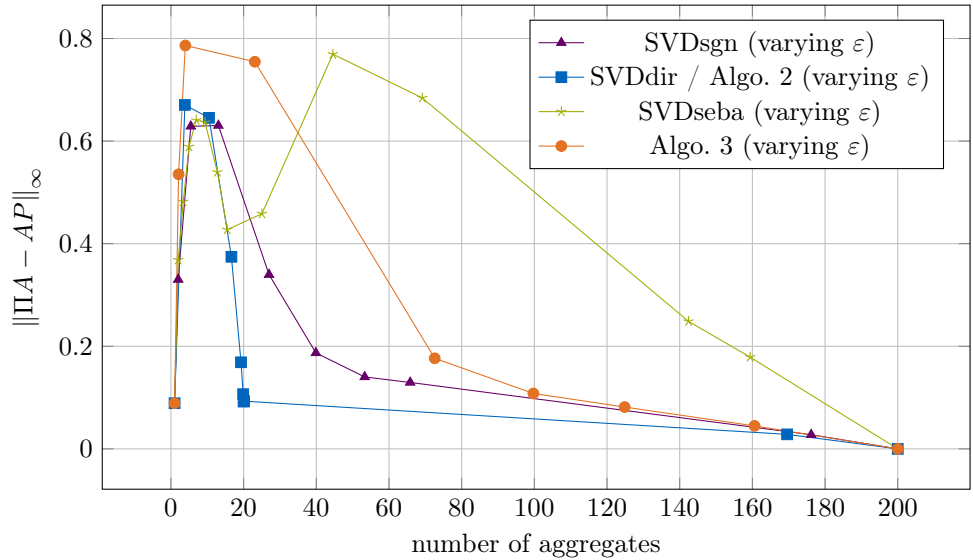
\begin{figure}[htb]
  \begin{center}
    \begin{tikzpicture}
      \begin{axis}[
        xlabel={number of aggregates},
        ylabel={$\norm{\Pi A - A P}_\infty$},
        title style={align=center, text width=10.00cm},
        height=8.00cm,
        width=13.00cm,
        grid=major,
        legend pos=north east
      ]
        \addplot[mark=triangle*,color=commentpurple] coordinates {
          (1.0000, 0.0890) (2.0000, 0.3301) (5.4600, 0.6288) (13.0600, 0.6304) (27.0000, 0.3396) (39.8900, 0.1871) (53.3200, 0.1402) (65.8400, 0.1295) (176.1800, 0.0276) (200.0000, 0.0000)
        };
        \addlegendentry{SVDsgn (varying $\varepsilon$)}
        \addplot[mark=square*,color=sectionblue] coordinates {
          (1.0000, 0.0890) (3.8300, 0.6705) (10.4700, 0.6454) (16.6400, 0.3744) (19.2800, 0.1688) (19.8800, 0.1067) (19.9900, 0.0921) (20.1300, 0.0930) (169.5500, 0.0283) (200.0000, 0.0000)
        };
        \addlegendentry{SVDdir / \hyperref[alg:svddir]{Algo.~\ref{alg:svddir}} (varying $\varepsilon$)}
        \addplot[mark=star,color=tumGreen] coordinates {
          (1.0000, 0.0890) (2.0000, 0.3683) (3.3800, 0.4818) (4.9600, 0.5891) (6.9500, 0.6407) (9.3900, 0.6366) (12.7800, 0.5394) (15.5000, 0.4273) (25.0400, 0.4584) (44.5600, 0.7695) (69.1600, 0.6842) (142.4400, 0.2485) (159.3900, 0.1789) (200, 0)
        };
        \addlegendentry{SVDseba (varying $\varepsilon$)}
        \addplot[mark=*,color=tumOrange] coordinates {
          (1.0000, 0.0890) (2.1000, 0.5353) (3.9600, 0.7864) (23.0900, 0.7547) (72.5900, 0.1766) (99.7300, 0.1080) (124.8600, 0.0815) (160.6000, 0.0448) (200.0000, 0.0000)
        };
        \addlegendentry{\hyperref[alg:ealmostexlump]{Algo.~\ref{alg:ealmostexlump}} (varying $\varepsilon$)}
      \end{axis}
    \end{tikzpicture}
  \end{center}
  \caption{SVDsgn, SVDdir, SVDseba and \autoref{alg:ealmostexlump} executed on 100 randomly
  generated almost aggregatable DTMCs with 200 states, 20 aggregates
  and a probability of 0.5 to have no transition between a particular
  pair of aggregates. The almost aggregatable DTMCs were obtained by
  random perturbation (with a magnitude of $0.002$) of the transition
  matrix of a randomly generated aggregatable DTMC.
  Each plotted point is an average resulting from running the algorithms
  with a particular fixed input parameter $\varepsilon$ on the 100 DTMCs.}
  \label{fig:almost_agg}
\end{figure}

In \autoref{fig:almost_agg}, we consider a range of randomly generated
almost aggregatable DTMCs. $\norm{\Pi A - A P}_\infty$ of the aggregation returned by the algorithms
(run with different input parameters) is plotted against the number of
aggregates which are found (which depends on the input parameter $\varepsilon$).
We can see that the SVD variants (except for SVDseba) perform
better than \autoref{alg:ealmostexlump} for these almost aggregatable chains,
which is no surprise. In addition, the improved stability of SVDdir
clearly pays off in comparison to SVDsgn: We can see a sharp drop in
the error bounds around 20 aggregates, which was the number of aggregates
in the almost aggregatable partition. For SVDsgn, the drop is more a
gradual decrease in the error bounds. \autoref{alg:ealmostexlump} does
not identify the almost aggregatable partition and only reaches a similar
error bound level for around 120 aggregates.

Note that the error bound
of around $0.09$ for a single aggregate (the leftmost point in the plots in
\autoref{fig:almost_agg}) means that the stationary distribution is very
close to the $\alpha$ distribution obtained via \eqref{eq:alphaprop}, as we have $\Pi = \begin{pmatrix}
  1
\end{pmatrix} \in \bbR^{1 \times 1}$ in this case and hence
\begin{align*}
  0.09 \approx \norm{\Pi A - A P}_\infty
  = \norm{\alpha_{\{S\}}^{\transp} - \alpha_{\{S\}}^{\transp} P}_\infty
\end{align*}
So $\alpha_{\{S\}}$ is close to being stationary. This is no big surprise
for randomly generated Markov chains as these tend to have stationary
distributions which are closer to the uniform distribution wehn compared
to Markov chains with more structure.

SVDseba performs similarly to
SVDdir for a low number of aggregates, but there is a sudden change
around 20 aggregates when SVDseba starts to perform worse than
all other algorithms. This is due to the fact that we limited the
maximum number of iterations of the SEBA algorithm (see \cite[Algorithm 3.1]{seba},
we took the MATLAB code given in \cite{seba} and translated it into Python)
to 300 iterations because of its high runtime. Regardless of the number
of maximum iterations, we could never observe SVDseba performing significantly
better than SVDdir in all our experiments. The latter is therefore a good alternative.
The better performance of SVDdir might be due to its specificity for the
given problem. The SEBA algorithm only tries to find a sparse basis for
the row space of the first $m$ rows of the matrix $V$ in the singular
value decomposition. It was designed for general applications, and does
not exploit the fact that the vectors $v(r)$ and $v(s)$ are approximate
multiples of one another for almost aggregatable DTMCs when $r$ and $s$ belong to the same aggregate.

In the following experiments, we will focus on SVDdir as the best compromise
between numerical stability and speed among the SVD approaches. We looked
at some of the models which were already used in the experimental section
of \cite{adaptformalagg} where the error bounds which were further developed
in this paper were first introduced. In \cite{adaptformalagg}, a very simple
aggregation strategy was used: the aggregation algorithm started with
singleton clusters (every state is its own cluster), then iterated over
all transitions in the model, in descending order of the transition probability,
and merged the clusters of the two connected states, unless this would
result in clusters above a user-defined maximal cluster size. This procedure
will in general result in a very high error bound $\norm{\Pi A - A P}_\infty$,
but \cite{adaptformalagg} then used a dynamic aggregate-splitting approach
while calculating transient distributions. The precalculated clusters were
stored during the whole computation, but whenever a certain cluster
had a transient probability above a given threshold at some point during
the time evolution, it was split into singletons. The singletons were merged again
into the old precalculated cluster if
the overall probability of being in that cluster would fall below the
threshold at a later point. As can be seen by looking at \autoref{thm:dtmc_bound} \ref{thm:dtmc_bound_precise},
this will also result in a low overall error bound, as the scalar product
will only become large if both the transient probability of an aggregate
and the error caused by that aggregate are large. In this paper, we
wanted to analyse whether a significant state space reduction is
already possible without dynamic aggregate-splitting if we use different
aggregation algorithms.

\begin{figure}[htb]
  \begin{center}
    \begin{tikzpicture}
      \begin{axis}[
        xlabel={number of aggregates},
        ylabel={$\norm{\Pi A - A P}_\infty$},
        title={Prokaryotic gene model (44k states)},
        title style={align=center, text width=10.00cm},
        height=8.00cm,
        width=13.00cm,
        grid=major,
        legend pos=north east
      ]
        \addplot[mark=square*,color=sectionblue] coordinates {
          (15.000000, 0.209718) (37.000000, 0.225960) (84.000000, 0.238598) (178.000000, 0.237885) (356.000000, 0.298907) (796.000000, 0.359471) (2262.000000, 0.508456) (4976.000000, 0.490453) (8122.000000, 0.494386)
        };
        \addlegendentry{SVDdir / \hyperref[alg:svddir]{Algo.~\ref{alg:svddir}} (fast)}
        \addplot[mark=square*,mark options={solid},color=tumGreen,dashed] coordinates {
          (15.000000, 0.506823) (37.000000, 0.522009) (84.000000, 0.541960) (178.000000, 0.538081) (356.000000, 0.609863) (796.000000, 0.681218) (2262.000000, 0.898290) (4976.000000, 0.836084) (8122.000000, 1.039987)
        };
        \addlegendentry{SVDdir / \hyperref[alg:svddir]{Algo.~\ref{alg:svddir}} (fast, uniform $\alpha$)}
        \addplot[mark=*,color=tumOrange] coordinates {
          (1.000000, 0.157307) (9.000000, 0.549811) (34.000000, 0.564369) (161.000000, 0.622178) (1620.000000, 0.503268) (3360.000000, 0.611461) (7432.000000, 0.447325) (12665.000000, 0.300710) (21075.000000, 0.125383) (25035.000000, 0.067493) (30476.000000, 0.005959) (31755.000000, 0.005453) (36653.000000, 0.003814) (43957.000000, 0.000000)
        };
        \addlegendentry{\hyperref[alg:ealmostexlump]{Algo.~\ref{alg:ealmostexlump}} (fast)}
        \addplot[mark=*,mark options={solid},color=linkred,dashed] coordinates {
          (1.000000, 0.400793) (9.000000, 0.550131) (34.000000, 0.561454) (161.000000, 0.622629) (1620.000000, 0.503255) (3360.000000, 0.612153) (7432.000000, 0.443399) (12665.000000, 0.300668) (21075.000000, 0.125319) (25035.000000, 0.067461) (30476.000000, 0.005959) (31755.000000, 0.005453) (36653.000000, 0.003814) (43957.000000, 0.000000)
        };
        \addlegendentry{\hyperref[alg:ealmostexlump]{Algo.~\ref{alg:ealmostexlump}} (fast, uniform $\alpha$)}
      \end{axis}
    \end{tikzpicture}
  \end{center}
  \caption{SVDdir and \autoref{alg:ealmostexlump} executed on a
  prokaryotic gene expression model already used in the experiments in
  \cite{adaptformalagg}, originally from \cite{prokgeneexpr}. The maximum population size
  was set to 5, resulting in 43~957 states. The CTMC was uniformised
  using the maximal exit rate $16.78$ as uniformisation rate.}
  \label{fig:gene_model}
\end{figure}
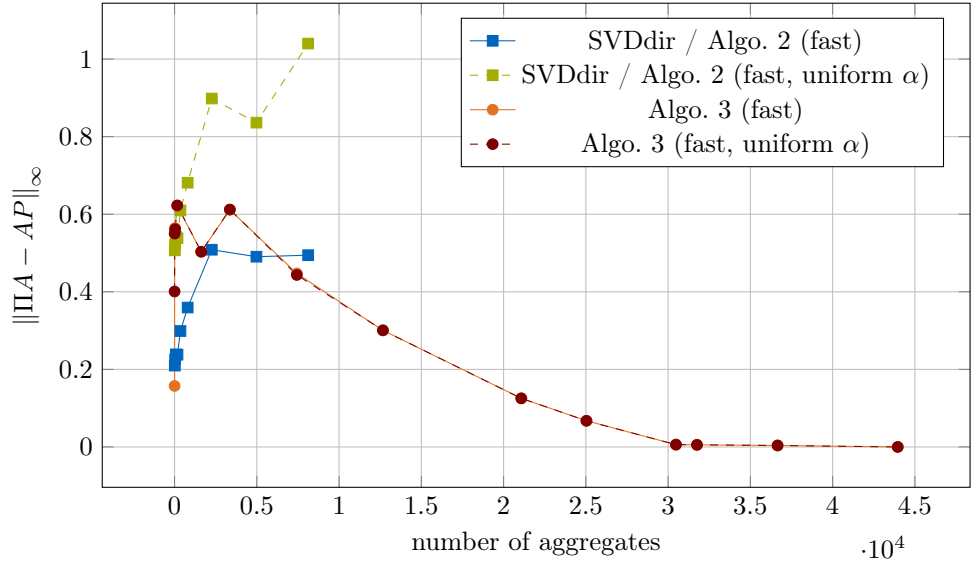

In \autoref{fig:gene_model}, we considered a prokaryotic gene expression
model \cite{prokgeneexpr} with a state space of 44k states, which is the limit of what the
fast SVDdir algorithm could handle on the machine which we used in all
our experiments (single-threaded execution on an Intel Core i7-1260P
CPU with a maximum frequency of 4.7 GHz). The graph of the SVDdir algorithm
stops at around 8k aggregates due to its runtime of $\calO(n^2l)$ (with $n$
the number of states and $l$ the number of singular values considered -- the more
values considered, the more aggregates the algorithm will usually return)
which did not allow us to compute more points. We can see that
\autoref{alg:ealmostexlump} can find state space reductions down to
around 30k states with very low error bounds (the exact values here are 30~476 aggregates with
$\norm{\Pi A - A P}_\infty \approx 0.0060$) which corresponds to a state space reduction
to around 70\%. Below that, the error bound starts to increase gradually,
resulting in aggregations were the error bounds get too high to be useful
in practice. SVDdir performed only slightly better for a low number of
aggregates, but the resulting error bounds are still too high for practical
purposes.

\autoref{fig:gene_model} also compares two ways of setting the $\alpha$
distributions: the default method of \eqref{eq:alphaprop} and setting
$\alpha(s) = \frac{1}{\abs{\omega(s)}}$, resulting in uniform $\alpha$
distributions and referred to as uniform $\alpha$ in the remaining paper.
The method makes no difference to the error bound achieved by
\autoref{alg:ealmostexlump} which can be explained by the fact that
\autoref{alg:ealmostexlump} anyway identifies aggregates where the sums
of incoming probabilities used in \eqref{eq:alphaprop} are close to being
identical (note that the orange and red dashed curve are almost exactly
on top of each other in \autoref{fig:gene_model}). However, the error bound achieved by SVDdir (\autoref{alg:svddir})
is much better if \eqref{eq:alphaprop} is used, because SVDdir, in contrast
to \autoref{alg:ealmostexlump}, can also detect aggregations where the
$\alpha$ distributions are not uniform.

\begin{figure}[htb]
  \begin{center}
    \begin{tikzpicture}
      \begin{axis}[
        xlabel={number of aggregates},
        ylabel={$\norm{\Pi A - A P}_\infty$},
        title={Lotka-Volterra model (10k states)},
        title style={align=center, text width=10.00cm},
        height=8.00cm,
        width=13.00cm,
        grid=major,
        legend pos=north east
      ]
        \addplot[mark=square*,color=sectionblue] coordinates {
          (126.000000, 1.440325) (202.000000, 0.721608) (546.000000, 1.015005) (2258.000000, 0.898579) (8968.000000, 1.697105) (10040.000000, 0.589484)
        };
        \addlegendentry{SVDdir / \hyperref[alg:svddir]{Algo.~\ref{alg:svddir}} (fast)}
        \addplot[mark=square*,mark options={solid},color=tumGreen,dashed] coordinates {
          (126.000000, 1.457281) (202.000000, 0.814348) (546.000000, 1.122734) (2258.000000, 1.013539) (8968.000000, 1.697240) (10040.000000, 0.654360)
        };
        \addlegendentry{SVDdir / \hyperref[alg:svddir]{Algo.~\ref{alg:svddir}} (fast, uniform $\alpha$)}
        \addplot[mark=*,color=tumOrange] coordinates {
          (1.000000, 0.009197) (1.000000, 0.009197) (142.000000, 0.358564) (148.000000, 0.306097) (158.000000, 0.253308) (168.000000, 0.200518) (178.000000, 0.147727) (185.000000, 0.094289) (5281.000000, 0.087622) (7903.000000, 0.038096) (9484.000000, 0.016191) (9972.000000, 0.005716) (9898.000000, 0.003571) (10201.000000, 0.000000)
        };
        \addlegendentry{\hyperref[alg:ealmostexlump]{Algo.~\ref{alg:ealmostexlump}} (fast)}
        \addplot[mark=*,mark options={solid},color=linkred,dashed] coordinates {
          (1.000000, 0.009215) (1.000000, 0.009215) (142.000000, 0.358564) (148.000000, 0.306097) (158.000000, 0.253308) (168.000000, 0.200517) (178.000000, 0.147727) (185.000000, 0.094331) (5281.000000, 0.087661) (7903.000000, 0.038113) (9484.000000, 0.016198) (9972.000000, 0.005717) (9898.000000, 0.003573) (10201.000000, 0.000000)
        };
        \addlegendentry{\hyperref[alg:ealmostexlump]{Algo.~\ref{alg:ealmostexlump}} (fast, uniform $\alpha$)}
      \end{axis}
    \end{tikzpicture}
  \end{center}
  \caption{SVDdir and \autoref{alg:ealmostexlump} executed on the
  Lotka-Volterra model already used in the experiments in
  \cite{adaptformalagg}, described in more detail e.g.~in \cite{stochsimcouplchem}. The maximum number of species
  was set to 100, resulting in 10~201 states. The CTMC was uniformised
  using the maximal exit rate $2078$ as uniformisation rate.}
  \label{fig:lotvol_model}
\end{figure}
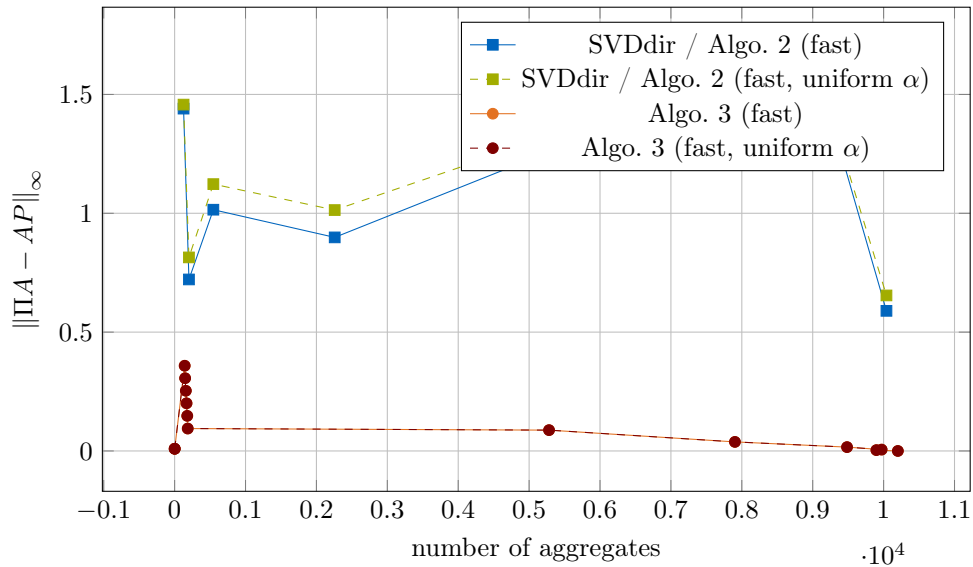

In \autoref{fig:lotvol_model}, we consider the Lotka-Volterra model (see, e.g., \cite{stochsimcouplchem})
which can be used to represent the number of predators and prey in
a simple setting. We can see that the error bounds achieved
by \autoref{alg:ealmostexlump} are comparatively low while using SVDdir
does not result in very helpful aggregations in this case. Also note that
$\norm{\Pi A - A P}_\infty \approx 0.0092$ at a single aggregate (i.e.~all states in one aggregate),
which is quite low. As discussed before, this implies that the stationary distribution is very
close to the distribution $\alpha_{\{S\}}$, and $\alpha_{\{S\}}$ is
the uniform distribution if we use uniform $\alpha$. This reduction is
only useful in practice if the initial distribution is close to the
uniform distribution. A better error bound is only achieved at
an aggregation with 9~972 aggregates and $\norm{\Pi A - A P}_\infty \approx 0.0057$
which is only a state space reduction to 98\% of the original size.

\begin{figure}[htb]
  \begin{center}
    \begin{tikzpicture}
      \begin{axis}[
        xlabel={number of aggregates},
        ylabel={$\norm{\Pi A - A P}_\infty$},
        title={Workstation cluster model (15k states)},
        title style={align=center, text width=10.00cm},
        height=8.00cm,
        width=13.00cm,
        grid=major,
        legend pos=north east
      ]
        \addplot[mark=square*,color=sectionblue] coordinates {
          (225.000000, 0.182584) (338.000000, 0.181834) (377.000000, 0.230695) (400.000000, 0.137703) (1088.000000, 0.257146) (1641.000000, 0.265994) (2895.000000, 0.313999)
        };
        \addlegendentry{SVDdir / \hyperref[alg:svddir]{Algo.~\ref{alg:svddir}} (fast)}
        \addplot[mark=square*,mark options={solid},color=tumGreen,dashed] coordinates {
          (225.000000, 0.587891) (338.000000, 0.620879) (377.000000, 0.657545) (400.000000, 0.447452) (1088.000000, 0.642780) (1641.000000, 0.656123) (2895.000000, 0.655926)
        };
        \addlegendentry{SVDdir / \hyperref[alg:svddir]{Algo.~\ref{alg:svddir}} (fast, uniform $\alpha$)}
        \addplot[mark=*,color=tumOrange] coordinates {
          (2.000000, 0.376450) (3.000000, 0.726932) (6.000000, 0.336626) (8.000000, 0.336626) (11.000000, 0.526975) (11.000000, 0.093032) (11.000000, 0.093032) (11.000000, 0.093032) (433.000000, 0.075162) (2537.000000, 0.004125) (2539.000000, 0.005468) (2778.000000, 0.004125) (4099.000000, 0.004125) (5333.000000, 0.000439) (5433.000000, 0.000240) (5523.000000, 0.000000)
        };
        \addlegendentry{\hyperref[alg:ealmostexlump]{Algo.~\ref{alg:ealmostexlump}} (fast)}
        \addplot[mark=*,mark options={solid},color=linkred,dashed] coordinates {
          (2.000000, 0.155710) (3.000000, 0.727021) (6.000000, 0.081164) (8.000000, 0.073851) (11.000000, 0.070570) (11.000000, 0.022498) (11.000000, 0.022498) (11.000000, 0.022498) (433.000000, 0.075162) (2537.000000, 0.002219) (2539.000000, 0.005461) (2778.000000, 0.002219) (4099.000000, 0.002219) (5333.000000, 0.000439) (5433.000000, 0.000240) (5523.000000, 0.000000)
        };
        \addlegendentry{\hyperref[alg:ealmostexlump]{Algo.~\ref{alg:ealmostexlump}} (fast, uniform $\alpha$)}
      \end{axis}
    \end{tikzpicture}
  \end{center}
  \caption{SVDdir and \autoref{alg:ealmostexlump} executed on a
  workstation cluster model already used in the experiments in
  \cite{adaptformalagg}, originally from \cite{modcheckdependability}. The number of workstations per cluster
  was set to 20, resulting in 15~540 states. The CTMC was uniformised
  using the maximal exit rate $50.08$ as uniformisation rate.}
  \label{fig:workcluster_model}
\end{figure}
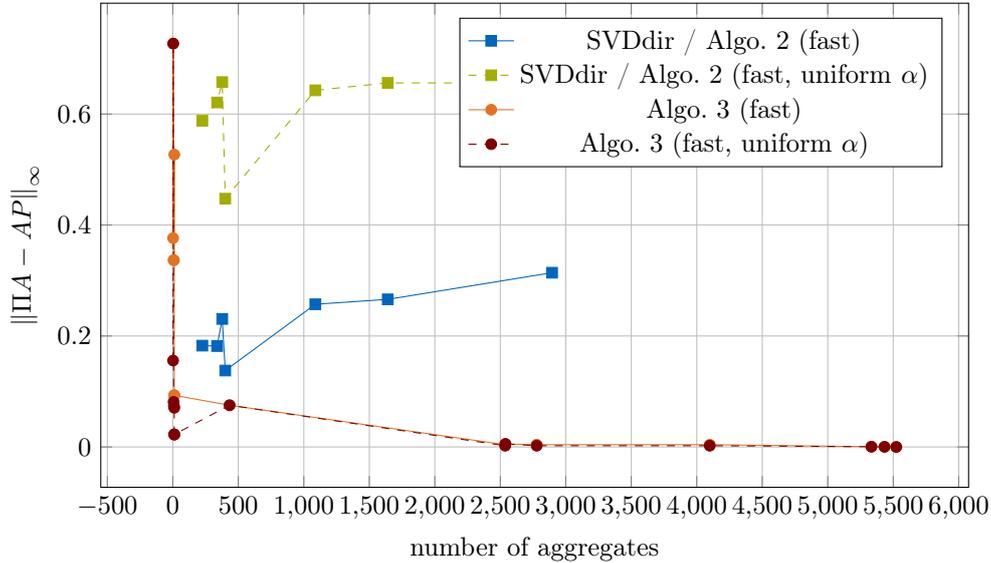

In \autoref{fig:workcluster_model}, a model of two
workstation clusters is considered where each workstation can break down
and a repair unit can repair failed components \cite{modcheckdependability}. \autoref{alg:ealmostexlump} finds
an aggregation with $\norm{\Pi A - A P}_\infty \approx 0.0022$ at 2~537 aggregates (reduction to 16\%),
and one with $\norm{\Pi A - A P}_\infty = 0$ at 5~523 aggregates (reduction to 36\%)
which could both be useful in practice. The latter is an exactly lumpable aggregation
which could already be inferred from the high-level model specification. The many points on an almost vertical line close to a single aggregate
arise from aggregations with few aggregates found by \autoref{alg:ealmostexlump} which
differ only by one or two in the number of aggregates but result in vastly different
error bounds. SVDdir (\autoref{alg:svddir}) performed
better than in \autoref{fig:lotvol_model}, but the error bounds are still quite high.

\begin{figure}[htb]
  \begin{center}
    \begin{tikzpicture}
      \begin{axis}[
        xlabel={number of aggregates},
        ylabel={$\norm{\Pi A - A P}_\infty$},
        title={RSVP model (842 states)},
        title style={align=center, text width=10.00cm},
        height=8.00cm,
        width=13.00cm,
        grid=major,
        legend pos=north east
      ]
        \addplot[mark=square*,color=sectionblue] coordinates {
          (23.000000, 0.671707) (41.000000, 0.880542) (105.000000, 0.702694) (284.000000, 0.623496) (643.000000, 0.333014) (777.000000, 0.101332) (778.000000, 0.002992)
        };
        \addlegendentry{SVDdir / \hyperref[alg:svddir]{Algo.~\ref{alg:svddir}}}
        \addplot[mark=square*,mark options={solid},color=tumGreen,dashed] coordinates {
          (23.000000, 1.527725) (41.000000, 1.228579) (105.000000, 1.158141) (284.000000, 1.217595) (643.000000, 0.844718) (777.000000, 0.740346) (778.000000, 0.666778)
        };
        \addlegendentry{SVDdir / \hyperref[alg:svddir]{Algo.~\ref{alg:svddir}} (uniform $\alpha$)}
        \addplot[mark=*,color=tumOrange] coordinates {
          (3.000000, 0.374301) (13.000000, 0.818800) (19.000000, 0.764901) (20.000000, 0.764901) (26.000000, 0.677599) (26.000000, 0.677599) (28.000000, 0.677599) (86.000000, 0.518288) (110.000000, 0.191048) (123.000000, 0.109732) (123.000000, 0.109732) (123.000000, 0.109732) (123.000000, 0.109732) (233.000000, 0.000526)
        };
        \addlegendentry{\hyperref[alg:ealmostexlump]{Algo.~\ref{alg:ealmostexlump}}}
        \addplot[mark=*,mark options={solid},color=linkred,dashed] coordinates {
          (3.000000, 0.382635) (13.000000, 0.343923) (19.000000, 0.282098) (20.000000, 0.282098) (26.000000, 0.282098) (26.000000, 0.282098) (28.000000, 0.274442) (86.000000, 0.100893) (110.000000, 0.057742) (123.000000, 0.001887) (123.000000, 0.001887) (123.000000, 0.001887) (123.000000, 0.001887) (233.000000, 0.000500)
        };
        \addlegendentry{\hyperref[alg:ealmostexlump]{Algo.~\ref{alg:ealmostexlump}} (uniform $\alpha$)}
      \end{axis}
    \end{tikzpicture}
  \end{center}
  \caption{SVDdir and \autoref{alg:ealmostexlump} executed on the DTMC arising when uniformising the RSVP model
  with $M = 7$, $N = 5$ and $3$ mobile nodes, resulting in a total of $842$
  states. The uniformisation rate was set to the maximal exit rate among all states,
  which is $30.01$ in this case.}
  \label{fig:rsvp_dtmc}
\end{figure}
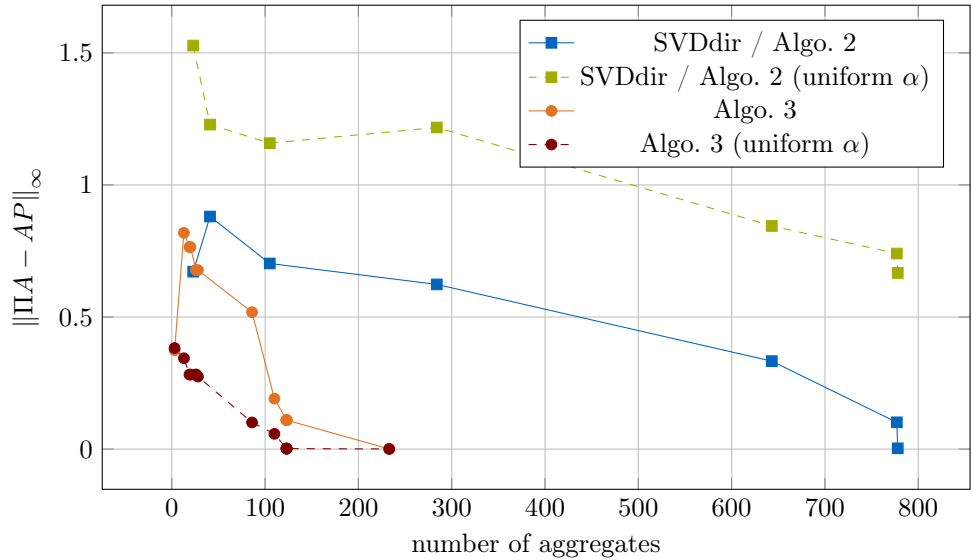 

Next to the models from \cite{adaptformalagg,prokgeneexpr,stochsimcouplchem,modcheckdependability},
we also considered a compositional stochastic process algebra model,
the RSVP model from \cite{rsvp}, at which we take a closer look. It comprises a lower network channel submodel
with capacity for $M$ calls, an upper network channel submodel with capacity for
$N$ calls, and a number of identical mobile nodes which request resources
for calls at a constant rate. Due to the mobile node symmetry in the model specification,
a lossless reduction is possible for this model. Comparing the different
algorithms in \autoref{fig:rsvp_dtmc}, we see that only \autoref{alg:ealmostexlump}
identifies the partition which results in a lossless reduction and which
is exactly lumpable: the
error bound $\norm{\Pi A - A P}_\infty$ is equal to $0$ for $234$ aggregates.
SVDdir performs much worse. \autoref{alg:ealmostexlump}
also finds a reduction to 123 aggregates (reduction to 15\%) at an error of
$\norm{\Pi A - A P}_\infty \approx 0.0019$.

\begin{figure}[htb]
  \begin{center}
    \begin{tikzpicture}
      \begin{axis}[
        xlabel={number of aggregates},
        ylabel={$\norm{\Theta A - A Q}_\infty$},
        title={RSVP model (842 states)},
        title style={align=center, text width=10.00cm},
        height=8.00cm,
        width=13.00cm,
        grid=major,
        legend pos=north east
      ]
        \addplot[mark=*,color=tumOrange] coordinates {
          (27.0000, 19.7667) (28.0000, 19.7667) (31.0000, 19.7667) (82.0000, 17.9266) (87.0000, 17.9266) (110.0000, 5.8330) (123.0000, 4.2418) (126.0000, 4.5281) (200.0000, 3.8068) (234.0000, 0.0000)
        };
        \addlegendentry{\hyperref[alg:ealmostexlump]{Algo.~\ref{alg:ealmostexlump}}}
        \addplot[mark=diamond*,color=tumDarkBlue] coordinates {
          (27.0000, 19.5271) (28.0000, 19.5271) (31.0000, 19.5271) (82.0000, 16.4346) (87.0000, 16.4346) (110.0000, 4.4900) (123.0000, 3.9709) (126.0000, 4.2203) (200.0000, 2.8572) (234.0000, 0.0000)
        };
        \addlegendentry{\hyperref[alg:ealmostexlump]{Algo.~\ref{alg:ealmostexlump}} (median)}
        \addplot[mark=*,mark options={solid},color=linkred,dashed] coordinates {
          (27.0000, 8.4658) (28.0000, 8.2360) (31.0000, 8.2360) (82.0000, 3.6167) (87.0000, 3.1362) (110.0000, 1.7328) (123.0000, 0.0566) (126.0000, 0.0382) (200.0000, 0.0240) (234.0000, 0.0000)
        };
        \addlegendentry{\hyperref[alg:ealmostexlump]{Algo.~\ref{alg:ealmostexlump}} (uniform $\alpha$)}
        \addplot[mark=diamond*,mark options={solid},color=commentpurple,dashed] coordinates {
          (27.0000, 6.0100) (28.0000, 6.0100) (31.0000, 6.0100) (82.0000, 2.0260) (87.0000, 2.0260) (110.0000, 1.6399) (123.0000, 0.0340) (126.0000, 0.0340) (200.0000, 0.0200) (234.0000, 0.0000)
        };
        \addlegendentry{\hyperref[alg:ealmostexlump]{Algo.~\ref{alg:ealmostexlump}} (uniform $\alpha$, median)}
      \end{axis}
    \end{tikzpicture}
  \end{center}
  \caption{\autoref{alg:ealmostexlump} executed on the CTMC arising from the RSVP model
  with $M = 7$, $N = 5$ and $3$ mobile nodes, resulting in a total of $842$
  states.}
  \label{fig:rsvp_ctmc}
\end{figure}
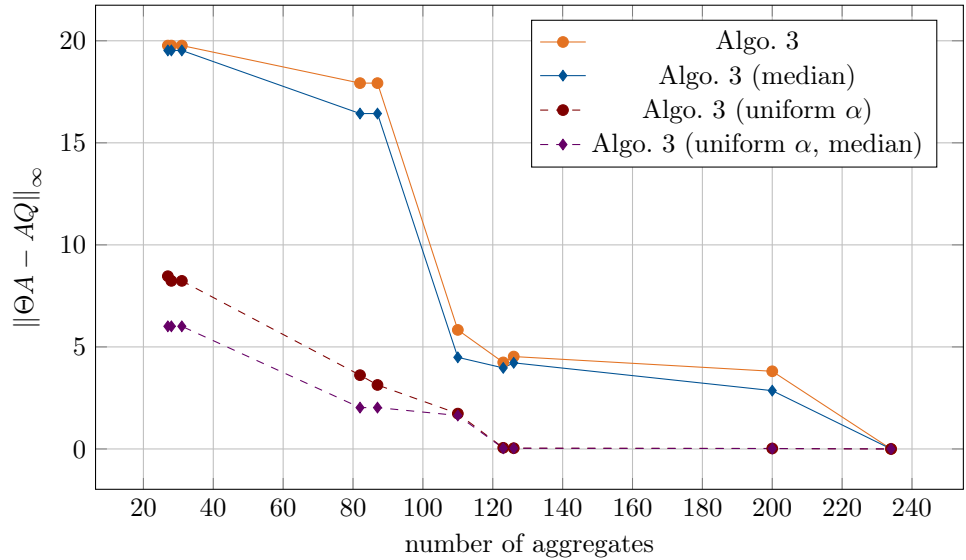

In \autoref{fig:rsvp_ctmc}, we apply \autoref{alg:ealmostexlump} directly
to the CTMC corresponding to the RSVP model and compare setting $\Theta = A Q \Lambda$ with the median-based
scheme from \autoref{ssec:median}. The error bounds are similar
in magnitude for the usual way of setting $\Theta = A Q \Lambda$
and for the median-based scheme. But we also see that the error bounds are
much higher than for the uniformised version (compare with \autoref{fig:rsvp_dtmc}). This is because $\norm{\Pi A - A P}_\infty$,
where $P$ and $\Pi$ are the uniformised versions of $Q$ and $\Theta$, corresponds to $\frac{1}{q}$ times $\norm{\Theta A - A Q}_\infty$
where $q$ is the uniformisation rate, which is
set to the maximal exit rate $30.01$ in this case.
In addition, \autoref{fig:rsvp_ctmc} demonstrates that uniform $\alpha$,
i.e.~setting $\alpha(s) = \frac{1}{\abs{\omega(s)}}$,
works better than proportional $\alpha$ as in \eqref{eq:alphaprop} for \autoref{alg:ealmostexlump}
for this model, while we saw few differences for the other models.

To get some insights into which states are put into a common aggregate
in the RSVP model, we take a closer look at the aggregates found by \autoref{alg:ealmostexlump}
in \autoref{fig:rsvp_ctmc}. If we set $\varepsilon = 0.01$, then our
implementation identifies the aggregation with $234$ aggregates which
is dynamic-exact. There are $68$ aggregates of size $6$ corresponding
to the states for which each of the three mobile nodes is in a distinct state,
$134$ aggregates of size $3$ corresponding to the states for which two mobile nodes are in the same state,
and $32$ aggregates of size $1$ corresponding to the states for which
all mobile nodes are in the same state.

For $\varepsilon=0.1$, the identified aggregation consists of $123$ aggregates
and the error bound is still relatively low with $\norm{\Theta A - A Q}_\infty \approx 0.057$
for uniform $\alpha$ and $\Theta = A Q \Lambda$ (the median-based
scheme with uniform $\alpha$ results in $\norm{\Theta A - A Q}_\infty \approx 0.034$).
For a comparison, note that the bound for $\norm{\Theta A - A Q}_\infty$ given in \eqref{eq:almost_exlump_tau_im} would yield
$123 \cdot \frac{78 - 1}{1} \cdot 0.1 = 947.1$ in this case, which is more than $10^4$ times
higher than the actual value.
The unique maximal
aggregate of size $78$ consists of a subset of the states in the original
RSVP model where one mobile node is in state $MN_0$ (the idle state),
and the other two mobile nodes are in states $MN_1$ (waiting for network resources),
$MN_3$ (waiting for network after handover) or $MN_4$ (releasing resources).
The upper network channel is in one of the states $UNC_0$, $UNC_1$ or $UNC_2$
(so there are still free resources in the upper channel), the lower network
channel is in one of the states $LNC_0$ to $LNC_6$ (everything except for
fully loaded), and the channel monitor in one of the states $CM_0$ to $CM_5$
(the number of sessions to expire after a handover has not yet reached
its maximal allowed value). \autoref{alg:ealmostexlump} has thus aggregated
states with one idle mobile node, two mobile nodes either releasing
or requesting resources, and with the network channels not yet a full
capacity.

For $\varepsilon=1.6$, we still get the same $123$ aggregates, but for
$\varepsilon=3.2$, this number is reduced to $110$, resulting in $\norm{\Theta A - A Q}_\infty \approx 1.7$.

Apart from the randomly generated
almost aggregatable chains, \autoref{alg:ealmostexlump} performed better
than SVDdir in all our experiments.
Furthermore, we saw that \autoref{alg:ealmostexlump}
can sometimes reduce the state space with an associated error bound which
is useful in practice, at size reductions in a wide range from down to e.g.~15\% to 70\% of the original size,
the size reduction depending very much on the model. As it
only takes $\calO(km^2)$ time to calculate transient distributions
of the aggregated model with $k$ the time horizon
and $m$ the number of aggregates,
compared to $\calO(kn^2)$ for the full model, such size reductions can
significantly speed up the computation of transient distributions
(e.g.~by a factor of around 44 for a reduction to 15\% and a factor
of around 2 for a reduction to 70\%). Hence, \autoref{alg:ealmostexlump}
with its runtime of $\calO(m^4n)$
for the fast variant (which is only a crude upper bound) can be a
useful tool in practice to check if a state space reduction is possible.
On the other hand, even faster algorithms and algorithms taking into
account the initial distribution are an interesting avenue for further
research.

\section{Conclusion \& outlook}

We extended the error bounds originally derived in \cite{adaptformalagg}
to a more general and abstract setting where an aggregate no longer
necessarily is a group of states, and where each state within an aggregate
can be assigned an individual weight. We also analysed the meaning
of these bounds for CTMCs. The presented error bounds are the best possible bounds in general
for the difference between the transient distribution of an
aggregated Markov chain and the original chain. Our analysis also
showed a relation of the error bounds to existing lumpability concepts.
Surprisingly, the general case for which the correct transient distributions
of a Markov chain can be derived from the aggregated model
(an exact or dynamic-exact aggregation) had only
been considered in this context in a small part of the existing literature
(see e.g.~\cite{mcaggreactnet} or \cite{bisimwgtautom}), even though
the concept of dynamic-exactness and exactness has already appeared
under various names in the previous decades. In addition, we
showed that the error bounds can also be applied to bound the distance
of an approximated stationary distribution to stationarity.

Calculating an aggregation which results in a good approximation of
the original dynamics is difficult for general Markov chains. We compared
two algorithms which identify two different settings in which the error bounds
are low. The SVD algorithm from \cite{probalgmcagg}, augmented with the clustering by
vector direction (i.e.~SVDdir) had a much lower runtime than
the SVD approach combined with SEBA from \cite{seba}, while preserving
the quality of the returned aggregations at the same time. For almost aggregatable
Markov chains, the SVD algorithm is a good choice for identifying
aggregates. However, in most of our experiments, we saw that
\autoref{alg:ealmostexlump} (based on the concepts from \cite{exactperfequiv}) performed
better than the SVD variants, which makes it a promising alternative.
We also experimented with an alternating least deviations approach
similar to the one in \cite{approxaggregaltleastsqu}, which
did not deliver competitive results.

A detailed comparison to
the adaptive reaggregation from \cite{adaptformalagg} and further
possible aggregation algorithms would be necessary to
fully evaluate accuracy and runtime of the different possibilities.
However, from a theoretical perspective, this paper already established
that the error bounds from \cite{adaptformalagg} are a good tool to
bound the aggregation error in much more general settings than originally
envisioned, and first experiments
showed that identifying almost exactly lumpable partitions with \autoref{alg:ealmostexlump} might be
a good way to find suitable aggregations.

Another interesting topic would be to develop an efficient algorithm which directly
finds an approximate solution to $\Pi A = A P$ while taking the initial
error into account at the same time. It is not clear, however, if
such an efficient algorithm exists at all.

\bibliographystyle{elsarticle-num} 
\bibliography{references}

\begin{thebibliography}{10}
\expandafter\ifx\csname url\endcsname\relax
  \def\url#1{\texttt{#1}}\fi
\expandafter\ifx\csname urlprefix\endcsname\relax\def\urlprefix{URL }\fi
\expandafter\ifx\csname href\endcsname\relax
  \def\href#1#2{#2} \def\path#1{#1}\fi

\bibitem{ncd}
H.~A. Simon, A.~Ando, Aggregation of variables in dynamic systems, Econometric 29~(2) (1961) 111--138.
\newblock \href {https://doi.org/10.2307/1909285} {\path{doi:10.2307/1909285}}.

\bibitem{finmc}
J.~G. Kemeny, J.~L. Snell, \href{https://link.springer.com/book/9780387901923}{Finite {M}arkov Chains}, Springer, 1976.
\newline\urlprefix\url{https://link.springer.com/book/9780387901923}

\bibitem{fincharweaklumpctmc}
G.~Rubino, B.~Sericola, A finite characterization of weak lumpable {M}arkov processes. {P}art {I}{I}: The continuous time case, Stochastic Processes and their Applications 45~(1) (1993) 115--125.
\newblock \href {https://doi.org/10.1016/0304-4149(93)90063-A} {\path{doi:10.1016/0304-4149(93)90063-A}}.

\bibitem{exactordlump}
P.~Buchholz, Exact and ordinary lumpability in finite {M}arkov chains, Journal of Applied Probability 31~(1) (1994) 59--75.
\newblock \href {https://doi.org/10.2307/3215235} {\path{doi:10.2307/3215235}}.

\bibitem{adaptformalagg}
A.~Abate, R.~Andriushchenko, M.~\v{C}e\v{s}ka, M.~Kwiatkowska, Adaptive formal approximations of {M}arkov chains, Performance Evaluation 148~(102207) (2021).
\newblock \href {https://doi.org/10.1016/j.peva.2021.102207} {\path{doi:10.1016/j.peva.2021.102207}}.

\bibitem{mcaggreactnet}
A.~Ganguly, T.~Petrov, H.~Koeppl, Markov chain aggregation and its applications to combinatorial reaction networks, Journal of Mathematical Biology 69~(3) (2014) 767--797.
\newblock \href {https://doi.org/10.1007/s00285-013-0738-7} {\path{doi:10.1007/s00285-013-0738-7}}.

\bibitem{probalgmcagg}
A.~Bittracher, C.~Sch\"utte, A probabilistic algorithm for aggregating vastly undersampled large {M}arkov chains, Physica D: Nonlinear Phenomena 416~(132799) (2021).
\newblock \href {https://doi.org/10.1016/j.physd.2020.132799} {\path{doi:10.1016/j.physd.2020.132799}}.

\bibitem{exactperfequiv}
P.~Buchholz, Exact performance equivalence: An equivalence relation for stochastic automata, Theoretical Computer Science 215 (1999) 263--287.
\newblock \href {https://doi.org/10.1016/S0304-3975(98)00169-8} {\path{doi:10.1016/S0304-3975(98)00169-8}}.

\bibitem{matrixcomputations}
G.~H. Golub, C.~F.~V. Loan, Matrix Computations, 4th Edition, Johns Hopkins University Press, 2013.
\newblock \href {https://doi.org/10.56021/9781421407944} {\path{doi:10.56021/9781421407944}}.

\bibitem{realanalysis}
H.~L. Royden, Real Analysis, 3rd Edition, Collier Macmillan, 1988.

\bibitem{markovbounds}
J.~Ledoux, L.~Truffet, Markovian bounds on functions of finite {M}arkov chains, Advances in Applied Probability 33~(2) (2001) 505--519.
\newblock \href {https://doi.org/10.1017/S0001867800010910} {\path{doi:10.1017/S0001867800010910}}.

\bibitem{bisimwgtautom}
P.~Buchholz, Bisimulation relations for weighted automata, Theoretical Computer Science 393 (2008) 109--123.
\newblock \href {https://doi.org/10.1016/j.tcs.2007.11.018} {\path{doi:10.1016/j.tcs.2007.11.018}}.

\bibitem{foxglynn}
B.~L. Fox, P.~W. Glynn, Computing poisson probabilities, Communications of the ACM 31~(4) (1988) 440--445.
\newblock \href {https://doi.org/10.1145/42404.42409} {\path{doi:10.1145/42404.42409}}.

\bibitem{seba}
G.~Froyland, C.~P. Rock, K.~Sakellariou, Sparse eigenbasis approximation: Multiple feature extraction across spatiotemporal scales with application to coherent set identification, Communications in Nonlinear Science and Numerical Simulation 77 (2019) 81--107.
\newblock \href {https://doi.org/10.1016/j.cnsns.2019.04.012} {\path{doi:10.1016/j.cnsns.2019.04.012}}.

\bibitem{approxaggregaltleastsqu}
P.~Buchholz, J.~Kriege, Approximate aggregation of markovian models using alternating least squares, Performance Evaluation 73 (2014) 73--90.
\newblock \href {https://doi.org/10.1016/j.peva.2013.09.001} {\path{doi:10.1016/j.peva.2013.09.001}}.

\bibitem{prokgeneexpr}
A.~M. Kierzek, J.~Zaim, P.~Zielenkiewicz, The effect of transcription and translation initiation frequencies on the stochastic fluctuations in prokaryotic gene expression*, Journal of Biological Chemistry 276~(11) (2001) 8165--8171.
\newblock \href {https://doi.org/10.1074/jbc.M006264200} {\path{doi:10.1074/jbc.M006264200}}.

\bibitem{stochsimcouplchem}
D.~T. Gillespie, Exact stochastic simulation of coupled chemical reactions, The Journal of Physical Chemistry 81~(25) (1977) 2340--2361.
\newblock \href {https://doi.org/10.1021/j100540a008} {\path{doi:10.1021/j100540a008}}.

\bibitem{modcheckdependability}
B.~R. Haverkort, H.~Hermanns, J.-P. Katoen, On the use of model checking techniques for dependability evaluation, in: Proceedings of the 19th IEEE Symposium on Reliable Distributed Systems, IEEE, 2000, pp. 228--237.
\newblock \href {https://doi.org/10.1109/RELDI.2000.885410} {\path{doi:10.1109/RELDI.2000.885410}}.

\bibitem{rsvp}
H.~Wang, D.~I. Laurenson, J.~Hillston, Evaluation of {RSVP} and mobility-aware {RSVP} using performance evaluation process algebra, 2008 IEEE International Conference on Communications (2008) 192--197\href {https://doi.org/10.1109/ICC.2008.43} {\path{doi:10.1109/ICC.2008.43}}.

\end{thebibliography}

\end{document}